\documentclass[leqno,10pt]{amsart}

\usepackage{graphicx}
\usepackage[english]{babel}
\usepackage[T1]{fontenc}
\usepackage[utf8]{inputenc}
\usepackage{epsfig}
\usepackage{fancyhdr}
\usepackage{mathrsfs}
\usepackage{amssymb}
\usepackage{amsmath}
\usepackage{amsfonts}
\usepackage{amsthm}
\usepackage{amscd}
\usepackage{color} %PACCHETTO PER SCRIVERE IL TESTO COLORATO. si usa \textcolor{colore o codice}{testo}. OKKIO: con questo pacchetto le linee sopra e sotto la pagina diventano blu SOLO SE GUARDATE DAL DVI, nel pdf sono normali (BOH...)
\usepackage[active]{srcltx}
\usepackage{hyperref}
\newcommand{\R}{\mathbb{R}}% insieme dei numeri reali
\newcommand{\loc}{\mathrm{loc}}
\newcommand{\ra}{\rightarrow}

\newcommand{\diver}{\mathrm{div}}

\newcommand{\eps}{\varepsilon}

\newcommand{\lip}{\mathrm{Lip}}
\newcommand{\disp}{\displaystyle}

\newcommand{\di}{\mathrm{d}}
\newcommand{\vol}{\mathrm{vol}}
\newcommand{\Dd}{\mathcal{D}}
\newcommand{\SR}{\mathrm{S}}
\newcommand{\GSR}{\mathrm{GS}}

\newcommand{\K}{\mathcal K}
\newcommand{\Kpt}{\mathcal K_{\psi,\theta}}
\newcommand{\A}{\mathcal{A}}
\newcommand{\B}{\mathcal{B}}
\newcommand{\F}{\mathcal{F}}

\newcommand{\wup}{W^{1,p}(\Omega)}

\newcommand{\wupz}{W^{1,p}_0(\Omega)}

\begin{document}

\newtheorem{Definition}{Definition}[section]
\newtheorem{Proposition}{Proposition}[section]
\newtheorem{Lemma}{Lemma}[section]
\newtheorem{Theorem}{Theorem}[section]
\newtheorem{Remark}{Remark}[section]
\newtheorem{Corollary}{Corollary}[section]
\newtheorem{Example}{Example}[section]
\newtheorem*{Question}{Question}
\numberwithin{equation}{section}
\newtheorem*{Notation}{Notation}
%
%
%%NUOVI COMANDI
%
\newcommand{\riem}{(M, \langle \, , \, \rangle)}
\newcommand{\Hess}{\mathrm{Hess}\, }
\newcommand{\hess}{\mathrm{hess}\, }
\newcommand{\cut}{\mathrm{cut}}
\newcommand{\ind}{\mathrm{ind}}
\newcommand{\ess}{\mathrm{ess}}
\newcommand{\longra}{\longrightarrow}
\newcommand{\metric}{\langle \, , \, \rangle}
\newcommand{\mmetric}{\langle\langle \, , \, \rangle\rangle}
\newcommand{\rad}{\mathrm{rad}}
\newcommand{\sm}{\mathrm{sm}}
\newcommand{\sn}{\mathrm{sn}}
\newcommand{\cn}{\mathrm{cn}}
\newcommand{\ink}{\mathrm{in}}
\newcommand{\capac}{\mathrm{cap}}
\newcommand{\Fk}{\mathcal{F}_k}
\newcommand{\dist}{\mathrm{dist}}
\newcommand{\grh}{\mathcal{G}^{(h)}}
\newcommand{\Ricc}{\mathrm{Ric}}
\newcommand{\LL}{\mathcal{L}}
\newcommand{\Kt}{\K_{\theta}}
\newcommand{\Aa}{\mathcal{A}}
\newcommand{\BB}{\mathcal{B}}
\newcommand{\wupst}{W^{1,p}(\Omega)^*}
\newcommand{\wupzst}{W^{1,p}_0(\Omega)^*}
\newcommand{\bKpt}{\bar{\K}_{\psi,\theta}}
\newcommand{\supp}{\operatorname{supp}}
\newcommand{\wedgedot}{\wedge\cdots\wedge}
\newcommand{\mob}{\mathrm{M\ddot{o}b}}
\newcommand{\mab}{\mathfrak{m\ddot{o}b}}
\newcommand{\Cf}{\mathcal{C}_f}
\newcommand{\cutf}{\mathrm{cut}_{f}}
\newcommand{\Cn}{\mathcal{C}_n}
\newcommand{\cutn}{\mathrm{cut}_{n}}
\newcommand{\Ca}{\mathcal{C}_a}
\newcommand{\cuta}{\mathrm{cut}_{a}}
\newcommand{\cutc}{\mathrm{cut}_c}
\newcommand{\cutcf}{\mathrm{cut}_{cf}}
\newcommand{\rk}{\mathrm{rk}}
\newcommand{\crit}{\mathrm{crit}}
\newcommand{\diam}{\mathrm{diam}}
\newcommand{\haus}{\mathcal{H}}
\newcommand{\po}{\mathrm{po}}
\newcommand{\gp}{\mathcal{G}_p}
\newcommand{\cal}{\mathcal}
\newcommand{\HH}{\mathbb{H}}
\newcommand{\green}{\mathcal{G}}
\newcommand{\II}{\mathrm{II}}%
%
%\begin{document}
%
%%\scriptsize \begin{center} Dipartimento di Matematica Pura e Applicata, Universit\`a degli Studi di Padova\\
%%Via Trieste 63, I-35121 Padova (Italy)\\
%%E-mail address: bianchini@dmsa.unipd.it
%%\end{center}
%%\scriptsize \begin{center} Departamento de Matem\'atica, Universidade Federal do Cear\'a\\
%%Av. Humberto Monte s/n, Bloco 914, 60455-760 Fortaleza (Brazil)\\
%%E-mail address: lucio.mari@libero.it
%%\end{center}
%%\scriptsize \begin{center} Dipartimento di Matematica,
%%Universit\`a
%%degli studi di Milano\\
%%Via Saldini 50, I-20133 Milano (Italy)\\
%%E-mail address: marco.rigoli55@gmail.com
%%\end{center}
%%
%%\normalsize
%%
%%%\footnote{\textbf{Mathematic subject classification 2010}: primary
%%%58J05, 35B40; secondary 53C21, 34C11, 35B09.} 
%%
%%
%%\vspace{0.3cm}
%

\author{Bruno Bianchini \and Luciano Mari \and Marco Rigoli}

\title[Yamabe type equations on manifolds]{Yamabe type equations with a sign-changing nonlinearity, and the prescribed curvature problem}

%\author{Alberto Giulio Setti}
%\address{Alberto Giulio Setti \\ Dipartimento di Fisica e Matematica\\
%Universit\`a dell'Insubria - Como\\
%via Valleggio 11\\
%I-22100 Como, Italy, EU}
%\email{alberto.setti@uninsubria.it}

\date{\today}

%\author{Bruno Bianchini}
\address{Bruno Bianchini\\ Dipartimento di Matematica Pura e Applicata, Universit\`a degli Studi di Padova\\
Via Trieste 63, I-35121 Padova (Italy)}
\email{bianchini@dmsa.unipd.it}

%\author{Luciano Mari}
\address{Luciano Mari\\ Departamento de Matem\'atica, Universidade Federal do Cear\'a\\
Av. Humberto Monte s/n, Bloco 914, 60455-760 Fortaleza (Brazil)}
\email{mari@mat.ufc.br}

%\author{Marco Rigoli}
\address{Marco Rigoli \\ Dipartimento di Matematica,
Universit\`a
degli studi di Milano\\
Via Saldini 50, I-20133 Milano (Italy)}
\email{marco.rigoli55@gmail.com}



\maketitle 

%
% inserire \index{espo@$\exp$} subito prima di un simbolo matematico (esempio, $\exp$) per identificare la pagina dove appare.
%\printindex dove voglio che compaia l'indice
%
%

%\pagestyle{headings}

%\pagestyle{fancy}
%% i comandi seguenti impediscono la scrittura in maiuscolo
%% dei nomi dei capitoli e dei paragrafi nelle intestazioni
%\renewcommand{\chaptermark}[1]{\markboth{#1}{}}
%\renewcommand{\sectionmark}[1]{\markright{\thesection\ #1}}
%\fancyhf{} % rimuove l'attuale contenuto dell'intestazione
%% e del pi\`e di pagina
%\fancyhead[LE,RO]{\bfseries\thepage}
%\fancyhead[LO]{\bfseries\rightmark}
%\fancyhead[RE]{\bfseries\leftmark}
%\renewcommand{\headrulewidth}{0.5pt}
%\renewcommand{\footrulewidth}{0pt}
%\addtolength{\headheight}{0.5pt} % riserva spazio per la linea
%\fancypagestyle{plain}{%
%\fancyhead{} % ignora, nello stile plain, le intestazioni
%\renewcommand{\headrulewidth}{0pt} % e la linea
%}

\begin{abstract} 
In this paper, we investigate the prescribed scalar curvature problem on a non-compact Riemannian manifold $(M, \metric)$, namely the existence of a conformal deformation of the metric $\metric$ realizing a given function $\widetilde s(x)$ as its scalar curvature. In particular, the work focuses on the case when $\widetilde s(x)$ changes sign. 
%where a satisfactory theory seems to be still missing. 
%The sign-changing case is considerably more difficult than the case $\widetilde s \le 0$ and still widely open, mainly because some basic tools to produce solutions in a non-compact setting fail to hold, and variational arguments seem difficult to apply in many cases of interest. 
Our main achievement are two new existence results requiring minimal assumptions on the underlying manifold, and ensuring a control on the stretching factor of the conformal deformation in such a way that the conformally deformed metric be bi-Lipschitz equivalent to the original one. The topological-geometrical requirements we need are all encoded in the spectral properties of the standard and conformal Laplacians of $M$. Our techniques can be extended to investigate the existence of entire positive solutions of quasilinear equations of the type
$$
\Delta_{p} u + a(x)u^{p-1} - b(x)u^\sigma = 0
$$
where $\Delta_p$ is the $p$-Laplacian, $\sigma>p-1>0$, $a,b \in L^\infty_\loc(M)$ and $b$ changes sign, and in the process of collecting the material for the proof of our theorems, we have the opportunity to give some new insight on the subcriticality theory for the Schr\"odinger type operator 
$$
Q_V' \ : \ \varphi \longmapsto -\Delta_p \varphi - a(x)|\varphi|^{p-2}\varphi.
$$
In particular, we prove sharp Hardy-type inequalities in some geometrically relevant cases, notably for minimal submanifolds of the hyperbolic space.

%its relations with an appropriate capacity theory, with Hardy-type inequalities and with the $p$-%parabolicity of $M$, which might be of independent interest.  
%
%
%
\end{abstract}

\vspace{0.3cm}

\tableofcontents

\section{Introduction, I: existence for the generalized Yamabe problem}
Generalizations of the classical Yamabe problem on a Riemannian manifold have been the focus of an active area of research over the past 30 years. Among these, the prescribed scalar curvature problem over non-compact manifolds appears to be challenging: briefly, given a non-compact Riemannian manifold $(M^m, \metric)$ with scalar curvature $s(x)$ and a smooth function $\widetilde s \in C^\infty(M)$, the problem asks under which conditions there exists a conformal deformation of $\metric$,
\begin{equation}\label{01}
\widetilde{\metric}=\varphi^2 \metric,   \qquad 0 < \varphi \in C^\infty(M),
\end{equation}
realizing $\widetilde s(x)$ as its scalar curvature. When the dimension $m$ of $M$ is at least $3$, writing $\varphi = u^{\frac{2}{m-2}}$, the problem becomes equivalent to determining a positive solution $u \in C^\infty(M)$ of the Yamabe equation
\begin{equation}\label{02}
\disp \Delta u -\frac{s(x)}{c_m}u+ \frac{\widetilde s(x)}{c_m} u^{\frac{m+2}{m-2}}=0, \qquad c_m = \frac{4(m-1)}{m-2}.
\end{equation}
Here, $\Delta$ is the Laplace-Beltrami operator of the background metric $\metric$. For $m=2$, setting $\varphi= e^{u}$ one substitutes \eqref{02} with
$$ 
2\Delta u-s(x)+\widetilde s(x) e^{2u}=0
$$
where now $u \in C^\infty (M)$ may change sign (see \cite{kazdan}). Hereafter, we will confine ourselves to dimension $m\ge 3$, and $M$ will always be assumed to be connected. Agreeing with the literature, we will call the linear operator in \eqref{02}:  
\begin{equation}\label{laplaconfo}
L_{\metric} \doteq - \Delta - \frac{s(x)}{c_m} = -\Delta + \frac{m-2}{4(m-1)}s(x) 
\end{equation}
the conformal Laplacian of $(M, \metric)$.\par
The original Yamabe problem is a special case of the prescribed scalar curvature problem, namely that when $\widetilde s(x)$ is a constant, and for this reason, in the literature, the prescribed scalar curvature problem is often called the generalized Yamabe problem. Besides establishing existence of a positive solution $u$ of \eqref{02}, it is also useful to investigate its qualitative behaviour since this reflects into properties of $\widetilde\metric$. For instance, $u \in L^{\frac{2m}{m-2}}(M)$ is equivalent to the fact that $\widetilde \metric$ has finite volume. Also, if $u$ is bounded between two positive constants, the identity map 
\begin{equation}\label{defidentita}
i \ \ : \ \ (M, \metric) \longrightarrow (M, \widetilde \metric)
\end{equation} 
is globally bi-Lipschitz, and thus $\widetilde \metric$ inherits some fundamental properties of $\metric$. For instance, geodesic completeness, parabolicity, Gromov-hyperbolicity, etc. (see \cite{gromov, gromov2}). Agreeing with the literature, when $C^{-1} \metric \le \widetilde{\metric} \le C\metric$ for some constant $C>0$ we will say that $\widetilde \metric$ and $\metric$ are uniformly equivalent.\par

%is a quasi-isometry (according to the definition in \cite{kanai}), and thus $\widetilde %\metric$ inherits some fundamental properties of $\metric$. For instance, geodesic %completeness, parabolicity, Gromov-hyperbolicity, etc... (see \cite{kanai, gromov, %gromov2}). Agreeing with the literature, when $i$ in \eqref{defidentita} is a quasi-%isometry we will say that $\widetilde \metric$ and $\metric$ are uniformly equivalent.
%
%  
Given the generality of the geometrical setting, it is reasonable to expect that existence or non-existence of the desired conformal deformation heavily depends on the topological and metric properties of $M$ and their relations with $\widetilde s(x)$. As we shall explain in awhile, a particularly intriguing (and difficult) case is when $\widetilde s(x)$ is allowed to change sign. In this situation, with the exception of a few special cases, a satisfactory answer to the prescribed curvature problem is still missing. To properly put our results into perspective, first we describe some of the main technical problems that arise when looking for solutions of \eqref{02} for sign-changing $\widetilde s(x)$. Then, we briefly comment on some classical and more recent approaches. In particular, we pause to describe in detail four results that allow us to grasp the situation in the relevant examples of Euclidean and hyperbolic spaces and to underline the key features of our new achievements. We stress that, when $\widetilde s(x) \le 0$, there is a vast literature and the interaction between topology and geometry is better understood. Among the various references on the  existence problem, we refer the reader to \cite{avilesmcowen,rrv,rrv2,brs2, litamyang}. \par
If $\widetilde s(x)$ is positive somewhere, basic tools to produce solutions are in general missing. More precisely, uniform $L^\infty$-estimates fail to hold on regions where $\widetilde s(x)$ is non-negative, and comparison theorems are not valid where $\widetilde s(x)$ is positive. This suggests why, in the literature, equation \eqref{02} in a non-compact ambient space has mainly been studied via variational and concentration-compactness techniques (\cite{holcman, zhang}) or radialization techniques (\cite{WMni, naito, kawano, avilesmcowen}). We also quote the interesting method developed in \cite{rrv, rrv2, bae}.\par
To the best of our knowledge, up to now there have been few attempts to adapt the variational approach to (non-compact, of course) spaces other than $\R^m$, \cite{holcman, zhang}. In this respect, a particularly interesting result is the next one due to Q.S. Zhang \cite{zhang}. 
\begin{Theorem}[\cite{zhang}, Thm. 1.1]\label{teo_zhang}
Let $(M^m, \metric)$ be a complete manifold with dimension $m \ge 3$ and scalar curvature $s(x) \ge 0$. Suppose that $\vol(B_r(x)) \le Cr^m$ for some uniform $C$ independent of $x$, and that $M$ has positive Yamabe invariant $Y(M)$:
\begin{equation}\label{Yamabeinvariant}
Y(M) = \inf \left\{ \int_M \Big[ |\nabla \phi|^2+ \frac{s(x)}{c_m}\phi^2\Big] \ \ : \ \ \phi \in \lip_c(M), \ \int_M \phi^{\frac{2m}{m-2}} =1 \right\},
\end{equation}
$c_m$ as in \eqref{02}. Assume further that 
\begin{itemize}
\item[-] $\widetilde s(x) \ge 0, \not\equiv 0$ and $\widetilde s(x)\ra 0$ as $r(x) \ra +\infty$, 
\item[-] $\widetilde s(x)$ is sufficiently flat around at least one of its maximum points.
\end{itemize}
Then, there exists a solution $u \in L^{\frac{2m}{m-2}}(M)$ of \eqref{02} such that 
\begin{equation}\label{upperbound_4}
u \le C \big(1+r(x)\big)^{-\frac{m-2}{2}},
\end{equation}
for some $C>0$. In particular, $\widetilde{\metric}= u^{\frac{4}{m-2}} \metric$ has finite volume and is geodesically incomplete.
\end{Theorem}

\begin{Remark}
\emph{The flatness condition above is the one usually required for the compact Yamabe problem, see \cite{escobar, escobarschoen}.
}
\end{Remark}

The above theorem is not, indeed, the most general statement of Zhang's result, but however the version here is a good compromise between generality and simplicity, and it is enough for the sake of comparison with our main theorems. On the positive side, topological conditions on $M$ are not so demanding. However, we underline that the polynomial volume growth assumption is essential for Zhang's method to work, hence this excludes the case of negatively curved manifolds like the hyperbolic space $\HH^m_\kappa$ of sectional curvature $-\kappa^2$. 
On the contrary, as the recent \cite{bmr3} highlights, the radialization methods developed by W.M. Ni, M. Naito and N. Kawano in \cite{WMni, naito, kawano} on $\R^m$, and by P. Aviles and R. McOwen in \cite{avilesmcowen} for $\HH^m$ are very flexible with respect to curvature control on $M$, but on the other hand they require $M$ to possess a pole (that is, a point $o$ for which the exponential map $\exp_o$ is a diffeomorphism), a quite restrictive topological assumption. We quote the two results, starting from Ni-Naito-Kawano's theorem.

\begin{Theorem}[\cite{WMni, naito, kawano}]\label{teo_ninaitokawano}
Let $\widetilde s(x) \in C^\infty(\R^m)$, $m \ge 3$, and suppose that there exists $B \in C^0(\R)$ such that
\begin{equation}\label{crescescaleugen}
|\widetilde s(x)| \le B\big(r(x)\big) \qquad \text{and} \qquad rB(r) \in L^1(+\infty).
\end{equation}
Then, there exists a small $\gamma_0>0$ such that, for each $\gamma \in (0, \gamma_0)$, there exists a conformal deformation $\widetilde \metric$ of the flat metric $\metric$ such that 
\begin{equation}\label{limit_Ninaitokawano}
\widetilde{\metric}_x \ra \gamma \metric_x \qquad \text{as } \, r(x) \ra +\infty.
\end{equation} 
\end{Theorem}

\begin{Theorem}[\cite{avilesmcowen}, Thm 4]\label{teo_avilesmcowen}
Let $(M^m, \metric)$ be a complete manifold with a pole and dimension $m \ge 3$, and suppose that there exist constants $\bar \kappa \ge\kappa>0$ such that sectional curvature $K$ of $M$ be pinched as follows:
\begin{equation}\label{ipo_AHaviles}
-\bar \kappa^2 \le K \le -\kappa^2 <0, \qquad \text{with }  \quad \bar \kappa^2 < \frac{(m-1)^2}{m(m-2)}\kappa^2.
\end{equation}
Suppose also that $\widetilde s(x) \in C^\infty(M)$ satisfies 
\begin{equation}\label{boundsnearinftyiperb}
-C_1 \le \widetilde s(x) \le -C_2 < 0 \qquad \text{outside of a compact set,} 
\end{equation}
for some constants $C_1,C_2>0$. Then, there exists $\delta>0$ sufficiently small such that, if
\begin{equation}\label{upbound_deltaAviles}
\widetilde s(x) \le \delta \qquad \text{on } \, M, 
\end{equation}
there exists a conformal deformation $\widetilde{\metric}$ realizing $\widetilde s(x)$ and satisfying
\begin{equation}\label{asy_hiperbolico}
C^{-1} \metric \le \widetilde{\metric} \le C\metric \qquad \text{on } \, M, 
\end{equation}
for some positive constant $C$.
\end{Theorem}

\begin{Remark}
\emph{Theorem \ref{teo_avilesmcowen} has later been improved in \cite{rrv} with a different technique: however, the main Theorem 0.1 in \cite{rrv} still requires \eqref{upbound_deltaAviles} and a couple of conditions on the curvatures of $M$ that, though more general than \eqref{ipo_AHaviles}, nevertheless are more demanding than \eqref{secdasopra}, \eqref{iposcalarehyp} appearing in our Corollary \ref{corquattro} below. 
}
\end{Remark}

\begin{Remark}
\emph{For the special case of the Hyperbolic space, in \cite{rrv2} (see Theorem 1.1 therein) the authors were able to guarantee the existence of a solution for the Yamabe equation giving rise to a complete metric even when \eqref{boundsnearinftyiperb} is replaced by the weaker
$$
- C r(x)^2 \le \widetilde s(x) <0 \qquad \text{outside of a compact set.}
$$
The counterpart of this improvement is that a control of the type \eqref{asy_hiperbolico} is no longer  available. We remark that, in Theorem 1.1 of \cite{rrv2}, condition \eqref{upbound_deltaAviles} still appears.
}
\end{Remark}

Inspired by Ni-Naito-Kawano's approach, in \cite{bmr3} we have obtained sharp existence theorems for \eqref{02} (and, more generally, for \eqref{03} below) on manifolds possessing a pole $o$ via mild assumptions on the radial sectional curvature $K_\rad$ (the sectional curvature restricted to $2$-planes containing $\nabla r$, with $r(\cdot) = \dist(\cdot, o)$). In the particular case of manifolds close to the hyperbolic space, our outcome has been the following result. Observe that condition \ref{crescescalhypgen} below guarantees the existence of solutions even when $\widetilde s(x)$ is strongly oscillating. On the other hand, \eqref{decayhypgen} implies that the conformally deformed metric is incomplete and has finite volume.

%
%Our main achievement has been to quantify, in a sharp and explicit way, the growth of the new scalar %curvature in dependence of $K_\rad$ and $s(x)$ in order to produce optimal results. The presence of a %nontrivial linear term $s(x)u$ in \eqref{02} made things much subtler than in \cite{WMni, naito, kawano}, %and required the introduction of new devices. We refer the interested reader to \cite{bmr3} for %deepening. 

\begin{Theorem}[\cite{bmr3}, Thm 2]\label{cor_tipoiperb}
Let $(M,\metric)$ be a complete manifold of dimension $m\ge 3$, with a pole $o$ and sectional curvature $K$ satisfying 
\begin{equation}\label{iposezionalehyp}
- \kappa^2 - \mathcal{K}\big(r(x)\big) \le K(x) \le - \kappa^2,
\end{equation}
for some constant $\kappa>0$ and some non-negative $\mathcal{K}\in C^0(\R^+_0) \cap L^1(\R^+)$. Suppose that the scalar curvature $s(x)$ of $M$ is such that
\begin{equation}\label{ssimilehyp}
s(x) \ge - \frac{(m-1)^3\kappa^2}{m-2} \qquad \text{on } M.
\end{equation}
Then, for each $\widetilde{s}(x) \in C^\infty(M)$ satisfying, for some $B \in C^0(\R^+_0)$,
\begin{equation}\label{crescescalhypgen}
|\widetilde{s}(x)| \le B(r(x)), \qquad e^{-2\kappa r }B(r) \in L^1(+\infty),
\end{equation} 
the metric $\metric$ can be conformally deformed to a smooth metric $\widetilde{\metric}$ of scalar curvature $\widetilde{s}(x)$, satisfying 
\begin{equation}\label{decayhypgen}
\Gamma_1 e^{-2\kappa r (x)}\metric_x \le \widetilde{\metric}_x \le \Gamma_2 e^{-2\kappa r (x)}\metric_x \qquad \forall \, x\in  M,
\end{equation}
for some $0<\Gamma_1 \le \Gamma_2$. In particular, $\widetilde \metric$ is incomplete and has finite volume. Furthermore, $\Gamma_2$ and consequently $\Gamma_1$ can be chosen to be as small as we wish. 
\end{Theorem}

\begin{Remark}
\emph{The growth conditions \eqref{crescescaleugen} and \eqref{crescescalhypgen} are sharp: it is proved in \cite{chenglin} (for $\R^m$) and \cite{BR} (for $\HH^m_\kappa$) that no conformal deformation exists whenever $\widetilde s(x) \le 0$ on $\R^m$ (respectively, on $\HH^m_\kappa$) and 
$$
\widetilde s(x) \le -\frac{C}{r(x)^2 \log r(x)} \qquad \left(\text{respectively, } \ \ \widetilde s(x) \le -  \frac{C e^{2\kappa r (x)}}{r(x) \log r(x)} \right)
$$
for some $C>0$ and large $r(x)$. 
%Taking into account the conditions on $\widetilde s(x)$ in Theorem \ref{teo_zhang}, this means, loosely speaking, that there is further ``space" to enlarge the positive part of $\widetilde s(x)$ still guaranteeing the existence of solutions (that possibly give rise to incomplete metrics), while no space is left on the negative side of $\widetilde s(x)$.
}
\end{Remark}

The above four results are, to the best of our knowledge, an up-to-date account of what is known on the prescribed scalar curvature problem, in dimension $m \ge 3$ and with sign-changing $\widetilde s(x)$, on non-compact manifolds close to $\R^m$ and $\HH^m_\kappa$. Figures 1 and 2 below summarize Theorems \ref{teo_zhang} to \eqref{cor_tipoiperb} when assumptions overlap. 

		 \begin{figure}[ht] 
		\includegraphics[width=9cm]{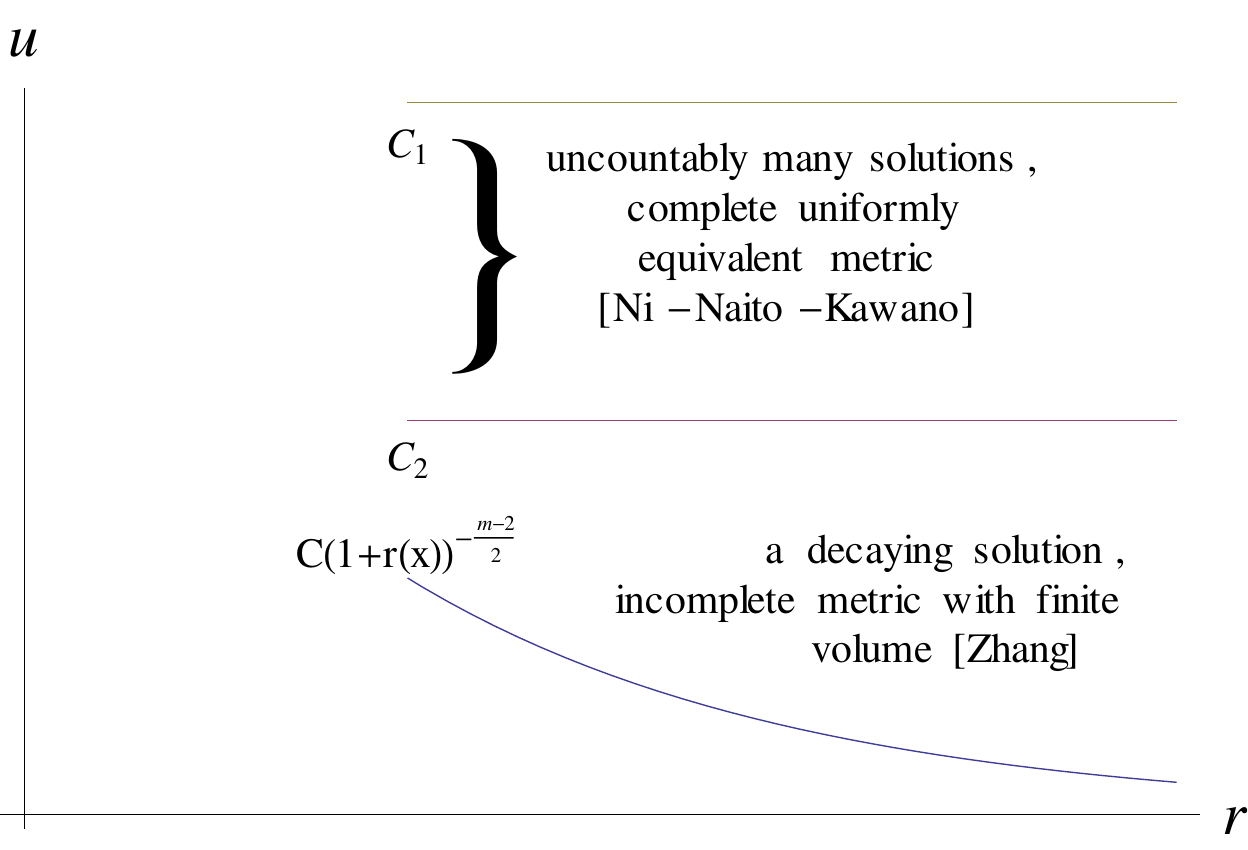}%{Euclidean.jpg}
		\caption{Euclidean space, \eqref{crescescaleugen} and Zhang's assumptions on $\widetilde s(x)$ in force.}
		 \end{figure}		 
 \begin{figure}[ht] 
		\includegraphics[width=9cm]{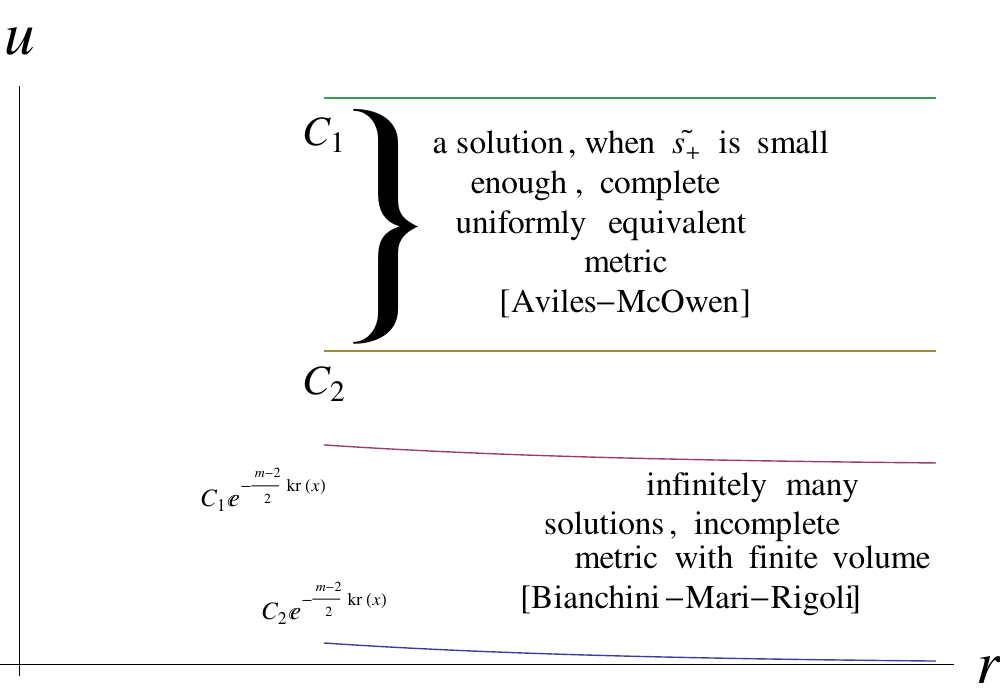}%{Hyperbolic.jpg}
		\caption{Manifolds close to $\HH^m_\kappa$, $-C_1 \le \widetilde s(x) \le -C_2 < 0$ for large $x$.}
		 \end{figure}

A first step in the direction of removing the pole requirement has been taken in \cite{bmr4} by adapting some ideas of \cite{bmr3} via the use of Green functions. Unfortunately, even though the requirements on $s(x)$ and $\widetilde s(x)$ in Theorem 5 of \cite{bmr4} are sharp, they express in a form that is generally difficult to check. In summary, the task of obtaining results of the type above but with a substantial weakening of the geometric assumptions calls for new ideas, and this is the objective of the present work. More precisely, we have a twofold concern in this paper. First, we aim to produce an existence theorem for sign-changing $\widetilde s(x)$ where topological and geometrical conditions are confined to a minimum. Second, we also want to keep control on the conformally deformed metric, in particular in such a way that $\metric$ and $\widetilde{\metric}$ are uniformly equivalent. Our contributions are Theorems \ref{coruno}, \ref{teo_tipohyperb} and \ref{teodue} below, a special case of Theorems \ref{teouno}, \ref{teouno_bis} which we are going to describe in awhile. 
\begin{Notation}
\emph{Hereafter, given $b\in L^\infty_\loc(M)$, we respectively denote with $b_+$ and $b_-$ its positive and negative parts, so that $b = b_+-b_-$. For $a \in L^\infty_\loc(M)$, we will write
$$
a(x) = O \big(b(x)\big) \qquad \big(\text{respectively,} \quad a(x) \asymp b(x) \ \big) \quad \text{as $x$ diverges}
$$
to indicate that there exists a constant $C>0$ and a compact set $\Omega$ such that 
$$
a(x) \le Cb(x) \qquad \big(\text{respectively,} \quad C^{-1}b(x) \le a(x) \le Cb(x) \ \big)
$$
on $M \backslash \Omega$.
}
\end{Notation}
Our first result deals with the case of non-parabolic manifolds with non-negative scalar curvature. We recall that a manifold $M$ is said to be non-parabolic if it admits a positive, non-constant solution of $\Delta u \le 0$. The notion of non-parabolicity will be recalled later in a more general setting (see Proposition \ref{prop_hyperbolicity_simple} and the subsequent discussion), here we limit to refer the interested reader to \cite{grigoryan} for deepening. The following theorem shall be compared to Theorems \ref{teo_zhang} and \ref{teo_ninaitokawano}. In particular, to compare Theorem \ref{coruno} below with Zhang's Theorem \ref{teo_zhang} we need some tools that will be defined in the next introduction, and therefore we postpone the analysis to Remark \ref{rem_Yamabeshort}.
\begin{Theorem}\label{coruno}
Let $(M, \metric)$ be a non-parabolic manifold of dimension $m \ge 3$ and scalar curvature $s(x)$ satisfying
\begin{equation}\label{scalnonneginteg}
s(x) \ge 0 \text{ on M,} \qquad s \in L^1(M),
\end{equation}
and let $\widetilde s \in C^\infty(M)$ with the following properties:
\begin{equation}\label{nuovascalinteg}
\text{$\widetilde s_+$ has compact support,} \qquad \widetilde s \in L^1(M).
\end{equation}
Then, for each constant $C>0$, $\metric$ can be pointwise conformally deformed to a new metric $\widetilde{\metric}$ of scalar curvature $\widetilde s(x)$ such that $\widetilde{\metric} \le C\metric$. Moreover, if $s$ and $\widetilde s$ have compact support, each such $\widetilde{\metric}$ can be chosen to be uniformly equivalent to $\metric$. 
\end{Theorem}
\begin{Remark}
\emph{Non-parabolicity is a very mild requirement, and it is necessary to guarantee existence in all the cases investigated in Theorem \ref{coruno}. In fact, if $M$ is scalar flat and taking $\widetilde s(x)$ to be compactly supported, non-negative and not identically zero, any eventual solution $u$ of the Yamabe equation \eqref{02} would be a (non-constant) positive solution of $\Delta u = - \widetilde s(x)/c_m u^{\frac{m+2}{m-2}} \le 0$, showing that $M$ must be non-parabolic. 
}
\end{Remark}

Theorem \ref{coruno} applies, for instance, to the physically relevant setting of asymptotically flat spaces. According to \cite{leeparker}, $(M^m, \metric)$ is called asymptotically flat if 
\begin{itemize}
\item[-] its scalar curvature $s(x)$ satisfies \eqref{scalnonneginteg}, 
\item[-] there exists a compact set $K \subset M$ such that each connected component $U_j$ of $M\backslash K$ has a global chart $\Psi_j : (\R^m\backslash B_R(0), \metric_{\mathrm{can}}) \ra U_j$ for which the local expression $g_{ij}$ of $\metric$ satisfies
\begin{equation}\label{asflat}
|g_{ij} - \delta_{ij}| = O\big(r^{-p}\big), \quad |\partial_k \, g_{ij}| = O\big( r^{-p-1}\big), \quad |\partial^2_{kl} \, g_{ij}| = O\big( r^{-p-2}\big)
\end{equation}
as $r(x) = |x| \ra +\infty$, for some $p>(m-2)/2$ and for each $1 \le i,j,k,l \le m$. 
\end{itemize}
\begin{Corollary}\label{cor_asflat}
Let $(M, \metric)$ be an asymptotically flat manifold of dimension $m \ge 3$. Then, for each smooth function $\widetilde s(x)$ satisfying
\begin{equation}
\text{$\widetilde s_+$ has compact support} \qquad \text{and} \qquad \widetilde s \in L^1(M),
\end{equation} 
and for each constant $C>0$, $\widetilde{s}(x)$ is realizable via a conformal deformation $\widetilde \metric$ of $\metric$ satisfying 
$\widetilde \metric \le C\metric$. Furthermore, if $\widetilde s \equiv 0$ outside some compact set, then $\widetilde \metric$ can be chosen to be uniformly equivalent to $\metric$.
\end{Corollary}

%\begin{Corollary}\label{cortre}
%Let $(M, \metric)$ be a manifold of dimension $m \ge 3$ satisfying
%\begin{equation}\label{iporicci}
%\Ricc \ge 0 \quad \text{on } M, \qquad \Ricc = 0 \quad \text{outside some compact set.}
%\end{equation}
%Then, if $\widetilde s(x)$ has compact support, there exists a conformal deformation $\widetilde \metric$ of $\metric$ realizing $\widetilde s(x)$ and uniformly equivalent to $\metric$.
%\end{Corollary}
%
%ATTENZIONE!! SIAMO SICURI CHE NON CI SIANO RISULTATI DI GAP SOTTO L'IPOTESI \eqref{iporicci}???
%
Now, we deal with manifolds whose original scalar curvature can be somewhere negative.
\begin{Theorem}\label{teo_tipohyperb}
Let $(M^m, \metric)$ be a non-parabolic Riemannian manifold of dimension $m \ge 3$ with scalar curvature $s(x)$. Suppose that the conformal Laplacian \eqref{laplaconfo} admits a positive Green function on $M$. \par 
Let $\widetilde s(x) \in C^\infty(M)$ be such that 
\begin{equation}\label{daisu}
\text{$\widetilde s_+$ has compact support,} \qquad  \widetilde s(x) \asymp s(x) \, \text{ as $x$ diverges.} 
\end{equation}
Then, there exists $\delta>0$ such that if
\begin{equation}\label{labbrutta}
\widetilde s(x) \le \delta \qquad \text{on } M, 
\end{equation}
then $\widetilde s$ is realizable via a uniformly equivalent, conformal deformation $\widetilde \metric$ of $\metric$.
\end{Theorem}
\begin{Remark}
\emph{Observe that \eqref{daisu} implies that the original scalar curvature $s(x)$ is non-positive outside a compact set.  
}
\end{Remark}
Note that Theorems \ref{coruno} and \ref{teo_tipohyperb} seem to be new even in the simpler case $\widetilde s(x) \le 0$ on $M$. In this respect, these are skew with the main theorem in \cite{litamyang} and with Theorem 2.30 of \cite{bmr2}.\par

The requirement \eqref{laplaconfo} shows the central role played by the conformal Laplacian $L_\metric$ for the Yamabe equation. The relevance of $L_{\metric}$ for the original Yamabe problem is well-known and highlighted, for instance, in the comprehensive \cite{leeparker}. We will spend a considerable part of the paper to discuss on assumptions like \eqref{laplaconfo}. Clearly, Theorem \ref{teo_tipohyperb} is tightly related to Aviles-McOwen's Theorem \ref{teo_avilesmcowen}, a parallel which is even more evident in view of the next
\begin{Corollary}\label{corquattro}
Let $(M, \metric)$ be a complete manifold of dimension $m \ge 3$ with a pole $o$ and sectional curvature $K$ satisfying 
\begin{equation}\label{secdasopra}
K \le -\kappa^2,
\end{equation}
for some constant $H > 0$. Suppose further that 
\begin{equation}\label{iposcalarehyp}
s(x) \ge -\frac{(m-1)^3}{(m-2)}\kappa^2
\end{equation}
on $M$. Let $\widetilde s(x) \in C^ \infty(M)$ satisfying
\begin{equation}\label{nuovascalarehyp}
\disp -C_1 \le \widetilde s(x) \le -C_2 <0 \qquad \text{outside some compact set,}
\end{equation}
for some constants $0<C_2 < C_1$. Then, there exists $\delta>0$ such that if 
\begin{equation}\label{labbrutta2}
\widetilde s(x) \le \delta \qquad \text{on } M, 
\end{equation}
then $\widetilde s(x)$ is realizable by a conformal deformation $\widetilde{\metric}$ of $\metric$ which is uniformly equivalent to $\metric$. 
\end{Corollary}
\begin{Remark}\label{rem_pocapositivita}
\emph{Corollary \ref{corquattro} improves on Theorem \ref{teo_avilesmcowen}, since requirement \eqref{ipo_AHaviles} in Theorem \ref{teo_avilesmcowen} implies \eqref{secdasopra}, \eqref{iposcalarehyp}. In particular, the hyperbolic space $\HH^m_\kappa$ of sectional curvature $-\kappa^2$ satisfies all the assumptions of Corollary \ref{corquattro}, being $s(x) = -m(m-1)\kappa^2$. Moreover, as the proof in \cite{avilesmcowen} shows, \eqref{ipo_AHaviles} is essential to ensure the existence of $\delta>0$ in \eqref{upbound_deltaAviles}; the case of equality in \eqref{iposcalarehyp} seems, therefore, hardly obtainable with the approach described in \cite{avilesmcowen}. It is worth to observe that the existence of a pole and the pinching assumption \eqref{ipo_AHaviles} on the sectional curvature are needed in \cite{avilesmcowen} to apply both the Laplacian comparison theorems (from above and below) for the distance function $r(x)= \dist(x,o)$ in order to find suitable radial sub- and supersolutions. On the contrary, here the weaker \eqref{secdasopra} and \eqref{iposcalarehyp} are just used to ensure that the conformal Laplacian has a positive Green function.
}
\end{Remark}
%
%
%In Section \ref{sec_applications} below, we discuss in detail other relations between Corollary \ref{corquattro} and results in the recent literature for the prototype case of $\HH^m_\kappa$. Combining with contributions in \cite{avilesmcowen,rrv,rrv2,rigolizamperlin,bmr3}, the existence problem for the prescribed scalar curvature equation on $\HH^m_\kappa$ reveals a quite intriguing scenario, that is different and richer than the one appearing in the case of $\R^m$.\\
%\par
%%
%
We pause for a moment to comment on assumption \eqref{labbrutta}. Both Theorems \ref{coruno} and \ref{teo_tipohyperb} will be consequences of Theorem \ref{teodue} below, and thus they will be proved via a common technique. The reason why assumption \eqref{labbrutta} is required in Theorem \ref{teo_tipohyperb} but not in Theorem \ref{coruno} can be summarized in the existence, in the second case, of a global, positive supersolution for the conformal Laplacian (that is, a solution $w$ of $L_{\metric}w \ge 0$ on $M$) which is bounded both from below and from above by positive constants; one can take, for instance, $u \equiv 1$. Such a function is not possible to construct in the general setting of Theorem \ref{teo_tipohyperb} (see Remark \ref{rem_counterex1} below for deepening). We stress that, unfortunately, the value of $\delta$ in Corollary \ref{corquattro} is not explicit: indeed, it depends on a uniform $L^\infty$ bound for solutions of some suitable PDEs, which is shown to exist via an indirect method.\\
The need of \eqref{labbrutta} to obtain existence for $\widetilde{\metric}$ is investigated in Remark \ref{rem_labrutta}. It is very interesting that the same condition \eqref{labbrutta} appears both in our theorem and in Aviles-McOwen's one, as well as in Theorems 0.1 in \cite{rrv} and 1.1 in \cite{rrv2}, although the techniques to prove them are different. This may suggest that, in general, \eqref{labbrutta} could not be removable. However, at present we still have no counterexample showing that \eqref{labbrutta} is necessary. For future work, we thus feel interesting to investigate the next 
\begin{Question}
Can assumption \eqref{labbrutta} in Theorem \ref{teo_tipohyperb} be removed, even without expecting  the new metric to be uniformly equivalent to $\metric$?
\end{Question}
%
%we have not been able to remove \eqref{labbrutta}  
% 
%
%
%For instance, in the hyperbolic space $\HH^m_\kappa$, for each solution of 
%$$
%Lu = -\Delta u - \frac{m(m-2)\kappa^2}{4} u = 0,  
%$$
%its spherical mean $\bar u$ (which still solves $L\bar u=0$) satisfies
%$$
%\bar u(r) \sim C \exp{ -\frac{m-2}{2} H r\} \qquad \text{as } r= r(x) \ra +\infty,
%$$
%for some $C>0$. A proof of this fact, can be deduced from Theorem 
%
%

%
%to investigate the existence of positive solutions of \eqref{02} when the assigned scalar curvature $K(x)$ is allowed to change sign and the manifold has non trivial topology, in the sense that $M$ is not, in general, diffeormphic to $\R^n$. We mention that in a recent paper \cite{bmr3} we succeded in solving the problem on complete manifold with a pole under very mild "sign changing" conditions on $K(x)$ related to the geometry of $M$. \par 
%When the sign of $K(x)$ is fixed, we refer to the litterature, for instance \cite{holokoskela}, \cite{tolksdorf2}, \cite{tolksdorf}, \cite{troyanov2} and the references therein for various results. 
%

\section{Introduction, II: our main results in their general setting}
Although the prescribed scalar curvature problem is the main focus of our investigation, the techniques developed here allow us to study more general classes of PDEs, namely nonlinear extensions (described in \eqref{04}) of the equation 
\begin{equation}\label{03}
\Delta u + a(x) u -b(x) u^\sigma = 0 \qquad \text{on } M, \qquad u>0 \ \text{ on } M
\end{equation}
with $\sigma>1$, $a,b \in C^\infty(M)$ and sign-changing $b$. Note that the signs of $a,b$ are reversed with respect to those of $s,\widetilde s$ in \eqref{02}, and that $\sigma$ can be greater than $\frac{m+2}{m-2}$, preventing a direct use of variational techniques. However, when $\sigma \le \frac{m+2}{m-2}$ and $b(x) < 0$ on $M$, the investigation of \eqref{03} on Euclidean space is still the core of a very active area of research. In this respect, we quote the seminal \cite{brezisnirenberg} and, for sign-changing $b(x)$ (and singular $a(x)$), the recent \cite{fellipistoia}.\par 
As a matter of fact, even for \eqref{03} the spectral properties of the linear part $L = -\Delta -a(x)$ play a prominent role, in particular the analysis of the fundamental tone $\lambda_1^L(M)$ of the Friedrichs extension of $\left( L,C^\infty_c (M) \right)$. We recall that $\lambda^L_1(M)$ is characterized via the Rayleigh quotient as follows:
\begin{equation}\label{primoautovalore}
\lambda_{1}^{L}(M) = \inf_{0 \not \equiv \varphi \in \lip_c(M)} \frac{\int_M \big[|\nabla \varphi|^2 -a(x) \varphi^2\big]\di x}{\|\varphi\|^2_{L^2(M)}}.
\end{equation}
For instance, if $\lambda_1^L (M) <0$, the situation is somewhat rigid:
\begin{itemize}
\item[(a)] if $b(x) \le 0$, then \eqref{03} has no positive solutions. This follows from a direct spectral argument\footnote{Indeed, assume the contrary and let $u>0$ solves \eqref{03}; then, $u$ is a positive solution of $Lu \ge 0$, and a result of \cite{fischercolbrieschoen, mosspie, allegretto} implies that $\lambda_1^L(M) \ge 0$, contradicting our assumption.}; 
\item[(b)] if $b(x) \ge 0$ and the zero set of $b$ is small in a suitable spectral sense, then there always exist a minimal and a maximal (possibly coinciding) positive solutions of \eqref{03}; see \cite{litamyang, prslogistic} and Section 2.4 in \cite{bmr2}. 
\end{itemize}
It is important to underline that, in both cases, the geometry of $M$ only reveals via the spectral properties of $L$. In other words, no \emph{a-priori} assumptions of completeness of $M$, nor curvature nor topological requests are made. As suggested by $(a)$ and $(b)$ above, it seems that the subtler case is that of investigating existence under the assumption $\lambda_1^L (M) \ge 0$. This condition is often implicitly met in the literature and it is automatically satisfied in many geometric situations. This happens, for instance, for Theorems \ref{teo_zhang} to \ref{cor_tipoiperb}.\par 
There is another aspect of the above picture which is worth mentioning. Partial differential equations similar to \eqref{03} are of interest even for quasilinear operators more general than the Laplacian. Just to give an example, of a certain importance in Physics, we can consider the general equation for radiative cooling
\begin{equation}
\kappa^{-1}\diver \big(\kappa |\nabla u|^{p-2}\nabla u\big) - \big(\tau\kappa^{-1}\big) u^4 = 0 \qquad \text{on } \R^m,
\end{equation}
where $\kappa>0$ is the coefficient of the heat conduction and $\tau$ is a function describing the radiation (see \cite{pucciserrin}, p.9). The existence problem for this type of quasilinear PDEs when the coefficient of the nonlinearity changes sign seems to be quite open. This suggests to extend our investigation to the existence of positive solutions to the quasilinear, Yamabe-type equation
\begin{equation}\label{04}
\Delta_{p,f} u + a(x) u^{p-1} -b(x) F(u) = 0 \qquad \text{on } M, 
\end{equation}
where $f \in C^\infty(M)$,  
\begin{equation}\label{05}
\Delta_{p,f}u = e^{f}\diver \big(e^{-f} |\nabla u|^{p-2} \nabla u \big)
\end{equation}
and $F(u)$ is a nonlinearity satisfying the following assumptions:
\begin{equation}\label{assu_F}
\left\{ \begin{array}{l}
F \in C^0(\R), \quad F(0)=0, \quad F>0 \ \text{ on } \R^+, \\[0.2cm]
\disp \frac{F(t)}{t^{p-1}} \ \text{ is strictly increasing on $\R^+$}, \\[0.3cm]
\disp \lim_{t \ra 0^+} \frac{F(t)}{t^{p-1}} = 0, \qquad \lim_{t \ra +\infty} \frac{F(t)}{t^{p-1}} = +\infty
\end{array}\right.
\end{equation} 
The prototype example of $F(t)$ is $F(t) = t^\sigma$ for $\sigma>p-1$ and $t \ge 0$. Of course \eqref{03} is recovered by choosing $p=2$ and $f$ constant. We underline that, even for \eqref{04} in the Euclidean space and with $F(t)=t^\sigma$, there seems to be no result covering the cases described in Theorems \ref{teouno} and \ref{teouno_bis}. The family of operators in \eqref{05} above encompasses two relevant geometrical cases: $\Delta_{p,0}$ the standard $p$-Laplacian $\Delta_p$, and $\Delta_{2,f}$ the drifted Laplacian, $\Delta_f$, appearing, for instance, in the analysis of Ricci solitons and quasi-Einstein manifolds. Note that the radiative cooling equation is of type \eqref{04} provided $1<p \le 5$. Note also that, since the definition of $\Delta_{p,f}$ is intended in the weak sense, solutions will be, in general, only of class $C^{1,\mu }_\loc(M)$ by \cite{tolksdorf}.   

\begin{Notation}
\emph{Hereafter, with a slight abuse of notation, with $C^{1,\mu}_\loc(M)$ we mean that for each relatively compact open set $\Omega \subset M$, there exists $\mu = \mu(\Omega) \in (0,1)$ such that $u \in C^{1,\mu}(\overline \Omega)$.
}
\end{Notation}

\begin{Remark}
\emph{We stress that, with possibly the exception of Theorem \ref{teo_zhang} when $F(t)=t^\sigma$ and $\sigma \le p^*-1$, the techniques used to prove Theorems \ref{teo_zhang} to \ref{cor_tipoiperb} seem hard to extend to deal with \eqref{05} for nonradial $f$, even for $p=2$. For constant $f$, it seems also very difficult to adapt them to investigate \eqref{05} when $p \neq 2$. In particular, the transformation performed in \cite{bmr3} for $p=2$ to absorb the linear term $a(x)u$ can only be applied when the driving operator is linear.  
}
\end{Remark}

To state our main result, Theorem \ref{teouno} below, we need to introduce some terminology. Let $\di \mu_f$ be the weighted measure $e^{-f} \di x$, with $\di x$ the Riemannian volume element on $M$. For $V  \in L^\infty_\loc(M)$, we consider the functional $Q_V$ defined on $\lip_c (M)$ by
\begin{equation}\label{06}
Q_V(\varphi) = \frac{1}{p}\left[\int_M |\nabla \varphi|^p \di \mu_f - \int_M V |\varphi|^p \di \mu_f\right]. 
\end{equation}
Its Gateaux derivative $Q'_V$ is given by 
\begin{equation}\label{07}
Q_V' (w)= -\Delta_{p,f}w - V |w|^{p-2}w \qquad \text{for } w \in W^{1,p}_\loc(M).
\end{equation}
When $M = \R^m$, the spectral properties of $Q_V$ have been investigated in \cite{allegrettohuang, allegrettohuang2, garciamelian_sabinadelis} and in a series of papers by Y. Pinchover and K. Tintarev (see in particular \cite{pinchovertintarev,pinchovertintarev2}). From now on, we follow the notation and terminology in \cite{pinchovertintarev2}. In the linear case $p=2$, $f \equiv 0$, that is, for the Schr\"odinger operator $Q'_V= -\Delta-V$, we refer the reader to \cite{mosspie, allegretto, piepenbrink, piepenbrink2, murata}.\par 
To begin with, and  according to \cite{murata, pinchovertintarev},  we recall the following:
\begin{Definition}\label{subcritic}
For $V  \in L^\infty_\loc(M)$, define $Q_V$ as in \eqref{06} and let $\Omega \subseteq M$ be an open set.  
\begin{itemize}
\item[i)] $Q_V$ is said to be \emph{\textbf{non-negative}} on $\Omega$ (shortly, $Q_V \ge 0$) if and only if $Q_V(\varphi)\ge 0$ for each $\varphi \in \lip_c(\Omega)$, that is, if and only if the Hardy type inequality
\begin{equation}\label{08}
\int_M V(x)|\varphi|^p\di \mu_f \le \int_M |\nabla \varphi|^p\di \mu_f \qquad \forall \  \varphi \in \lip_c(\Omega).
\end{equation}
holds.
\item[ii)] $Q_V$ is said to be \emph{\textbf{subcritical}} (or \emph{\textbf{non-parabolic}}) on $\Omega$ if and only if $Q_V \ge 0$ there, and there exists $w \in L^1_\loc(\Omega)$, $w \ge 0$, $w \not\equiv 0$ on $\Omega$, such that
\begin{equation}\label{09}
\int_M w(x) |\varphi|^p\di \mu_f \le Q_V(\varphi) \qquad \forall \  \varphi \in \lip_c(\Omega). 
\end{equation}
\end{itemize}
\end{Definition}
Sometimes, especially in dealing with the prescribed scalar curvature problem and when no possible confusion arises, we also say that $Q_V'$, and not $Q_V$, is non-negative (or subcritical). The term ``non-parabolic" is justified by the following statement for $Q_0$ (that is, $Q_V$ with $V(x) \equiv 0$), which is part of Proposition \ref{prop_hyperbolicity} below:

\begin{Proposition}\label{prop_hyperbolicity_simple}
Let $(M, \metric)$ be Riemannian, $f \in C^0(M)$ and $p>1$. Then, $Q_0$ is subcritical on $M$ if and only if there exists a non-constant, positive weak solution $g \in C^0(M) \cap W^{1,p}_\loc(M)$ of $\Delta_{p,f} g \le 0$.
\end{Proposition}
According to the literature, the existence of such $g$ is one of the equivalent conditions that characterize $M$ as being not $p$-parabolic; there are various other characterizations of $p$-parabolicity, given in terms of Green kernels, $p$-capacity of compact sets, Ahlfors' type maximum principles, and so on. We refer to the survey \cite{grigoryan} for deepening in the linear case $p=2$, and to (see \cite{troyanov, pigolasettitroyanov, holokoskela, holopainen2, holopainen}) for $p \neq 2$. The equivalence in Proposition \ref{prop_hyperbolicity_simple} has been observed, in the linear setting, by \cite{ancona, carron, liwang3}, and for $p \neq 2$ it has also recently been proved in \cite{dambrosiodipierro} with a technique different from our.\par 
In fact, all of these characterizations of the non-parabolicity of $-\Delta_{p,f}$ can be seen as a special case of a theory developed in \cite{murata, pinchovertintarev_JFA} (when $p=2$) and in \cite{pinchovertintarev,pinchovertintarev2} for operators $Q_V$ with potential. In Section \ref{sec_criticality}, we recall the main result in \cite{pinchovertintarev, pinchovertintarev2}, the ground state alternative, and we give a proof of it by including a further equivalent condition, see Theorem \ref{teo_alternative} below; as a corollary, we prove Propositions \ref{prop_hyperbolicity} and \ref{prop_hyperbolicity_simple}. 

\begin{Remark}
\emph{In the prescribed scalar curvature problem, the role of $\metric$ and $\widetilde \metric$ can be exchanged. Such a symmetry suggests that those geometric conditions which are invariant with respect to a conformal change of the metric turn out to be more appropriate to deal with the Yamabe equation. This is the case for the non-negativity and the subcriticality of the conformal Laplacian $L_{\metric}$ of $(M, \metric)$ in \eqref{laplaconfo}. In fact, the covariance of $L$ with respect to the conformal deformation $\disp\widetilde {\metric} = u^{\frac{4}{m-2}} \metric$ of the metric:
$$
L_{\widetilde \metric}(\cdot) = u^{-\frac{m+2}{m-2}}L_{\metric}(u \cdot)
$$ 
implies that, for each $0 \le w \in L^1_\loc(M)$ and $\varphi \in \lip_c(M)$, 
$$
\disp \int_M \Big[\|\widetilde{\nabla} \varphi\|^2 + \frac{\widetilde{s}}{c_m}\varphi^2 \Big] \di \widetilde 
x - \int_M \widetilde w\varphi^2 \di \widetilde x = \int_M \Big[|\nabla(u\varphi)|^2+ \frac{s}{c_m}(u\varphi)^2\Big]\di x - \int_M w (u\varphi)^2 \di x,
$$
where $\sim$ superscript indicates quantities referred to $\widetilde \metric$, $\|\cdot\|$ is its induced norm, and 
$$
\widetilde w = wu^{-\frac{4}{m-2}} \in L^1_\loc(M), \quad \widetilde w \ge 0.
$$
Consequently, $L_{\metric}$ is non-negative (resp. subcritical) if and only if so is $L_{\widetilde \metric}$.
}
%
%\begin{equation}\label{stessospettro}
%\disp \int_M \|\widetilde{\nabla} \varphi\|^2\di \widetilde x
% + \int_M \frac{\widetilde{s}}{c_m}\varphi^2\di \widetilde 
%x = \int_M |\nabla(u\varphi)|^2\di x +\int_M
%\frac{s}{c_m}(u\varphi)^2\di x
%\end{equation}
%for each $\varphi \in \lip_c(M)$, where a $\sim$ superscript indicates quantities referred to $\widetilde \metric$, and $\|\cdot\|$ is its induced norm. Moreover, for each $w \in L^1_\loc(M)$, $w \ge 0$,
%$$
%\disp \int_M \Big[\|\widetilde{\nabla} \varphi\|^2 + \frac{\widetilde{s}}{c_m}\varphi^2 \Big] \di \widetilde 
%x - \int_M \widetilde w\varphi^2 \di \widetilde x = \int_M \Big[|\nabla(u\varphi)|^2+ \frac{s}%{c_m}(u\varphi)^2\Big]\di x - \int_M w (u\varphi)^2 \di x
%$$
%for $\varphi \in \lip_c(M)$, where
%The above two equalities show that $L_{\metric}$ is non-negative (resp. subcritical) if and only %if so is $L_{\widetilde \metric}$. In other words, the non-negativity and the subcriticality of %the conformal Laplacian are invariant properties under conformal deformations of the metric.
\end{Remark}
\begin{Remark}\label{rem_Yamabeshort}
\emph{As a direct consequence of the ground state alternative, the positivity of the Yamabe invariant $Y(M)$ in Zhang's Theorem \ref{teo_zhang} implies that $L_{\metric}$ is subcritical (see Remark \ref{rem_Yamabeinv}). On the other hand, in our Theorem \ref{coruno} the subcriticality of $L_{\metric}$ follows combining the non-parabolicity of $M$ and $s(x) \ge 0$. Although, in general, the positivity of $Y(M)$ might not imply the non-parabolicity of $M$, this is so if $M$ is scalar flat outside a compact set and $\vol(M)=+\infty$. Indeed, if $s(x) \equiv 0$ outside a compact set $K$, then $Y(M)>0$ gives the validity of an $L^2$-Sobolev inequality on $M\backslash K$, and coupling with $\vol(M)=+\infty$ the non-parabolicity of $M$ follows by a result in \cite{carron2, pigolasettitroyanov}. We underline that, in the same assumptions, again by \cite{carron2, pigolasettitroyanov} property $\vol(M)=+\infty$ is automatic when $M$ is geodesically complete. Summarizing, if the manifold in Zhang's Theorem \ref{teo_zhang} is scalar flat near infinity, the geometric requirements there properly contain those of our Theorem \ref{coruno}. 
}
\end{Remark}

We are now ready to state
\begin{Theorem}\label{teouno} 
Let $M^m$ be a Riemannian manifold, $f \in C^\infty(M)$ and $p \in (1,+\infty)$. Suppose that $Q_0$ is subcritical on $M$, and let $a \in L^\infty_\loc(M)$ be such that $Q_a$ is subcritical on $M$. Consider $b \in L^\infty_\loc(M)$, and assume
\begin{itemize}
\item[$i)$] $b_-(x)$ has compact support;
\vspace{0.1cm}
\item[$ii)$] $a(x)=O\big(b(x)\big)$ as $x$ diverges;
\vspace{0.1cm}
\item[$iii)$] for some $\theta>0$, $\big(a(x)-\theta b_+(x) \big)_- \in L^1(M, \di \mu_f)$.
\end{itemize}
Fix a nonlinearity $F(t)$ satisfying \eqref{assu_F}. Then, there exists $\delta>0$ such that if
\begin{equation}\label{labbrutta_general}
b(x) \ge -\delta \qquad \text{on }M,
\end{equation}
there exists a weak solution $u \in C^{1,\mu }_\loc(M)$ of
\begin{equation}\label{011}
\left\{\begin{array}{l}
\disp \Delta_{p,f} u + a(x) u^{p-1} - b(x)F(u) = 0 \qquad \text{ on M}\\[0.2cm]
0 < u \le \|u\|_{L^\infty(M)} < +\infty.
\end{array}\right.
\end{equation}
If we replace $ii)$ and $iii)$ by the stronger condition 
$$
iv)\qquad b_+(x) \asymp a(x) \qquad \text{as } x \text{ diverges,}
$$
and we keep the validity of \eqref{labbrutta_general}, then $u$ can also be chosen to satisfy 
\begin{equation}\label{012}
\inf_Mu >0.
\end{equation}
\end{Theorem}
%
%
%\begin{Remark}
%\emph{Requirement $(iv)$ is sharp to ensure \eqref{012}, consult Remark \ref{rem_sharplowerpositive} below.
%}
%\end{Remark}

In the next theorem, we remove requirement \eqref{labbrutta_general}; see also Remark \ref{rem_labrutta} for a related discussion.
\begin{Theorem}\label{teouno_bis} 
Let $M^m$ be a Riemannian manifold, $f \in C^\infty(M)$ and $p \in (1,+\infty)$. Suppose that $Q_0$ is subcritical on $M$ and let $a \in L^\infty_\loc(M)$ be such that $Q_a$ is subcritical on $M$. Consider $b \in L^\infty_\loc(M)$, and assume
\begin{itemize}
\item[$i)$] $b_-(x)$ has compact support;
\vspace{0.1cm}
\item[$ii')$]  $a(x) \le 0$ outside a compact set;
\vspace{0.1cm}
\item[$iii')$]  $a(x),b(x) 	\in L^1(M,\di \mu_f)$. 
\end{itemize}
Fix a nonlinearity $F(t)$ satisfying \eqref{assu_F}. Then, there exists a sequence $\{u_k\} \subset C^{1,\mu }_\loc(M)$ of distinct weak solutions of
\begin{equation}\label{011primo}
\left\{\begin{array}{l}
\disp \Delta_{p,f} u_k + a(x) u_k^{p-1} - b(x) F(u_k) = 0 \qquad \text{ on M}\\[0.2cm]
0 < u_k \le \|u_k\|_{L^\infty(M)} < +\infty,
\end{array}\right.
\end{equation}
such that $\|u_k\|_{L^\infty(M)} \ra 0$ as $k \ra +\infty$. If we replace $ii')$ and $iii')$ by the stronger condition 
$$
iv') \qquad a(x), b(x) \quad \text{have compact support,}
$$
then each $u_k$ also satisfies $\inf_M u_k >0$.
\end{Theorem}
One of the main features in the proof of Theorems \ref{teouno} and \ref{teouno_bis} above is a new flexible technique, which is based on a direct use of the non-negativity and subcriticality assumptions on $Q_a$ and $Q_0$. Consequently, all the geometric information needed on $M$ is encoded in the spectral behaviour of $Q_0$ and $Q_a$. For this reason, in Sections \ref{sec_criticality} and \ref{sec_examples} we concentrate on operators $Q_V$ to show that the assumption on $Q_0,Q_a$ in Theorems \ref{teouno} and \ref{teouno_bis} can be made explicit and easily verifiable in various relevant cases. \par
%More precisely, in Section and relate their subcriticality with a number of other conditions, notably the validity of %appropriate Hardy type inequalities and the positivity of the $Q_V$-capacity of relatively compact, open sets (see %Theorem \ref{teo_alternative} and Proposition \ref{prop_hyperbolicity} below).  \ref{sec_examples} 
%
%Our investigation complements some recent works in this direction (\cite{pinchovertintarev, pinchovertintarev2}),  %\par
%We underline that the possibility of develop a theory for the $Q_V$-capacity along the lines of that of the standard $p$-capacity in \cite{troyanov, pigolasettitroyanov, holokoskela, holopainen2} depends on a fundamental technical step that we will indicate as the ``pasting lemma". Its proof, which uses the solution to an obstacle problem in the present generality, is reported in the Appendix of the paper.\par
We now come to the strategy to prove Theorems \ref{teouno} and \ref{teouno_bis}. The lack of tools to produce solutions in the present generality forces us to proceed along very simple, general schemes. In particular, the argument can be roughly divided into three parts:
\begin{itemize}
\item[(1)] For a big relatively compact domain $\Omega$, we solve locally \eqref{04} with boundary condition 1 and $b(x)$ replaced by $b_+(x)$. Call $z_\Omega$ the solution. This is the easiest part, and is addressed in Lemma \ref{lem-diriproblem}.
\item[(2)] We find uniform $L^\infty$ estimates from below and above for $z_\Omega$, independent of $\Omega$. According to our geometric assumptions, these estimates can be on the whole $M$ or on a relatively compact set $\Lambda$. The proof of this step combines Lemmas \ref{lem_uniformP} and \ref{lem_linftyestimates}, and Proposition \ref{prop_lasolulimitata}.
\item[(3)] Making use of the results in Step (2), we ``place" $b_-$ in the Dirichlet problem for \eqref{04} on a domain $\Omega$ via an iterative procedure, to produce a local solution of \eqref{04} that possesses uniform upper and lower bounds. The desired global solution is then obtained by passing to the limit along an exhaustion $\{\Omega_j\}$. Note that this is the point where a distinction between Theorems \ref{teouno} and \ref{teouno_bis} appears.
\end{itemize}

Among the lemmas, which are of independent interest, we underline and briefly comment on the next uniform $L^\infty$-estimate, Lemma \ref{lem_linftyestimates}. This result is a cornerstone both for steps (2) and (3).  
\begin{Lemma}[\textbf{Uniform $L^\infty$-estimate}]\label{lem_linftyestimates_intro}
Let $M$ be a Riemannian manifold, $f \in C^\infty(M)$, $p \in (1,+\infty)$. Let $A,B \in L^\infty_\loc(M)$ with $B \ge 0$ a.e. on $M$. Assume that either
\begin{itemize}
\item[$(i)$] $B \equiv 0$ and $Q_A$ is subcritical, or
\vspace{0.1cm}
\item[$(ii)$] $B \not\equiv 0$ and $Q_A$ is non-negative.
\end{itemize}
Suppose that there exist a smooth, relatively compact open set $\Lambda \Subset M$ and a constant $c>0$ such that 
\begin{equation}\label{assunz_intro}
A \le cB \qquad \text{a.e. on } M \backslash \Lambda, 
\end{equation}
and fix a smooth, relatively compact open set $\Lambda'$ such that $\Lambda \Subset \Lambda'$, and a nonlinearity $F(t)$ satisfying \eqref{assu_F}.\\ 
Then, there exists a constant $C_\Lambda>0$ such that, for each smooth, relatively compact open set $\Omega$ with $\Lambda' \Subset \Omega$, the solution $0<z \in C^{1,\mu}(\overline \Omega)$ of
\begin{equation}\label{localYamabe_intro}
\left\{ \begin{array}{ll}
\Delta_{p,f} z + A(x)z^{p-1} - B(x)F(z) = 0 & \quad \text{on } \Omega, \\[0.2cm]
z = 1 & \quad \text{on } \partial \Omega.
\end{array}\right.
\end{equation} 
satisfies
\begin{equation}\label{eq245intro}
z \le C_\Lambda \qquad \text{on } \Omega.
\end{equation}
\end{Lemma}
The proof of the above estimate is accomplished by using non-negativity (resp. subcriticality) of $Q_A$ alone. As far as we know, the argument in the proof seems to be new and applicable beyond the present setting. Clearly, when $z$ is $C^2$ and $B \in C^0(M)$, $B>0$ on $M$, possibly evaluating \eqref{localYamabe_intro} at a interior maximum point $x_0$ we get
\begin{equation}\label{estitrivial}
\frac{F(z)}{z^{p-1}}(x_0) \le \sup_M \left(\frac{A_+}{B}\right),
\end{equation}
whence by \eqref{assu_F} $z(x_0)$ is uniformly bounded from above. Taking into account the boundary condition for $z$, in this case the uniform $L^\infty$ estimate is trivial with no assumption on $Q_A$. On the other hand, even a single point at which $B(x)=0$ makes this simple argument to fail, and actually Lemma \ref{lem_linftyestimates_intro} will be applied in cases when we have no control at all on the zero-set of $B$. Observe that \eqref{assunz_intro} is just assumed to hold outside of a compact set, hence $A$ is not required to be non-positive on the set where $B=0$. This suggests that the validity of \eqref{estitrivial} cannot be recovered ``in the limit" by using approximating positive functions $B_\eps$ for $B$ and related solutions $z_\eps$ for $z$. Note also that, when $B \equiv 0$,  Proposition \ref{pro115} below shows that $Q_A$ is necessarily non-negative, for otherwise $z$ might not exist for sufficiently large $\Omega$'s. Therefore, in the present generality at least the non-negativity of $Q_A$ on the whole $M$ needs to be assumed in any case.\par
%
%
%
%ù
%MA E' VERA UNA VERSIONE DEL MAIN THEOREM DOVE ALLA SUBCRITICALITA' DI $Q_a$ SI SOSTITUISCE LA PIU' RESTRINGENTE SUBCRITICALITA' DI $Q_{a_+}$, ma $b_+ \asymp a$ diventa $b_+ \asymp a_+$?? QUESTO MIGLIOREREBBE IL COROLLARIO SULLE VARIETA' CON SCALARE NON NEGATIVA, PERCHE' NON RICHIEDEREBBE CHE LA SCALARE ORIGINARIA SIA ZERO FUORI DA UN COMPATTO PER AVERE LA QUASI-ISOMETRIA.
%
%
%
%
%
We pause for a moment to comment on the subcriticality of $Q_V$. By its very definition, a sufficient condition for $Q_V$ to be subcritical is the coupling of the following two:
\begin{itemize}
\item[-] $Q_0$ is subcritical, thus there exists $w \in L^1_\loc(M)$, $w \ge 0$, $w \not\equiv 0$ such that
\begin{equation}\label{lahardysenzapot}
\int_M w(x)|\varphi|^p \di \mu_f \le \int_M |\nabla \varphi|^p \di \mu_f \qquad \forall \ \varphi \in \lip_c(M), \ \text{ and}
\end{equation}
\item[-] $V \le w$, $V \not \equiv w$.
\end{itemize}
Therefore, when $Q_0$ is subcritical, we can state simple, explicit conditions guaranteeing the subcriticality of $Q_V$ provided that we know explicit $w \in L^1_\loc(M)$, $w \ge 0$, $w \not \equiv 0$ satisfying \eqref{lahardysenzapot}. We define each of these $w$ a Hardy weight for $\Delta_{p,f}$.\par 
In the literature, there are conditions to imply the subcriticality of $Q_0$ that involve curvature bounds, volume growths, doubling properties and Sobolev type inequalities. For example, when $f\equiv 0$, in \cite{holokoskela} it is proved that a complete, non-compact manifold $M$ with non-negative Ricci curvature outside a compact set is not $p$-parabolic (i.e. $Q_0$ is subcritical) if and only if $p<m$. The interested reader can also consult \cite{holopainen2,rigolisettidue}. However, it seems challenging to obtain explicit Hardy weights in the setting of \cite{holokoskela,holopainen2,rigolisettidue}. Nevertheless, Hardy weights have been found in some interesting cases, starting with the famous Hardy type inequality for Euclidean space   
\begin{equation}\label{euclihardy}
\left(\frac{m-p}{p}\right)^p \int_{\R^m} \frac{|\varphi|^p}{r^p}\di x \le \int_{\R^m} |\nabla \varphi|^p\di x \qquad \forall \ \varphi \in \lip_c(\R^m),
\end{equation}
where $r(x) = |x|$ and $m>p$. In recent years (\cite{carron, liwang3, bmr2,   barbatisfilippastertikas, adimurthisekar, dambrosiodipierro, devyverfraaspinchover, devyverpinchover}) it has been observed how Hardy weights are related to positive Green kernels for $\Delta_{p,f}$. By exploiting the link established in Proposition \ref{prop_hyperbolicity} below, we will devote Section \ref{sec_examples} to produce explicit Hardy weights in various  geometrically relevant cases, see Theorems \ref{teo_laprimahardy}, \ref{multihardyCH}, \ref{teo_hardyminimal} below: in fact, a typical construction of Hardy weights via the Green kernel is compatible with comparison results for the Laplacian of the distance function, and thus Hardy weights can be transplanted from model manifolds to general manifolds, as observed in \cite{bmr2}, Theorem 4.15 and subsequent discussion. Moreover, the set of Hardy weights is convex in $L^1_\loc(M)$, thus via simple procedures one can produce new weights, such as multipole Hardy weights or weights blowing up along a fixed submanifold of $M$. Hardy weights can also be transplanted to submanifolds, but this procedure is more delicate and requires extra care. Let  $N^n$ be a Cartan-Hadamard manifold (i.e. a simply connected, complete manifold of non-positive sectional curvature), and let $M^m$ be a minimal submanifold of $N^n$. Suppose that the sectional curvature $\bar K$ of $N$ satisfies $\bar K \le -\kappa^2$, for some constant $\kappa \ge 0$. In \cite{carron,liwang3}, the authors proved the following Hardy type inequality: 
\begin{equation}\label{hardymineucl}
\left(\frac{m-2}{2}\right)^2\int_M \frac{\varphi^2}{\rho^2}\di x \le \int_M |\nabla \varphi|^2 \di x \qquad \forall \ \varphi \in \lip_c(M),
\end{equation}
$\rho(x)$ being the extrinsic distance in $N$ from a fixed origin $o \in N$. The Hardy weight in \eqref{hardymineucl} is sharp if $\kappa =0$ (in particular, if $N$ is the Euclidean space), but not if $\kappa>0$. Here, we will prove \eqref{hardymineucl} as a particular case of Theorem \ref{teo_hardyminimal} below, which also  strengthen \eqref{hardymineucl} to a sharp inequality when $\kappa>0$, in particular for minimal submanifolds of hyperbolic spaces. We stress that our Hardy weight for $\kappa>0$ is skew with the one found in \cite{liwang3}.\par
Using the Hardy inequalities mentioned before, we can rewrite the subcriticality assumption for $Q_0$ and $Q_V$ in Theorems \ref{teouno}, \ref{teouno_bis} in simple form for a wide class of manifolds; by a way of example, see Corollary \ref{cordue} in Section \ref{sec_applications}. We conclude by rephrasing Theorems \ref{teouno}, \ref{teouno_bis} in the setting of the generalized Yamabe problem. 
\begin{Theorem} \label{teodue}
Let $(M, \metric)$ be a non-parabolic Riemannian manifold of dimension $m \ge 3$ and scalar curvature $s(x)$. Suppose that the conformal Laplacian $L_{\metric}$ in \eqref{laplaconfo} is subcritical, and let $\widetilde s\in C^\infty(M)$. 
\begin{itemize}
\item[$(I)$] Assume that
\vspace{0.1cm}
\begin{itemize}
\item[$i)$] $\widetilde s_+$ has compact support;
\vspace{0.1cm}
\item[$ii)$] $s_-(x)=O\big(\widetilde s_-(x)\big)$ as $x$ diverges;
\vspace{0.1cm}
\item[$iii)$] for some $\theta>0$, $\big(\theta \widetilde s_-(x) - s_-(x)\big)_+ \in L^1(M)$.
\vspace{0.1cm}
\end{itemize}
Then, there exists $\delta>0$ such that if
\begin{equation}\label{labbrutta_generalYamabe}
\widetilde s(x) \le \delta,
\end{equation}
the metric $\langle \, , \, \rangle$ can be pointwise conformally deformed to a new metric $\widetilde{\langle \, , \, \rangle}$ with scalar curvature $\widetilde s(x)$ and satisfying
\begin{equation}\label{Cupper}
\widetilde{\metric} \le C \metric \qquad \text{on } M,
\end{equation}
for some constant $C>0$. Moreover, if $ii)$ and $iii)$ are replaced by the stronger
$$
iv) \qquad s(x) \asymp \widetilde s(x) \qquad \text{as $x$ diverges,} 
$$
then (under the validity of \eqref{labbrutta_generalYamabe}) there exists a pointwise conformal deformation $\widetilde{\langle \, , \, \rangle}$  of $\metric$ as above and satisfying 
\begin{equation}\label{doubleboundYamabe}
C_1 \metric \le \widetilde{\metric} \le C_2\metric \qquad \text{on } M,
\end{equation}
for some constants $0<C_1\le C_2$. In particular, $\widetilde \metric$ is non-parabolic, and it is complete whenever $\metric$ is complete.\par
\vspace{0.2cm}
\item[$(II)$] If $ii)$ and $iii)$ are replaced with
\vspace{0.1cm}
\begin{itemize}
\item[$ii')$] $s(x) \ge 0$ outside a compact set;
\vspace{0.1cm}
\item[$iii')$] $s(x), \widetilde s(x) \in L^1(M)$,
\vspace{0.1cm}
\end{itemize}
then the existence of the desired conformal deformation is guaranteed without the requirement \eqref{labbrutta_generalYamabe}, and moreover the constant $C$ in \eqref{Cupper} can be chosen as small as we wish (so that, indeed, there exist infinitely many conformal deformations realizing $\widetilde s$). If $ii')$ and $iii')$ are replaced with 
$$
iv') \qquad s(x),\widetilde s(x) \qquad \text{have compact support,} 
$$
each of these conformally deformed metrics $\widetilde \metric$ satisfies \eqref{doubleboundYamabe}.
\end{itemize}
\end{Theorem}
%
%
%
%\begin{Remark}
%\emph{Note that, in Theorem \ref{teodue}, the conformal Laplacian $L_{\metric}$ is required to be subcritical, while in Theorem \ref{teo_tipohyperb} we suppose that $L_{\metric}$ admits a positive Green function. Indeed, as we shall see, these two conditions are equivalent. See Remark \ref{rem_linearegreen}.
%}
%\end{Remark}
%
%
%
The paper is organized as follows. In Section \ref{sec_prelim} we collect some basic material on $Q_V$ and $Q_V'$. Section \ref{sec_criticality} will then be devoted to the criticality theory for $Q_V$, its link with Hardy weights and with a $Q_V$-capacity theory. In Section \ref{sec_examples}, we use comparison geometry to produce sharp Hardy inequalities. Section \ref{esistenza} contains the proof of Lemma \ref{lem_linftyestimates_intro} and of our main Theorems \ref{teouno}, \ref{teouno_bis}. Then, in Section \ref{sec_applications} we derive our geometric corollaries, and we place them among the existing literature. Finally, in the Appendix we give a full proof of the pasting lemma, an important technical result for the $Q_V$-capacity theory. Besides the presence of new results, a major concern of Sections \ref{sec_prelim} to \ref{sec_examples} is to help the reader to get familiar with various aspects of the theory of Schr\"odinger type operators $Q'_V$. For this reason, the experienced reader may possibly skip them and go directly to Section \ref{esistenza}.

\section{Preliminaries}\label{sec_prelim}
In this section we recall some general facts for the operators $\Delta_{p,f}, Q_V, Q'_V$ that are extensions, to a quasilinear setting, of some classical results of spectral theory (see \cite{murata, pinchovertintarev_JFA, mosspie, allegretto}). The interested reader may consult \cite{allegrettohuang, allegrettohuang2, garciamelian_sabinadelis, pinchovertintarev, pinchovertintarev2} for further information. 

\begin{Notation}
\emph{
Hereafter, given two open subsets $\Omega, U$, with $\Omega \Subset U$ we indicate that $\Omega$ has compact closure contained in $U$. We say that $\{\Omega_j\}$ is an exhaustion of $M$ if it is a sequence of relatively compact, connected open sets $\Omega_j$ with smooth boundary and such that $\Omega_j \Subset \Omega_{j+1} \Subset M$, $M = \bigcup_j \Omega_j$. The symbol $1_U$ denotes the characteristic function of a set $U$, and the symbol $\doteq$ is used to define an object. 
%Given a real function $f$, $f_+= \max\{f,0\}$ and $f_-= - \min\{f,0\}$ are its positive and negative parts, so that $f=f_+-f_-$.
}
\end{Notation}

\begin{Definition}
Let $\Omega \subset M$ be an open set and, for $V \in L^\infty_\loc(\Omega)$, let $Q_V,Q_V'$ be as in \eqref{06}, \eqref{07}. We say that $w \in W^{1,p}_\loc(\Omega)$ is a supersolution on $\Omega$ (respectively, subsolution, solution) if $Q_V'(w) \ge 0$ weakly on $\Omega$ (resp. $\le 0$, $=0$) that is, if
$$
\int_\Omega |\nabla w|^{p-2} \langle \nabla w, \nabla \varphi \rangle \di \mu_f  - \int_\Omega V|w|^{p-2}w\varphi  \ge 0 \qquad \text{(resp. } \le 0, \, = 0)
$$
for each non-negative $\varphi \in \lip_c(\Omega)$.
\end{Definition}
The basic technical material that is necessary for our purposes is summarized in the following

\begin{Theorem}\label{teo11}
Let $\Omega \Subset M$ be a relatively compact, open domain with $C^{1,\alpha}$ boundary for some $0<\alpha<1$. Let $f \in C^\infty (M)$, $p \in (1,+\infty)$, $V \in L^\infty_\loc (M)$ and define $Q_V,Q'_V$ as in \eqref{06}, \eqref{07}. Let $g \in L^\infty(\Omega)$, $\xi \in C^{1,\alpha}(\partial \Omega)$ and suppose that $u \in W^{1,p}(\Omega)$ is a solution of 
\begin{equation}\label{12}
\left\{ \begin{array}{ll}
Q'_V(u)= g & \quad \text{on } \Omega, \\[0.2cm]
u= \xi & \quad \text{on } \partial \Omega.
\end{array}\right.
\end{equation}
Then, 
\begin{itemize}
\item[$(1)$] \emph{[Boundedness]} $u \in L^\infty_\loc(\Omega)$, and for any relatively compact, open domains  $B \Subset B' \Subset \Omega$ there exists a positive constant $C=C(p,f,m,g,\xi, \Omega, \|u\|_{L^p(B',\di \mu_f)})$ such that 
$$
\|u\|_{L^\infty(B)} \le C. 
$$
If $\xi \in C^{2,\alpha}(\partial \Omega)$, $C$ can be chosen globally on $\Omega$, and thus $u \in L^\infty(\Omega)$.
\item[$(2)$] \emph{[$C^{1,\mu}$-regularity]} When $u \in L^\infty(\Omega)$, there exists $\mu \in (0,1)$ depending on $p,f,m,g,\alpha$ and on upper bounds for $\|u\|_{L^\infty}, \|g\|_{L^\infty}, \|\xi\|_{C^{1,\alpha}},\|V\|_{L^\infty}$ on $\Omega$ such that 
$$
\|u\|_{C^{1,\mu}(\overline\Omega)} \le C 
$$
for some constant $C$ depending on $\alpha, p$, the geometry of $\Omega$ and upper bounds for  $\|u\|_{L^\infty}$, $\|g\|_{L^\infty}, \|\xi\|_{C^{1,\alpha}}, \|V\|_{L^\infty}$ on $\Omega$.
\item[$(3)$] \emph{[Harnack inequality]}. For any relatively compact open sets $B \Subset B' \Subset \Omega$ there exists $C=C(f,p,m,B,B')>0$ such that, for each $u \ge 0$, $u \in W^{1,p} (\Omega)$ solution of $Q'_V (u) = 0$ on $\Omega$,
 \begin{equation} \label{13}
 \sup_B u \le C \inf_{B'} u.
 \end{equation}
In particular, either $u>0$ on $\Omega$ or $u\equiv 0$ on $\Omega$.
 \item[$(3a)$]\emph{[Half-Harnack inequalities]} For any  relatively compact, open sets $B \Subset B' \Subset \Omega$ the following holds:
 \begin{itemize}
 \item[] \emph{(Subsolutions)} for each $s>p-1$, there exists $C=C(f,p,m,B,B',V,s)>0$ such that for each $u \ge 0$ , $u \in W^{1,p} (\Omega)$ solution of $Q'_V (u) \le 0$ on $\Omega$ 
\begin{equation} \label{13b}
 \sup_B u \le C \|u\|_{L^s (B')};
 \end{equation} 
\item[] \emph{(Supersolutions)} for each 
$$
s \in \left(0,\frac{(p-1)m}{m-p}\right) \quad \text{if } p < m, \qquad s \in (0,+\infty) \quad \text{if } p \ge m,
$$ 
there exists $C=C(f,p,m,B,B',V,s)>0$ such that for each $u \ge 0$ , $u \in W^{1,p} (\Omega)$ solution of $Q'_V (u) \ge 0$ on $\Omega$
\begin{equation} \label{13c}
\|u\|_{L^s (B')} \le C \inf_B u.
 \end{equation} 
\end{itemize}
\item[$(4)$] \emph{[Hopf lemma]} Suppose that $\xi \ge 0, g \ge 0$ and let $u \in C^{1}(\overline \Omega)$ be a solution of \eqref{12} with $u \ge 0$, $u \not	\equiv 0$. If $x \in \partial \Omega$ is such that $u(x)=\xi(x)=0$, then, indicating with $\nu$ the inward unit normal  vector to $\partial \Omega$ at $x$ we have $\langle \nabla u, \nu\rangle(x)>0$.
\end{itemize}
\end{Theorem}

\begin{Remark} \label{rem14}
\emph{
\begin{itemize}
\item[(1)] The local boundedness of $u$ is a particular case of Serrin's theorem, see \cite{pucciserrin}, Theorem 7.1.1, and does not need the boundary condition. When $\xi \in C^{2,\alpha}(\partial \Omega)$, global boundedness can be reached via a reflection technique described at page 54 of \cite{garciamelian_sabinadelis}, see also \cite{tolksdorf2, barles}.
\item[(2)] is a global version, Theorem 1 of \cite{lieberman}, of a local regularity result in \cite{tolksdorf} and \cite{dibenedetto}. 
\item[(3)] is due to J. Serrin, see Theorem 7.2.1 in \cite{pucciserrin} for $p <m$, the discussion at the beginning of Section 7.4 therein for $p=m$, and Theorem 7.4.1 for $p>m$. 
\item[(3a)] The half-Harnack for subsolutions can be found in Theorem 7.1.1 of \cite{pucciserrin}, the one for supersolutions in the subsequent Theorems 7.1.2 (case $p<m$) and 7.4.1 (case $p >m$) of \cite{pucciserrin}. Again, see the discussion at the beginning of Section 7.4 of \cite{pucciserrin}. 
\item[(4)] The Hopf lemma can be found in Corollary 5.5 of \cite{puccirigoliserrin}.
\end{itemize} 
}
\end{Remark}

An important tool for our investigation is the following Lagrangian representation in \cite{pinchovertintarev}. 
\begin{Proposition}\label{prop_lagrangian}
For each $\varphi,g \in W^{1,p}_\loc(M)$, with $\varphi/g$ a.e. finite on $M$, the Lagrangian 
\begin{equation}\label{deflagrangian}
\mathcal{L}(\varphi,g) = |\nabla \varphi|^p + (p-1) \left(\frac{\varphi}{g}\right)^{p}|\nabla g|^p - p \left( \frac{\varphi}{g}\right)^{p-1} |\nabla g|^{p-2} \langle \nabla g, \nabla \varphi\rangle 
\end{equation}
satisfies $\mathcal{L}(\varphi,g) \ge 0$ on $M$, and $\mathcal{L}(\varphi,g)\equiv 0$ on some connected open set $U$ if and only if $\varphi$ is a constant multiple of $g$ on $U$.\\ 
Moreover, suppose that $g \in W^{1,p}_\loc(M)$ is a positive solution of $Q_V'(g)=0$ (resp. $Q_V'(g) \ge 0$) on $M$. Then, for $\varphi \in L^\infty_c(M) \cap W^{1,p}(M)$, $\varphi \ge 0$ it holds
\begin{equation}\label{lagrangian}
Q_V(\varphi) = \int_M \mathcal{L}(\varphi, g)\di \mu_f \qquad (\text{resp, } \ge ).
\end{equation}
\end{Proposition}
\begin{proof}
The non-negativity of $\mathcal{L}(\varphi,g)$ follows by applying Cauchy-Schwarz  and Young inequalities on the third addendum in \eqref{deflagrangian}, and analyzing the equality case, $\mathcal{L}(\varphi,g) \equiv 0$ if and only if $\varphi = cg$ on $M$, for some constant $c \in \R$.\\
We now prove the integral (in)equality in \eqref{lagrangian}. By Harnack inequality, $g$ is locally essentially bounded from below on $M$. This, combined with our regularity requirement on $\varphi$, guarantees that $\varphi^p/g^{p-1} \in W^{1,p}(M)$ and is compactly supported. Thus, we integrate on $M$ the pointwise identity 
$$
|\nabla \varphi|^p - |\nabla g|^{p-2} \langle \nabla g, \nabla \left(\frac{\varphi^p}{g^{p-1}}\right)\rangle  = \mathcal{L}(\varphi,g),
$$
and couple with the weak definition of $Q_V'(g)=0$ (resp. $\ge 0$) applied to the test function $\varphi^p/g^{p-1}$:
$$
0 = \int_{M} |\nabla g|^{p-2} \langle \nabla g, \nabla \left(\frac{\varphi^p}{g^{p-1}}\right)\rangle \di \mu_f - \int_M Vg^{p-1}\left(\frac{\varphi^p}{g^{p-1}}\right)\di \mu_f \qquad \text{(resp. }  \le \text{)}
$$
to deduce \eqref{lagrangian}.
\end{proof}

Rewriting the expression of $\mathcal{L}(\varphi,g)$ we deduce the next useful Picone type inequality due to  \cite{allegrettohuang2,diazsaa,anane}.
% see Proposition 1 in \cite{anane}. 
%
%
\begin{Proposition}\label{anane}
Let $M$ be a Riemannian manifold, and let $\Omega \Subset M$ be a relatively compact, connected open set. Then, the functional
\begin{equation}\label{111'}
I(w,z)= \int_\Omega |\nabla w|^{p-2} \left\langle \nabla w, \nabla \frac{w^p-z^p}{w^{p-1}}\right\rangle \di \mu_f-\int_\Omega |\nabla z|^{p-2} \left\langle \nabla z, \nabla \frac{w^p-z^p}{z^{p-1}}\right\rangle \di \mu_f
\end{equation}
is non-negative on the set 
$$
{\cal D}_\Omega = \left\{ (w,z) \in W^{1,p}(\Omega)\times W^{1,p}(\Omega)  \ :\ w,z \ge 0 \ {\rm on} \ \Omega \ , \ \frac{w}{z}, \frac{z}{w} \in L^\infty (\Omega) \right\}.
$$
Furthermore, $I(w,z)=0$ if and only if $w=Cz$ on $\Omega$, for some constant $C>0$. 
\end{Proposition}

\begin{proof}
Since $w, z \in \mathcal{D}_\Omega$, it is easy to see that $\frac{w^p}{z^{p-1}}, \frac{z^p}{w^{p-1}} \in W^{1,p}(\Omega)$. We can therefore expand the integrand in \eqref{111'} and rearrange to deduce that
$$
I(w,z) = \int_\Omega \big[\mathcal{L}(w,z) + \mathcal{L}(z,w)\big] \di \mu_f,
$$
with $\mathcal{L}$ as in \eqref{deflagrangian}. The first part of previous proposition then gives the desired inequality.
\end{proof}
Now, we investigate property $Q_V \ge 0$ and its consequences. By its very definition, $Q_V \ge 0$ on an open set $\Omega\subset M$ is equivalent to the non-negativity of the fundamental tone
\begin{equation}\label{010}
\lambda_{V}(\Omega) = \inf_{0 \not \equiv \varphi \in \lip_c(\Omega)} \frac{pQ_V(\varphi)}{\|\varphi\|^p_{L^p(\Omega, \di \mu_f)}}.
\end{equation}
If $\Omega$ is a relatively compact domain with smooth boundary, then it is well-known that the infimum \eqref{010} is attained by a first eigenfunction $\phi \not\equiv 0$ solving Euler-Lagrange equation
$$
\left\{
\begin{array}{l}
\disp Q_V'(\phi) = \lambda_V(\Omega)|\phi|^{p-2}\phi \qquad \text{on } \, \Omega, \\[0.2cm]
\disp \phi=0 \qquad \text{on } \, \partial \Omega,
\end{array}
\right.
$$
and $\phi>0$ on $\Omega$ up to changing its sign \footnote{Briefly, $|\phi|$ still minimizes the Rayleigh quotient in \eqref{010}, thus it satisfies the Euler-Lagrange equation $Q_V'(|\phi|) = \lambda_V(\Omega) |\phi|^{p-1}$, hence $|\phi|>0$ on $\Omega$ by Harnack inequality in Theorem \ref{teo11}, $(3)$.}.
Furthermore, by Harnack inequality, if $\Omega \subset \Omega'$ are two relatively compact open sets and $\Omega'\backslash \Omega$ has non-empty interior, then $\lambda_V(\Omega) > \lambda_V(\Omega')$.\par
% (the reader is also suggested to see \cite{lindqvist}).\
%
The next comparison result will be used throughout the paper, and improves on Theorem 5 of \cite{garciamelian_sabinadelis}.
\begin{Proposition}\label{prop_compagen}
Let $M^m, p, f$ be as above and, for $A \in L^\infty_\loc(M)$, define $Q_A, Q_A'$ as in \eqref{06}, \eqref{07} with $V(x)=A(x)$. Consider a relatively compact, open set $\Omega\Subset M$ with smooth boundary, and let $u_1,u_2 \in C^{1,\mu}(\overline\Omega)$, for some $\mu \in (0,1)$. Furthermore, suppose that, for some non-negative $B \in L^\infty(\Omega)$ and a nonlinearity $F(t)$ satisfying \eqref{assu_F}, 
\begin{equation}\label{16}
\left\{ \begin{array}{l}
\Delta_{p,f}u_1 + A|u_1|^{p-2}u_1 - BF(u_1) \ge 0, \\[0,2cm]
\Delta_{p,f}u_2 + A|u_2|^{p-2}u_2 - BF(u_2) \le 0, \\[0,2cm]
u_1 \le u_2 \ \text{ on } \partial \Omega, \qquad u_1 \ge 0, \, u_2 >0 \ \text{ on } \Omega.
\end{array}\right.
\end{equation}
Then, either 
\begin{itemize}
\item[$i)$] $u_1 \le u_2$ on $\Omega$, or
\item[$ii)$] $B(x) \equiv 0$ on $\Omega$, $u_2$ satisfies $Q_A'(u_2)=0$, $u_2 \equiv 0$ on $\partial \Omega$, and $\lambda_A(\Omega) = 0$. 
\end{itemize}
\end{Proposition}
\begin{proof}
%
%We define 
%\begin{equation}\label{17}
%-g= \Delta_{p,f}u_2 + a(x)u_2^{p-1} - b(x) u_2^\sigma. 
%\end{equation}
%Thus, in our assumptions, $g \in L^\infty(\Omega)$ and by \eqref{16} 
%$$
%g \ge 0
%$$. 
We let $\disp\xi={u_2}_{|\partial \Omega} \in C^{1,\mu } (\partial \Omega)$ and let 
$$
V =A(x)-B(x)\frac{F(u_2)}{u_2^{p-1}}. 
$$
Note that $V \in L^\infty(\Omega)$ by \eqref{assu_F}. For $x \in \partial \Omega$ such that $u_2(x)=0$, we let $\nu$ be the inward unit normal to $\partial\Omega$ at $x$. Then, applying Theorem \ref{teo11} $(4)$ we deduce that
$$
\langle \nabla u_2(x), \nu \rangle(x) >0,
$$
by continuity, there thus exists a constant $C>0$ such that 
% 999
%By continuity, the above is true on $\partial \Omega \cap B_\epsilon (x)$ for a ball $B_\epsilon(x)$ with radius $\epsilon >0$ sufficiently small. Thus we can find $C_\epsilon >0$ such that $u_1 \le C_\epsilon u_2$ on $\partial\Omega \cap B_\epsilon (x)$. The same is of course true if $u_2 (x) >0$. Since $\partial\Omega$ is compact, we can choose $C>0$ such that
\begin{equation}\label{18}
u_1 \le C u_2 \qquad \text{on } T(\partial\Omega),
\end{equation}
for some tubular neighbourhood $T(\partial \Omega)$ of $\partial \Omega$. Using assumption $u_2>0$ on $\Omega$ we can suppose that \eqref{18} is true on all of $\overline{\Omega}$ with $C>1$. Because of \eqref{16} and since $B(x) \ge 0$, $C>1$ and, by \eqref{assu_F}, $F(t)/t^{p-1}$ is increasing on $\R^+$, $Cu_2$ is still a supersolution: 
$$
\begin{array}{l}
\disp \Delta_{p,f}(Cu_2) + A(Cu_2)^{p-1} - BF(Cu_2) = C^{p-1} \big[\Delta_{p,f} u_2 +Au_2^{p-1}\big] - BF(Cu_2) \\[0.2cm]
\disp \le C^{p-1} BF(u_2) - BF(Cu_2) \le B(Cu_2)^{p-1} \left[ \frac{F(u_2)}{u_2^{p-1}} - \frac{F(Cu_2)}{(Cu_2)^{p-1}}\right] \le 0. 
\end{array} 
$$
Using $u_1 \ge 0$ as a subsolution, by \eqref{18} and applying the method of sub- and supersolutions, see Theorem 4.14, page 272, in \cite{diaz}, we find a solution $v$ of 
\begin{equation}\label{19}
\left\{ \begin{array}{l}
\Delta_{p,f}v + A|v|^{p-2}v - BF(v) = 0, \\[0,2cm]
v = u_2 \qquad \text{on } \partial \Omega, 
\end{array}\right.
\end{equation}
satisfying
\begin{equation}\label{110} 
 u_1 \le v \le Cu_2 \qquad {\rm on} \ \Omega
 \end{equation}
By the $C^{1,\mu}$-regularity of Theorem \ref{teo11}, $v \in C^{1, \alpha}(\overline \Omega)$ for some $\alpha \in (0,1)$. If we show that $v \le u_2$ then \eqref{110} implies $u_1 \le u_2$ on $\Omega$ which is the conclusion $i)$ of the Proposition. \par 
Suppose that this is not the case, that is, assume that the open set $U= \{ v > u_2\}$ is non-empty. We are going to prove that $ii)$ holds. Since $v \ge u_1 \ge 0$ and $v$ is positive on $U$, then $v >0$ on $\Omega$ as a consequence of the Harnack inequality in Theorem \ref{teo11} (use $V  = A(x) -B(x)F(v)/v^{p-1}$, which by \eqref{assu_F} is bounded on $\Omega$). Alternatively, one can use the version of the strong maximum principle in Theorem 5.4.1 in \cite{pucciserrin}.
%to deduce 
%
%
%
%\par 
%Indeed, by contradiction let $v(y)=0$ for some $y \in \Omega$. 
%%Note that, since $-g \le 0$, from \eqref{19} $v$ satisfies
%%\begin{equation}\label{111}
%%\Delta_{p,f} v + a(x) v^{p-1}-b(x) v^\sigma \le 0 \qquad {\rm on} \ \Omega.
%%\end{equation}
%Choose a sufficiently small ball $B_\epsilon (y) \subset \Omega$, and rewriting \eqref{19} in the form
%$$
%{\rm div} \left({\rm e}^{-f} |\nabla v|^{p-2} \nabla v \right)+ \left(a(x) v^{p-1}-b(x) v^{\sigma} \right){\rm e}^{-f} = 0 \qquad {\rm on} \ B_\epsilon (y)
%$$
%setting
%$$
%\lambda=\inf_{B_\epsilon (y)} {\rm e}^{-f} \quad,\quad \Lambda=\sup_{B_\epsilon (y)} {\rm e}^{-f}
%$$
%for $\epsilon>0$ sufficiently small, $\disp \frac{\Lambda}{\lambda} < (3+2\sqrt{2})^2$; furthermore for
%$$
%B(x,v,\nabla v)= \left(a(x)v^{p-1}-b(x)v^\sigma \right){\rm e}^{-f(x)}
%$$
%we have
%\begin{equation}\label{B1}\tag{B1}
%B(x,v,\nabla v) \ge - \left( \alpha v^{p-1}+\beta v^\sigma \right)\Lambda \qquad {\rm on} \ B_\epsilon (y)
%\end{equation}
%for some constants $\alpha, \beta >0$. Thus
%$$
%h(v)= v^{p-1} \left(\alpha+\beta v^{\sigma-p+1} \right)
%$$
%satisfies $h(0)=0$ and is non decreasing on some interval $(0,\delta)$ for $\delta>0$ sufficiently small. We now apply the strong maximum principle, Theorem 5.4.1 page 111 in \cite{pucciserrin}, to deduce that $v \equiv 0$ on $B_\epsilon (y)$. A connectedness argument shows that $v \equiv 0$ on $\Omega$; contradiction. \par 
%Thus $v>0$ on $\Omega$. 
Now, again by the Hopf Lemma of Theorem \ref{teo11}, $\langle  \nabla v(x) , \nu(x) \rangle >0$, $\nu$ the inward unit normal to $\partial\Omega$ at $x$, at each point $x \in \partial \Omega$ where $v(x) = u_2(x)=0$. Hence, the ratio $u_2/v$ is well defined at $x$ along the half line determined by $\nu$. This shows that $u_2/v$ and similarly $v/u_2$ are in $L^\infty (\Omega)$. Applying Proposition \ref{anane} on $U$ we deduce $I(u_2,v) \ge 0$, and $I(u_2,v)=0$ if and only if $u_2$ and $v$ are proportional on $U$.
%Next, on
%$$
%{\cal D} (I)= \left\{ (w,z) \in W^{1,p}(\Omega) \times W^{1,p}(\Omega) \quad:\quad w,z \ge 0 \ {\rm on} \ \Omega \ , \ \frac{w}{z}, \frac{z}{w} \in L^\infty (\Omega) \ , \ w=z \quad {\rm on} \ \partial\Omega \right\}
%$$
%we consider the functional
%\begin{equation}\label{111'}
%I(w,z)= \int_\Omega |\nabla w|^{p-2} \left\langle \nabla w, \nabla \frac{w^p-z^p}{w^{p-1}}\right\rangle \di \mu_f-\int_\Omega |\nabla z|^{p-2} \left\langle \nabla z, \nabla \frac{w^p-z^p}{z^{p-1}}\right\rangle \di \mu_f.
%\end{equation}
%By a result of \cite{diazsaa} and \cite{anane} (see Proposition 1 in this latter), we have that $\forall \ (w,z) \in D(I)$,
%\begin{equation}\label{piconeid}
%I(w,z) \ge 0
%\end{equation}
%equality holding if and only if $w=Cz$ for some constant $C>0$. 
However, the positivity of the test function $(v^p - u_2^p)/u_2^{p-1} >0$ on $U$ implies, by \eqref{16} and \eqref{19}, that
$$
0  \le  \disp I(u_2, v) \le  
  \disp - \int_UB\big(u_2^p -v^p\big)\left(\frac{F(u_2)}{u_2^{p-1}} - \frac{F(v)}{v^{p-1}}\right)\di \mu_f.
$$
Being $F(t)/t^{p-1}$ strictly increasing on $\R^+$ and $B \ge 0$, we deduce 
\begin{equation}\label{112}
B\big(u_2^p -v^p\big)\left(\frac{F(u_2)}{u_2^{p-1}} - \frac{F(v)}{v^{p-1}}\right) \ge 0
\end{equation}
on $U$, whence $I(u_2,v) = 0$. We therefore conclude that $u_2=cv$ on $U$, for some constant $c$ which, because of the definition of $U$, satisfies $c>1$. Using that $v=u_2$ on $\partial U$, we necessarily have $v=u_2=0$ on $\partial U$, hence $U \equiv \Omega$. Substituting $u_2=cv$ on $\Omega$ into \eqref{112} we deduce
$$
B \left( c^p-1 \right)v^{p-1}\left(\frac{F(cv)}{(cv)^{p-1}} - \frac{F(v)}{v^{p-1}}\right) \equiv 0.
$$
Since $v>0$ on $\Omega$ and $F(t)/t^{p-1}$ is strictly increasing, $B \equiv 0$ and, from \eqref{19}, $v>0$ solves
$$
\left\{
\begin{array}{l}
\disp \Delta_{p,f} v + A(x) |v|^{p-2} v = 0 \qquad \text{on } \, \Omega \\[0.2cm]
\disp v = 0 \qquad \text{on } \, \partial\Omega.
\end{array}
\right.
$$
Consequently, $0$ admits a positive eigenfunction of $Q_A'$. By a result in \cite{anane}, $\lambda_A(\Omega)=0$, showing the validity of $ii)$.
\end{proof}

\begin{Remark}
\emph{We underline that, in the above proposition, the non-negativity of $Q_A$ is not required. However, if $B \equiv 0$, $u_2$ turns out to be a positive solution of $Q_A(u) \ge 0$, and using Proposition \ref{pro115} below we automatically have $\lambda_A(\Omega) \ge 0$.
}
\end{Remark}
In what follows we shall frequently use the next formula: for $I \subset \R$ and $\alpha \in C^2 (I)$ with $\alpha' > 0$ on $I$, and for $u \in C^0(M) \cap W^{1,p}_\loc (M)$ with $u(M) \subset I$ we have, weakly on $M$,
\begin{equation}\label{113}
\Delta_{p,f} \alpha (u)= \alpha' (u) \left| \alpha' (u) \right|^{p-2} \Delta_{p,f} u+ (p-1) \alpha'' (u) \left| \alpha' (u) \right|^{p-2} |\nabla u|^p.
\end{equation}
A second ingredient is the following existence result that goes under the name of the Allegretto-Piepenbrink theorem, see \cite{allegrettohuang, allegrettohuang2, garciamelian_sabinadelis}. We include a proof of the next slightly more general version, for the sake of completeness. 
\begin{Proposition}\label{pro115}
Let $\riem$ be a non-compact Riemannian manifold, $f \in C^\infty (M)$, $p \in (1,+\infty)$ and, for $V \in L^\infty_\loc(M)$, set $Q'_V,Q_V$ as in \eqref{07} and \eqref{06}. Then, the following statements are equivalent:
\begin{itemize}
\item[$i)$] There exists $w \in C^0(M) \cap W^{1,p}_\loc(M)$, $w > 0$ weak solution of
\begin{equation}\label{116}
Q'_V(w) \ge 0 \qquad \text{on } \, M;
\end{equation}
\item[$ii)$] There exists $u \in C^{1,\mu}_\loc(M)$, $u>0$ weak solution of
\begin{equation}\label{117}
Q'_V(u) = 0 \qquad \text{on } \, M;
\end{equation}
\item[$iii)$] $Q_V \ge 0$ on $M$.
\item[$iv)$] For each relatively compact domain $\Omega \Subset M$ with $C^{1,\alpha}$ boundary for some $\alpha \in (0,1)$, and for each $\xi \in C^{1,\alpha}(\partial \Omega)$, $\xi \ge 0$, there exists a unique solution $\varphi \in C^{1, \mu }(\overline \Omega)$ of 
\begin{equation}\label{118}
\left\{\begin{array}{ll}
Q_V'(\varphi) = 0 & \quad \text{on } \, \Omega, \\[0.2cm]
\varphi = \xi & \quad \text{on } \, \partial \Omega
\end{array}\right.
\end{equation}
satisfying $\varphi \ge 0$ on $\Omega$. Moreover, if $\xi \not \equiv 0$, then $\varphi>0$ on $\Omega$.
\end{itemize}
\end{Proposition}

\begin{proof}
The scheme of proof is $ii) \Rightarrow i) \Rightarrow iii) \Rightarrow iv) \Rightarrow ii);$ Note that the first implication is trivial. \par 
$i) \Rightarrow iii)$. It follows immediately from Proposition \ref{prop_lagrangian} and the non-negativity of $\mathcal{L}$. \par
%
%Let $w >0$ be as in $i)$ and define  $h=\log w$. Then, using \eqref{113}, we have
%$$
%\Delta_{p,f} h = \disp \frac{\Delta_{p,f} w}{w^{p-1}} - (p-1) \frac{|\nabla w|^{p}}{w^p}  \le  - V  - (p-1)|\nabla h|^p.
%$$
%Let $\varphi \in \lip_c(M)$; multiplying the above by $|\varphi|^p$ and integrating   we have
%$$
%\begin{array}{l}
%\disp \int_M V |\varphi|^p\di \mu_f + (p-1) \int_M \frac{|\nabla w|^{p}}{w^p}|\varphi|^p \di \mu_f \\[0.4cm]
%\disp \qquad \qquad \qquad \qquad \le - \int_M  (\Delta_{p,f} h) |\varphi|^p \di \mu_f =  \int_M p |\varphi|^{p-1} |\nabla h|^{p-2} \langle \nabla h,\nabla |\varphi|\rangle \di \mu_f\\[0.4cm]
%\disp \qquad \qquad \qquad \qquad = \int_M p |\varphi|^{p-1} \frac{|\nabla w|^{p-2}}{w^{p-1}} \langle \nabla w , \nabla |\varphi| \rangle \di \mu_f \\[0.4cm] 
%
%\disp \qquad \qquad \qquad \qquad  \le p\int_M  |\varphi|^{p-1} \frac{|\nabla w|^{p-1}}{w^{p-1}} \big|\nabla |\varphi|\big|\di \mu_f \le p \int_M |\varphi|^{p-1} \frac{|\nabla w|^{p-1}}{w^{p-1}} |\nabla \varphi| \di \mu_f.\\[0.4cm]
%\end{array}
%$$
%Using Young's inequality we then obtain
%$$
%\disp \int_M V |\varphi|^p\di \mu_f + (p-1) \int_M \frac{|\nabla w|^{p}}{w^p}|\varphi|^p \di \mu_f \le   \int_M |\nabla \varphi|^p\di \mu_f + (p-1) \int_M \frac{|\nabla w|^p}{w^p} |\varphi|^p\di\mu_f
%$$
%and simplifying
%$$
%\int_M V  |\varphi|^p \di \mu_f \le \int_M |\nabla \varphi|^p \di \mu_f.
%$$
%This shows that $Q_V$ is non negative on $M$. \par
%
$iii) \Rightarrow iv)$. By assumption $\lambda_V (M) \ge 0$; it follows that, for $\Omega$ as in $iv)$, by the monotonicity property for eigenvalues $\lambda_V(\Omega)>0$. Hence, the variational problem associated to \eqref{118} is coercive and sequentially weakly lower-semicontinuous (see also Theorem \ref{A5} in Appendix). Therefore \eqref{118} admits a weak solution $\varphi\in W^{1,p}(\Omega)$. By the $C^{1,\mu}$-regularity of Theorem \ref{teo11} we have that $\varphi \in C^{1,\mu}(\overline\Omega)$ for some $\mu \in (0,1)$. Moreover, by the local Harnack inequality of item $(3)$, $\varphi>0$ on $\overline \Omega$ whenever $\phi \ge 0$  on $\Omega$, unless $\varphi \equiv 0$. \par 
By contradiction  suppose that $\varphi$ is somewhere negative in $\Omega$. Since $\xi \ge 0$, on $\partial\Omega$, $\varphi_- = -\min\{\varphi,0\} \in \lip_0(\Omega)$ and it is thus an admissible test function for \eqref{118} on $\Omega$. We have
$$
0 = \disp Q_V'(\varphi)[-\varphi_-] \doteq -\int_\Omega \left\{ |\nabla \varphi|^{p-2} \langle \nabla \varphi , \nabla \varphi_- \rangle-V  |\varphi|^{p-2} \varphi \varphi_- \right\} \di \mu_f \equiv p Q_V (\varphi_-)
$$
and therefore $\lambda_V(\Omega) \le 0$, a contradiction.\par
$iv)\Rightarrow ii)$. Choose an exhaustion $\{\Omega_j\}$ of $M$. Let $u_j>0$, $u_j \in C^{1,\mu_j}(\overline{\Omega}_j) \subset W^{1,p} (\Omega_j)$ be a solution of
\begin{equation}\label{119}
\left\{ \begin{array}{l}
Q_V'(u_j) =0 \qquad \text{on } \Omega_j \\[0.2cm]
u_j=1 \qquad \text{on } \partial \Omega_j.
\end{array}\right.
\end{equation}
Fix $x_0 \in \Omega_1$ and rescale $u_j$ in such a way that $u_j(x_0)=1$ for every $j$. By Theorem \ref{teo11} $3)$, $\{u_j\}$ is uniformly locally bounded in $\Omega$, thus by Theorem \ref{teo11} $2)$ $\{u_j\}$ is uniformly locally bounded in $C^{1, \mu} (\Omega)$. It follows that $\{u_j\}$ has a subsequence converging weakly and pointwise to a weak solution $u \in C^{1,\mu}_\loc(M)$ of
$$
\left\{ \begin{array}{l}
Q_V'(u) =0 \qquad \text{on } \ M \\[0.2cm]
u(x_0)=1.
\end{array}\right.
$$
Since $u \ge 0$ and $u \not \equiv 0$, again by $3)$ of Theorem \ref{teo11} we deduce $u>0$ on $M$. This shows the validity of $ii)$.
\end{proof}
Next, we need a gluing result which we will call the pasting lemma. Although for $V  \equiv 0$ this is somehow standard (a simple proof can be given by adapting Lemma 2.4 in \cite{pinchovertertikastintarev}), the presence of a nonzero $V$ makes things more delicate. First, we introduce some definitions. We recall that, given an open subset $\Omega \subset M$ possibly with non-compact closure, the space $W^{1,p}_\loc(\overline{\Omega})$ is the set of all functions $u$ on $\Omega$ such that, for every relatively compact open set $U \Subset M$ with $U \cap \Omega \not= \emptyset$, $u \in W^{1,p}(\Omega \cap U)$. A function $u$ in this space is thus well behaved on relatively compact portions of $\partial \Omega$, while no global control is assumed on the $W^{1,p}$ norm of $u$. 
\begin{Lemma}[The pasting lemma]\label{lem_pasting} Let $(M, \metric)$ be a Riemannian manifold, $f \in C^\infty (M)$, $p \in (1,\infty)$, $V \in L^\infty_\loc (M)$. Let $\Omega_1$, $\Omega_2$ be open sets such that $\Omega_1 \subset \Omega_2$. For $j=1,2$, let $u_j \in C^0(\overline\Omega_j) \cap W^{1,p}_\loc(\overline\Omega_j)$ be a positive supersolution of $Q_V'$ on $\Omega_j$, that is, $Q_V' (u_j) \ge 0$ on $\Omega_j$. If 
\begin{equation}\label{u2leu1}
u_2 \le u_1 \qquad \text{on } \partial \Omega_1 \cap \Omega_2, 
\end{equation}
then the positive function
\begin{equation}\label{minimum_pasting}
u \doteq \left\{\begin{array}{ll}
\disp \min \{u_1,u_2\} &{\rm on} \ \overline\Omega_1 \\[0.4cm]
\disp u_2 &{\rm on} \ \Omega_2 \backslash \Omega_1
\end{array}
\right.
\end{equation}
is in $W^{1,p}_\loc(\overline\Omega_2)$ and it satisfies $Q'_V (u) \ge 0$ on $\Omega_2$.
\end{Lemma}
When $\Omega_1 \equiv \Omega_2$, a general result of V.K. Le, \cite{Le}, guarantees that $\min\{u_1,u_2\}$ is a supersolution. The pasting lemma can then be deduced by an approximation argument, along the lines described in \cite{albamaririgoli}, and we leave the details to the interested reader. In the Appendix below, we give a quite different proof by using the obstacle problem for $Q_V$ and the minimizing properties of its solutions, that might have an independent interest.

%\begin{remark}
%\emph{When $V \equiv 0$ and $u_2$ is a positive constant $c>0$, a simpler proof can be given by adapting Lemma 2.4 in \cite{pinchovertertikastintarev}. We sketch here their argument: for a given $0 \le \varphi \in \lip_c(M)$ and $\eps>0$, set
%$$
%v_\eps = \sqrt{(u_1-c)^2+\eps}, \qquad \psi_\eps = \frac{v_\eps - (u_1-c)}{2v_\eps} \ge 0.
%$$
%Plugging $\psi_\eps$ in the weak definition of $\Delta_{p,f}u_1 \le 0$ and letting $\eps \ra 0$ we get
%$$
%\int_{\{u<c\}} |\nabla u|^{p-2} \langle \nabla u, \nabla \varphi \rangle \ge 0.
%$$
%This last expression is the weak definition 
%}
%\end{remark}
%
%
% 
%\begin{Lemma}\label{lem_pasting}{\bf(The pasting Lemma)}. Let $(M, \metric)$ be a Riemannian manifold, $f \in C^\infty (M)$, $p \in (1,\infty)$, $V \in L^\infty_\loc (M)$. Let $\Omega_1$, $\Omega_2$ be relatively compact open sets with smooth boundaries, $\Omega_1 \Subset \Omega_2$ and with $\lambda_V(\Omega_1) >0$. Let $u_j \in W^{1,p} (\Omega_j)$, $j=1,2$ be positive solutions of $Q_V' (u_j) \ge 0$ on $\Omega_j$. If $\disp (u_2-u_1)_+ \in W_0^{1,p} (\Omega_1)$, then the positive function
%\begin{equation}
%\left\{\begin{array}{ll}
%\disp \min \{u_1,u_2\} &{\rm on} \ \overline{\Omega}_1 \\[0.4cm]
%\disp u_2 &{\rm on} \ \Omega_2 \backslash \Omega_1
%\end{array}
%\right.
%\end{equation}
%is in $W^{1,p}(\Omega_2)$ and it satisfies $Q'_V (u) \ge 0$.
%\end{Lemma}
%
%
%
\section{Criticality theory for $Q_V$, capacity and Hardy weights}\label{sec_criticality}
The criticality theory for $Q_V$ reveals an interesting scenario, and extends in a nontrivial way the parabolicity theory for the standard Laplacian and for the $p$-Laplacian (developed, among others, in \cite{grigoryan, pigolasettitroyanov, troyanov}). Although a thorough description goes beyond the scope of this paper, nevertheless the validity of the pasting lemma gives us the opportunity to complement known results (especially those in \cite{pinchovertintarev, pinchovertintarev2}) by relating them to a capacity theory for $Q_V$, see Theorem \ref{teo_alternative} below. We underline that, although the $Q_V$-capacity theory is investigated by following the same lines as those for the standard $p$-Laplacian, as in the previous results the presence of a nontrivial $V$ makes things subtler.\\
\par
Let $Q_V \ge 0$ on $M$, and fix a positive \emph{supersolution} $g \in C^0(M) \cap W^{1,p}_\loc(M)$ of $Q'_V$, that is, a solution of 
\begin{equation}\label{ipotesigcapac}
Q_V'(g) \ge 0.
\end{equation}
For each $K \Subset \Omega \Subset M$, $K$ compact, $\Omega$ open, let
$$
\Dd(K, \Omega, g)  = \Big\{ \varphi \in C^0(\overline \Omega) \cap \wupz \ : \ \varphi \ge g \ \text{ in a neighbourhood of } K\Big\}
$$
and define the $Q_V$\textbf{-capacity}
$$
\disp \capac_{Q_V}(K, \Omega, g) \doteq \disp  \inf_{\varphi \in \Dd(K, \Omega, g) } Q_V(\varphi) 
$$
Clearly, $\capac_{Q_V}(K, \Omega, g) $ grows if we decrease $\Omega$, as well as if we increase $K$. If $V\equiv 0$, it is customary to choose $g \equiv 1$ as solution of $\Delta_{p,f}g=0$, and we recover the classical definition of capacity. We however underline that, for the next arguments to work, it is essential that the fixed $g$ solves $Q_V'(g)  \ge 0$ on $M$, for otherwise the basic properties needed in the next results could not hold.  
\begin{Proposition}\label{prop_capacattingida}
Let $K$ be the closure of an open domain and suppose that $\partial \Omega, \, \partial K$ are of class $C^{1,\alpha}$ for some $\alpha \in (0,1)$. Then 
\begin{equation}\label{qcapacity}
\capac_{Q_V}(K, \Omega, g) = Q_V(u) 
\end{equation}
where $u$ is the unique positive solution $u \in C^0(\overline \Omega) \cap \wupz \cap C^{1,\mu}(\overline\Omega\backslash K)$ of 
\begin{equation}\label{QVcapacitor}
\left\{ \begin{array}{l}
Q_V'(u) = 0 \qquad \text{on } \Omega \backslash K, \\[0.2cm]
u= g \quad \text{on } K, \qquad u=0 \quad \text{on } \partial \Omega.
\end{array}\right.
\end{equation}
We call such a solution $u$ the $Q_V$-capacitor of $(K, \Omega, g)$. 
\end{Proposition}
\begin{Remark}
\emph{The existence, uniqueness and positivity of $u$ is granted by $iv)$ in Proposition \ref{pro115}. As for regularity, the interior estimate $u \in C^{1,\mu}_\loc(\overline\Omega\backslash K)$ follows by \cite{tolksdorf,dibenedetto}, and the boundary continuity at $\partial K \cup \partial \Omega$ by Theorem 5.4, page 235 in \cite{ZM}. The fact that $u \in \wup$ follows by standard theory of Sobolev functions\footnote{In fact, $u-g \in W^{1,p}(\Omega \backslash K)$ has zero trace on $\partial K$, thus it is the $W^{1,p}$-limit (and also, up to extracting a subsequence, the pointwise limit) of some sequence $\{\varphi_j\} \subset C^\infty(\overline{\Omega \backslash K})$ where $\varphi_j \equiv 0$ in a neighbourhood of $\partial K$. Extending $\varphi_j$ to be zero on $K$ we have that $g+\varphi_j$ is Cauchy in $W^{1,p}$ and pointwise convergent to $u$, thus $g+\varphi_j \ra u \in \wup$.}. 
}
\end{Remark}
\begin{proof} 
Let $\varphi \in \Dd(K, \Omega,g)$. First, we claim that $\hat \varphi = \min\{\varphi,g\} \in \Dd(K, \Omega,g)$ solves $Q_V(\hat\varphi) \le Q_V(\varphi)$, whence we can assume, without loss of generality, that $\varphi \le g$ on $\Omega$ (and hence $\varphi =g$ on a neighbourhood of $K$). 

%Indeed, consider the open set $U=\{\varphi>g\}$. From the Lagrangian representation in Proposition 
%\ref{prop_lagrangian} it follows that
%$$
%Q_V(\varphi) \ge \int_M \mathcal{L}(\varphi,g)\di \mu_f \ge \int_{M\backslash U} \mathcal{L}(\varphi,g)
%\di \mu_f. 
%$$
%However, by the non-negativity of $\mathcal{L}$ and the fact that $\mathcal{L}(\psi,g) \equiv 0$ on $U$ %if and only if $\psi$ is a multiple of $g$, we deduce that the right hand-side of the above inequality is %exactly $Q_V(\hat \varphi)$, which proves our claim. 
%
Consider the open set $U = \{\varphi >g\}\Subset \Omega$. We test $Q_V'(g) \ge 0$ with the non-negative function $(\varphi^p-g^p)_+/g^{p-1} \in \wupz$ and use the non-negativity of the Lagrangian in \eqref{deflagrangian} to deduce  
\begin{equation}\label{arrayclassico}
\begin{array}{lcl}
0 & \le & \disp \disp \int_{U}|\nabla g|^{p-2} \langle \nabla g, \nabla\left(\frac{\varphi^p}{g^{p-1}}\right)\rangle \di \mu_f - \int_U V\varphi^p\di \mu_f - pQ_V\big(g_{|U}\big) \\[0.4cm]
& = & \disp p \int_U \left(\frac{\varphi}{g}\right)^{p-1}|\nabla g|^{p-2} \langle \nabla g, \nabla \varphi \rangle \di \mu_f - (p-1) \int_U \left(\frac{\varphi}{g}\right)^{p}|\nabla g|^{p}\di \mu_f \\[0.4cm]
& & \disp - \int_U V\varphi^p\di \mu_f  -p Q_V\big(g_{|U}\big) \\[0.4cm]
& = & \disp \int_U |\nabla \varphi|^p \di \mu_f - \int_U \mathcal{L}(\varphi,g)\di \mu_f - \int_U V\varphi^p\di \mu_f -p Q_V\big(g_{|U}\big)\\[0.4cm]
& \le & \disp p Q_V\big(\varphi_{|U}\big) -p Q_V\big(g_{|U}\big),
\end{array}
\end{equation}
%Using Cauchy-Schwarz  and Young inequalities 
%\begin{equation}\label{intercuu}
%\int_U \left(\frac{\varphi}{g}\right)^{p-1}|\nabla g|^{p-2} \langle \nabla g, \nabla \varphi \rangle \di \mu_f \le \frac{p-1}{p} \int_U \left(\frac{\varphi}{g}\right)^{p}|\nabla g|^{p}\di \mu_f + \frac{1}{p} \int_U |\nabla \varphi|^p \di \mu_f,
%\end{equation}
%and substituting into \eqref{arrayclassico} we conclude that $Q_V(g_{|U}) \le Q_V(\varphi_{|U})$. Consequently,
%$$
%Q_V\big( \hat \varphi\big) = Q_V\big( \hat \varphi_{|U}\big) + Q_V\big( \hat \varphi_{\Omega \backslash U}\big) = Q_V\big( g_{|U}\big) + Q_V\big( \varphi_{\Omega \backslash U}\big) \le Q_V\big(\varphi_{|U}\big) + Q_V\big(\varphi_{\Omega \backslash U}\big) = Q_V \big(\varphi\big),
%$$
hence $Q_V(\hat\varphi_{|U})= Q_V(g_{|U}) \le Q_V(\varphi_{|U})$. Since $\varphi \equiv \hat\varphi$ on $\Omega \backslash U$,   the claim follows. Let now $\varphi \in \Dd(K, \Omega, g)$ be such that $\varphi = g$ on $K$. By density, we can assume that $\varphi \in \lip_0(\Omega)$. We therefore have $u-\varphi \in \lip_0(\Omega \backslash K)$. Again by density, we can further assume that $\varphi =0$ in a neighbourhood of $\partial \Omega$. Thus, testing $Q_V'(u)=0$ with $(\varphi^p-u^p)/u^{p-1} \in \lip_0(\Omega \backslash K)$ and proceeding as above we obtain
\begin{equation}\label{arrayclassico2}
\begin{array}{lcl}
0 & = & \disp \disp \int_{\Omega\backslash K}|\nabla u|^{p-2} \langle \nabla u, \nabla\left(\frac{\varphi^p}{u^{p-1}}\right)\rangle \di \mu_f - \int_{\Omega\backslash K} V\varphi^p\di \mu_f - pQ_V\big(u_{|\Omega\backslash K}\big) \\[0.4cm]
& = & \disp p \int_{\Omega\backslash K} \left(\frac{\varphi}{u}\right)^{p-1}|\nabla u|^{p-2} \langle \nabla u, \nabla \varphi \rangle \di \mu_f - (p-1) \int_{\Omega\backslash K} \left(\frac{\varphi}{u}\right)^{p}|\nabla u|^{p}\di \mu_f \\[0.4cm]
& & \disp - \int_{\Omega\backslash K} V\varphi^p\di \mu_f - pQ_V\big(u_{|\Omega\backslash K}\big) \\[0.4cm]
& = & \disp \int_{\Omega \backslash K} |\nabla \varphi|^p \di \mu_f - \int_{\Omega \backslash K} \mathcal{L}(\varphi,u)\di \mu_f - \int_{\Omega\backslash K} V\varphi^p\di \mu_f - pQ_V\big(u_{|\Omega\backslash K}\big)\\[0.4cm]
& \le & \disp p Q_V\big(\varphi_{|\Omega \backslash K}\big) -p Q_V\big(u_{|\Omega \backslash K}\big).
\end{array}
\end{equation}
As $u=\varphi=g$ on $K$, we conclude $Q_V(u) \le Q_V(\varphi)$ and whence $Q_V(u) \le \capac(K,\Omega,g)$. Since $u$ lies in the $W^{1,p}$ closure of $\Dd(K, \Omega, g)$, equality \eqref{qcapacity} follows. 
\end{proof}
\begin{Remark}
\emph{By the pasting Lemma \ref{lem_pasting}, note that $u$ solving \eqref{QVcapacitor} is a supersolution on the whole $\Omega$, that is, $Q_V'(u) \ge 0$ on $\Omega$.
}
\end{Remark}
\begin{Proposition}\label{prop_QV esplicit}
In the assumptions of the previous theorem, suppose that $Q_V'(g)=0$ on a neighbourhood of $K$, and that $\partial K$ is smooth. Then,
\begin{equation}\label{explicQVcapac}
Q_V(u) = \frac{1}{p} \int_{\partial K} g\left[|\nabla g|^{p-2}\frac{\partial g}{\partial \nu}- |\nabla u|^{p-2}\frac{\partial u}{\partial \nu}\right]\di \sigma_f, 
\end{equation}
where $\nu$ is the unit normal to $\partial K$ pointing outward of $K$.
\end{Proposition}
\begin{proof}
Let $T \approx \partial K \times (-\eps_0, \eps_0) \Subset \Omega$ be a tubular neighbourhood of $\partial K$ where Fermi coordinates are defined, and let $\rho(x)$ be the smooth signed distance from $\partial K$, that is, $\rho(x) = \dist(x, \partial K)$ if $x \not \in K$, and $\rho(x) = - \dist(x,\partial K)$ if $x \in K$. Let $h \in \lip(\R^+_0)$ be such that $h(t)=0$ if $t\le 0$, $h(t)=t$ for $t \in[0,1]$ and $h(t)=1$ for $t \ge 1$ and, for small $\eps \in (0, \eps_0)$, set $h_\eps(t) \doteq h(t/\eps)$.
%\begin{equation}\label{barpsieps}
%\bar\psi_\eps(x) = \left\{ \begin{array}{ll}
%h\left(\frac{\rho(x)}{\eps}\right) & \quad \text{if } x \in K, \ \dist(x, \partial K) \le \delta, \\[0.2cm]
%0 & \quad \text{otherwise}
%\end{array} \right.
%\end{equation}
Applying \eqref{QVcapacitor} on $\Omega\backslash K$ to the test function $h_\eps(\rho) u \in \lip_0(\Omega \backslash K)$, using the coarea's formula and letting $\eps \ra 0$, since $g$ is $C^1$ we deduce
\begin{equation}\label{bordouu}
\disp 0 = \lim_{\eps \ra 0} Q_V'(u)\big[h_\eps(\rho) u\big] = p\disp Q_V\big(u_{\Omega \backslash K}\big) + \int_{\partial K} u |\nabla u|^{p-2} \frac{\partial u}{\partial \nu}\di \sigma_f.
\end{equation}
In a similar way, applying $Q_V'(g) = 0$ on $K$ to the non-negative test function $h_\eps(-\rho)g \in \lip_0(K)$ and letting $\eps \ra 0$ we deduce that 
\begin{equation}\label{bordogg}
\disp \disp 0 = \lim_{\eps \ra 0} Q_V'(u)\big[h_\eps(-\rho) g\big] = \disp pQ_V\big(u_{K}\big) - \int_{\partial K} g |\nabla g|^{p-2} \frac{\partial g}{\partial \nu}\di \sigma_f.
\end{equation}
Subtracting the two identities and using $u=g$ on $\partial K$ yields \eqref{explicQVcapac}.
\end{proof}
Next, we consider the $Q_V$-capacity of $K$ in the whole $M$:
\begin{equation}\label{def_globalcapacity}
\begin{array}{c}
\disp \Dd(K, g)  = \Big\{ \varphi \in C^0_c(M) \cap W^{1,p}_\loc(M) \ : \ \varphi \ge g \ \text{ in a neighbourhood of } K\Big\} \\[0.4cm]
\disp \capac_{Q_V}(K, g) \doteq \disp  \inf_{\varphi \in \Dd(K, g) } Q_V(\varphi). 
\end{array}
\end{equation}
Let $\{\Omega_j\}$ be an exhaustion of $M$ with $K \Subset \Omega_1$. Then, from the definitions it readily follows that 
$$
\capac_{Q_V}(K,g) = \inf_j \capac_{Q_V}(K, \Omega_j,g) = \lim_{j \ra +\infty} \capac_{Q_V}(K, \Omega_j,g).
$$
If $K$ is the closure of a open set and $\partial K$ is of class $C^{1,\alpha}$ for some $\alpha \in (0,1)$, let $u_j$ be the $Q_V$-capacitor of $(K, \Omega_j,g)$. By Proposition \ref{prop_capacattingida},
\begin{equation}\label{capaccomelimQV}
\capac_{Q_V}(K,g) = \lim_{j\ra +\infty} Q_V(u_j).
\end{equation}
Proposition \ref{prop_compagen} implies that $0\le u_j \le u_{j+1} \le g$ for each $j$, whence, by Dini theorem and elliptic estimates, $u_j$ converges locally uniformly on $M$, in $W^{1,p}_\loc(M)$ and in the $C^1$ topology on $M\backslash K$ to a weak solution $u \in C^0(M) \cap W^{1,p}_\loc(M) \cap C^{1,\mu}_\loc(M\backslash K)$ of
\begin{equation}\label{capacitorglobal}
\left\{ \begin{array}{l}
Q_V'(u) = 0 \qquad \text{on } M \backslash K, \\[0.2cm]
u= g \quad \text{on } K, \qquad 0<u \le g \quad \text{on } M\backslash K.
\end{array}\right.
\end{equation}
The pasting Lemma \ref{lem_pasting} guarantees that $Q_V'(u) \ge 0$ on the whole $M$. We call such a $u$ the $Q_V$-capacitor of $(K,g)$. 
\begin{Proposition}\label{claimbella}
In the assumptions of Proposition \ref{prop_QV esplicit}, if $Q_V'(g)=0$ on a neighbourhood of $K$, 
\begin{equation}\label{finedeltunnel}
\capac_{Q_V}(K,g) = \frac{1}{p} \int_{\partial K} g\left[|\nabla g|^{p-2}\frac{\partial g}{\partial \nu}- |\nabla u|^{p-2}\frac{\partial u}{\partial \nu}\right]\di \sigma_f,
\end{equation}
where $\nu$ is the unit normal to $\partial K$ pointing outward of $K$.
\end{Proposition}
\begin{proof}
By Proposition \ref{prop_capacattingida}, $\capac_{Q_V}(K,g) = \lim_j Q_V(u_j)$, where $u_j$ is the $Q_V$-capacitor of $(K, \Omega_j,g)$. Now, since $u_j \ra u$ in $C^1(\partial K)$, it is enough to pass to the limit in \ref{explicQVcapac}. 
\end{proof}

Next Theorem, the core of this section, relates the subcriticality of $Q_V$ and the $Q_V$-capacity with other basic properties, which we will define below. It is due to Y. Pinchover and K. Tintarev (see \cite{pinchovertintarev}), and it is known in the literature as the ground state alternative. The authors state it for $f$ constant and $M= \R^m$. Our contribution here is to include the $Q_V$-capacity properties to the above picture. However, since at some point of \cite{pinchovertintarev} the authors use inequalities for which we found no counterpart in a manifold setting, we prefer to provide a full proof which sometimes uses arguments that differ from those in \cite{pinchovertintarev, pinchovertintarev2}, though keeping the same guidelines. 
\begin{Definition}\label{def_weighted}
For $V  \in L^\infty_\loc(M)$, define $Q_V$ as in \eqref{06} and let $\Omega \subseteq M$ be an open set.  
\begin{itemize}
\item[iii)] $Q_V$ has a \emph{\textbf{weighted spectral gap}} on $\Omega$ if there exists $W \in C^0(\Omega)$, $W>0$ on $\Omega$ such that
\begin{equation}\label{114}
\disp \int_\Omega W(x) |\varphi|^p \di \mu_f \le Q_V (\varphi) \qquad \forall \ \varphi \in \lip_c (\Omega).
\end{equation}
\vspace{0.1cm}
\item[iv)] A sequence $\{\eta_j\} \in L^\infty_c(\Omega) \cap W^{1,p}(\Omega)$ is said to be a \emph{\textbf{null sequence}} if  $\eta_j \ge 0$ a.e. for each $j$, $Q_V(\eta_j) \ra 0$ as $j \ra +\infty$ and there exists a relatively compact open set $B\Subset M$ and $C>1$ such that $C^{-1} \le \|\eta_j\|_{L^p(B)} \le C$ for each $j$. 
\vspace{0.1cm}
\item[v)] A function $0 \le \eta \in W^{1,p}_\loc(\Omega)$, $\eta \ge 0$, $\eta \not \equiv 0$ is a \emph{\textbf{ground state}} for $Q_V$ on $\Omega$ if it is the $L^p_\loc(\Omega)$ limit of a null sequence. 
\end{itemize}
\end{Definition}
\begin{Theorem}\label{teo_alternative}
Let $(M, \metric)$ be connected and non-compact, and for $V \in L^\infty_\loc(M)$ consider an operator $Q_V \ge 0$ on $M$. Then, either $Q_V$ has a weighted spectral gap or a ground state on $M$, and the two possibilities mutually exclude. Moreover, the following properties are equivalent:
\begin{itemize}
\item[$(i)_{\SR}$] $Q_V$ has a weighted spectral gap. 
\vspace{0.1cm}
%\item[$(ii)$] There exists a positive solution $u \in W^{1,p}_\loc(M)$ of $Q_V(u) \ge 0$, $Q_V(u) \not \equiv 0$;
%
\item[$(ii)_{\SR}$] $Q_V$ is subcritical on $M$.
\vspace{0.1cm}
\item[$(iii)_{\SR}$] There exist two positive solutions $u_1,u_2 \in C^0(M) \cap W^{1,p}_\loc(M)$ of $Q'_V(u) \ge 0$ which are not proportional.
\vspace{0.1cm}
\item[$(iv)_{\SR}$] For some (any) $K \Subset M$ compact with non-empty interior, and for some (any) $0<g \in C^0(M) \cap W^{1,p}_\loc(M)$ solving  $Q_V'(g) \ge 0$, $\capac(K,g)>0$.
\vspace{0.1cm}
%
%\item[$(v)_{\SR}$] For each $K\subset M$ compact with non-empty interior, and for each $0<g \in C^0(M) \cap W^{1,p}_\loc(M)$ solving  $Q_V'(g) \ge 0$, $\capac(K,g)>0$.

\end{itemize}
When $Q_V$ has a ground state $\eta$, $\eta$ solves $Q_V'(\eta)=0$, and in particular $\eta \in C^{1,\mu}_\loc(M)$, $\eta>0$ on $M$. Furthermore, the next properties are equivalent:
\begin{itemize}
\item[$(i)_{\GSR}$] $Q_V$ has a ground state.
\vspace{0.1cm}
%
%\item[$(ii)$] There exists a positive solution $u \in W^{1,p}_\loc(M)$ of $Q_V(u) \ge 0$, $Q_V(u) \not \equiv 0$;
%
\item[$(ii)_{\GSR}$] All positive solutions $g \in C^0(M) \cap W^{1,p}_\loc(M)$ of $Q_V'(g)\ge 0$ are proportional; in particular, each positive supersolution is indeed a solution (hence, a ground state).
\vspace{0.1cm}
\item[$(iii)_{\GSR}$] For some (any) $K \Subset M$ compact with non-empty interior, and for some (any) $0<g \in C^0(M) \cap W^{1,p}_\loc(M)$ solving  $Q_V'(g) \ge 0$, $\capac(K,g)= 0$.
\vspace{0.1cm}
\end{itemize}
\end{Theorem}
%
%
%\begin{Theorem}\label{teo_alternative}
%Let $(M, \metric)$ be a non-compact Riemannian manifold, $f \in C^\infty (M)$, $p \in (1,+\infty)$ and, for $V \in L^\infty_\loc(M)$, define $Q_V$ as in \eqref{06}. Suppose that $Q_V$ is non-negative on $M$. Then, the following properties are equivalent:
%\begin{itemize}
%\item[$(i)$] There exist two positive solutions $u_1,u_2 \in W^{1,p}_\loc(M)$ of $Q'_V(u) \ge 0$ which are not proportional;
%%
%%\item[$(ii)$] There exists a positive solution $u \in W^{1,p}_\loc(M)$ of $Q_V(u) \ge 0$, $Q_V(u) \not \equiv 0$;
%%
%\item[$(ii)$] $Q_V$ is subcritical on $M$;
%%
%\item[$(iii)$] $Q_V$ has a weighted spectral gap. 
%%
%\item[$(iv)$] For some $K \Subset M$ compact, and for some $0<g \in \lip_\loc(M)$ solving  $Q_V'(g) \ge 0$, $\capac(K,g)>0$.
%%
%\item[$(v)$] For each $K\subset M$ compact, and for each $0<g \in \lip_\loc(M)$ solving  $Q_V'(g) \ge 0$, $\capac(K,g)>0$.
%\end{Theorem}
%

%
\begin{proof}
Hereafter, $L^p$ and $W^{1,p}$ spaces will be considered with respect to the measure $\di \mu_f$. We begin with the following fact.\\[0.2cm]
\noindent \textbf{Claim 1: } Fix an open set $U \Subset M$. If $\{\phi_j\} \subset L^\infty_c(M) \cap W^{1,p}(M)$ is such that $\|\phi_j\|_{L^p(U)} + Q_V(\phi_j) \le C$ for some $C>0$ independent of $j$, then $\{\phi_j\}$ is locally bounded in $W^{1,p}(M)$.\\[0.2cm]
\textbf{Proof of Claim 1. } Up to replacing $\phi_j$ with $|\phi_j|$, we can assume that $\phi_j \ge 0$ a.e. on $M$. Using that $Q_V$ is non-negative on $M$, choose a positive solution $g \in C^{1,\mu}_\loc(M)$ of $Q_V'(g)=0$ and consider the Lagrangian representation of $Q_V(\phi_j)$:
$$
Q_V(\phi_j) = \int_M \mathcal{L}(\phi_j,g)\di \mu_f.
$$
Fix $\Omega \Subset M$ containing $\overline U$. Since $Q_V(\phi_j) \le C$ and $\mathcal{L}(\phi_j,g) \ge 0$ on $M$, it holds
\begin{equation}\label{walkon}
0 \le \limsup_{j \ra +\infty} \int_{\Omega} \mathcal{L}(\phi_j,g)\di \mu_f \le C.
\end{equation}
By using Cauchy-Schwarz and Young inequalities on the third addendum of the expression of $\mathcal{L}(\phi_j,g)$ we deduce that, for each $\eps >0$,
\begin{equation}\label{lowerb}
\mathcal{L}(\phi_j,g) \ge (1-\eps^p) |\nabla \phi_j|^p + (p-1)\left(1-\eps^{-\frac{p}{p-1}}\right)\left(\frac{\phi_j}{g}\right)^{p}|\nabla g|^p.
\end{equation}
In our assumptions, $|\nabla \log g|\in L^\infty(\Omega)$. Setting $\eps= 1/2$ in \eqref{lowerb}, integrating and using \eqref{walkon} we get
\begin{equation}
\limsup_{j \ra +\infty} \left[ \left(1-2^{-p}\right) \int_\Omega |\nabla \phi_j|^p\di \mu_f + (p-1)\left(1-2^{\frac{p}{p-1}}\right)\|\nabla \log g\|_{L^\infty(\Omega)} \int_{\Omega} |\phi_j|^p\di \mu_f \right] \le C.
\end{equation}
From this inequality we argue the existence of constants $C_1,C_2>0$ independent of $j$ such that 
\begin{equation}\label{usefulestimate}
\|\nabla \phi_j\|_{L^p(\Omega)} \le C_1\|\phi_j\|_{L^p(\Omega)} + C_2.
\end{equation}
Suppose now, by contradiction, that $\{\phi_j\}$ is not bounded in $W^{1,p}(\Omega)$. By \eqref{usefulestimate}, $\|\phi_j\|_{L^p(\Omega)}$ diverge. Set $\bar\phi_j = \phi_j/\|\phi_j\|_{L^p(\Omega)}$. Then, by \eqref{usefulestimate} $\{\bar\phi_j\}$ is bounded in $W^{1,p}(\Omega)$, thus it has a subsequence (still called $\{\bar\phi_j\}$) converging weakly in $W^{1,p}$, strongly in $L^p$ and pointwise almost everywhere to some non-negative function $\Psi$ satisfying $\|\Psi\|_{L^p(\Omega)} = \lim_j\|\bar\phi_j\|_{L^p(\Omega)}=1$. Since $|\nabla \log g| \in L^\infty(\Omega)$ we straightforwardly have that 
\begin{equation}\label{semplici}
\begin{array}{rll}
(i) & \bar\phi_j^p|\nabla \log g|^p \ra \Psi^p|\nabla \log g|^p & \quad \text{in } L^1(\Omega), \\[0.2cm]
(ii)& \bar\phi_j^{p-1}|\nabla \log g|^{p-1} \ra \Psi^{p-1}|\nabla \log g|^{p-1} & \quad \text{in } L^{\frac{p}{p-1}}(\Omega).
\end{array}
\end{equation}
We just prove $(ii)$, the other being a consequence of the proof. By the elementary inequalities 
\begin{equation}\label{elementaty1}
(x-y)^\beta \le x^\beta - y^\beta \le Cx^{\beta-1}(x-y) \qquad \text{for } y \in [0,x] \text{ and } \beta>1, 
\end{equation}
for some $C$ depending on $\beta$, with a final application of H\"older inequality we deduce that
% 
%as $\phi_j \ra \Psi$ in $L^p(\Omega)$ implies that $\phi_j \ra \Psi$ also in $L^p(\Omega)$ with reference measure $|\nabla \log g|^p\di \mu_f$). Next,  
%\begin{equation}
%\phi_j^{p-1}|\nabla \log g|^{p-1} \longrightarrow \Psi^{p-1}|\nabla \log g|^{p-1} \qquad \text{in } L^{\frac{p}{p-1}}(\Omega).
%\end{equation}
\begin{equation}\label{utilistime}
\begin{array}{l}
\disp \|(\bar\phi_j^{p-1}-\Psi^{p-1})|\nabla \log g|^{p-1}\|_{L^{p/(p-1)}(\Omega)}^{p/(p-1)} \le \|\nabla \log g\|_{L^\infty(\Omega)}^p\int_\Omega \big|\bar\phi_j^{p-1}-\Psi^{p-1}\big|^{\frac{p}{p-1}}\di \mu_f \\[0.3cm]
\disp \le \|\nabla \log g\|_{L^\infty(\Omega)}^p\int_\Omega \big|\bar\phi_j^{p}-\Psi^{p}\big|\di \mu_f \le C\|\nabla \log g\|^p_{L^\infty(\Omega)} \int_\Omega \big(\max\{\bar\phi_j,\Psi\}\big)^{p-1}\big|\bar\phi_j-\Psi|\di \mu_f, \\[0.3cm]
\disp \le C\|\nabla \log g\|^p_{L^\infty(\Omega)} \left\|\bar\phi_j + \Psi\right\|_{L^p(\Omega)}^{(p-1)/p}\|\bar\phi_j-\Psi\|_{L^p(\Omega)},
\end{array}
\end{equation}
and this latter goes to zero $\bar\phi_j \ra \Psi$ in $L^p(\Omega)$. Now, coupling $(ii)$ with the weak convergence of $\bar\phi_j$ to $\Psi$ in $W^{1,p}(\Omega)$, we get
$$
\int_\Omega \left(\frac{\bar\phi_j}{g}\right)^{p-1} |\nabla g|^{p-2} \langle \nabla g, \nabla \bar\phi_j \rangle \di \mu_f \longrightarrow \int_\Omega \left(\frac{\Psi}{g}\right)^{p-1} |\nabla g|^{p-2} \langle \nabla g, \nabla \Psi \rangle \di \mu_f,
$$
so that, combining $(i), (ii)$ and the weak lower semicontinuity of $\|\cdot\|_{W^{1,p}(\Omega)}$, 
$$
0 \le \int_\Omega \mathcal{L}(\Psi,g)\di \mu_f \le \liminf_{j\ra +\infty}\int_\Omega \mathcal{L}(\bar\phi_j,g)\di \mu_f = \|\phi_j\|^{-p}_{L^p(\Omega)} \int_\Omega \mathcal{L}(\phi_j,g)\di \mu_f \ra 0 
$$
as $j \ra +\infty$. Hence $\mathcal{L}(\Psi,g)=0$ on $\Omega$, so by Proposition \ref{prop_lagrangian} $\Psi=cg$ for some constant $c \ge 0$. However, since $\|\phi_j\|_{L^p(U)}\le C$ for each $j$,
$$
0 \le \|\Psi\|_{L^p(U)} = \lim_{j \ra +\infty} \|\bar\phi_j\|_{L^p(U)} = \lim_{j \ra +\infty} \frac{\|\phi_j\|_{L^p(U)}}{\|\phi_j\|_{L^p(\Omega)}} \le \limsup_{j \ra +\infty} \frac{C}{\|\phi_j\|_{L^p(\Omega)}} =0,
$$
hence $c=0$ and $\Psi \equiv 0$ on $\Omega$, contradicting the fact that $\|\Psi\|_{L^p(\Omega)}=1$. This concludes the proof of the claim.\\[0.2cm]
\noindent \textbf{Claim 2: } Either $Q_V$ has a weighted spectral gap or a ground state, but not both.\\[0.2cm]
\noindent \textbf{Proof of Claim 2.} For a relatively compact open set $U$, define the $L^p$-capacity $c$ as follows:
$$
\Dd(U) = \Big\{\varphi \in L^\infty_c(M) \cap W^{1,p}(M) : \|\varphi\|_{L^p(U)} =1\Big\}, \qquad c(U) = \inf_{\varphi \in \Dd(U)} Q_V(\varphi),
$$
where, as usual, the $L^p$-norm is computed with respect to $\di \mu_f$. Then, two mutually exclusive cases may occur: ether $c(U) >0$ for each $U$, or $c(U) =0$ for some $U$. In the first case, it is easy to see that $Q_V$ has a weighted spectral gap. Indeed, let $\{U_j\}$ be a locally finite covering of $M$ via relatively compact, open sets, set $c_j = c(U_j)>0$ and let $\{t_j\}$ be a sequence of positive numbers such that $\sum_jt_j=1$. For each $\varphi \in \lip_c(M)$, by the definition of $c_j$ we get
$Q_V(\varphi) \ge c_j \|\varphi\|_{L^p(U_j)}^p$, and summing up we get
$$
Q_V(\varphi) = \left[\sum_{j=1}^{+\infty} t_j\right] Q_V(\varphi) \ge \sum_{j=1}^{+\infty} t_j c_j \int_{U_j}|\varphi|^p \di \mu_f = \int_M \hat W|\varphi|^p \di \mu_f,
$$
where
$$
\hat W(x) = \sum_{j=1}^{+\infty} t_jc_j 1_{U_j}(x) >0  \qquad \text{on } M.
$$
We can thus choose a weighted spectral gap $W$ by taking a positive, continuous function $W$ not exceeding $\hat W$.\\
Now, suppose that $c(U)=0$ for some $U$. We show that there exists a ground state. Indeed, by the definition of $c(U)$ there exists $\{\eta_j\} \subset L^\infty_c(M) \cap W^{1,p}(M)$ such that $\|\eta_j\|_{L^p(U)}=1$ and $Q_V(\eta_j) \ra 0$. Up to replacing $\eta_j$ with $|\eta_j|$, we can suppose that $\eta_j \ge 0$ a.e. on $M$. By Claim 1, $\eta_j$ is locally bounded in $W^{1,p}$, and a Cantor type argument on an increasing exhaustion of $M$ ensures the existence of a subsequence, still called $\{\eta_j\}$, converging weakly in $W^{1,p}_\loc(M)$ and strongly in $L^p_\loc(M)$ to some function $\eta \in W^{1,p}_\loc(M)$. By definition, $\eta$ is a ground state.\\[0.2cm] 
We now show our desired equivalences. To establish those involving $(iv)_{\SR}$ and $(iii)_{\GSR}$, where the ``some/all" alternative appears, then we will always assume the weakest alternative and prove the strongest one.\\[0.2cm]
$\mathbf{(i)_{\GSR} \Rightarrow (ii)_{\GSR}}$. Let $\eta \ge 0$ be a ground state, and let $\{\eta_j\} \subset L^\infty_c(M) \cap W^{1,p}(M)$ be a null sequence converging in $L^p_\loc$ to $\eta$. Then, by Claim 1 $\{\eta_j\}$ is locally bounded in $W^{1,p}$, thus up to passing to a subsequence we can assume that also $\eta_j \ra \eta$ weakly in $W^{1,p}_\loc$. Consider a positive solution $g \in C^0(M) \cap  W^{1,p}_\loc(M)$ of $Q_V'(g)\ge 0$. Fix $\Omega \Subset M$. Up to multiply $g$ by a large positive constant, we can suppose that $\{x \in \Omega: \eta(x) < g(x)\}$ has positive measure. Let $\bar\eta_j= \min\{\eta_j,g\}$, and note that $\{\bar\eta_j\}$ is still a null sequence, converging weakly in $W^{1,p}_\loc$ to $\bar\eta= \min\{\eta,g\}$. Consider the Lagrangian representation
\begin{equation}\label{lagrine}
Q_V(\bar\eta_j) \ge \int_M \mathcal{L}(\bar\eta_j,g)\di \mu_f
\end{equation}
guaranteed by Proposition \ref{prop_lagrangian}. We claim that
\begin{equation}\label{claimweakly}
\frac{\bar\eta_j^p}{g^{p-1}} \ra \frac{\bar\eta^p}{g^{p-1}} \qquad \text{weakly in } W^{1,p}(\Omega). 
\end{equation}
To see this, we follows arguments analogous to those yielding \eqref{semplici}. Choose a constant $c_\Omega>0$ large enough to satisfy $g \ge c_\Omega^{-1}$ on $\Omega$ and $\bar\eta_j \le c_\Omega$. By a  direct computation, $\|\bar\eta_j^p/g^{p-1}\|_{W^{1,p}(\Omega)}$ is uniformly bounded so that, by density, it is enough to check the weak convergence with test function $\varphi \in \lip_0(\Omega)$. From
$$ 
\left|\int_\Omega\frac{\bar\eta_j^p-\bar\eta^p}{g^{p-1}} \varphi \di \mu_f \right| \le c_\Omega^{p-1}\|\varphi\|_{L^\infty(\Omega)} \int_\Omega \big|\bar\eta_j^p-\bar\eta^p\big|\di \mu_f,
$$
applying the inequalities in \eqref{utilistime} from the second line to the end we deduce that $\bar\eta_j^p/g^{p-1}\ra \bar\eta^p/g^{p-1}$ weakly in $L^p(\Omega)$. Regarding the gradient part,  
$$
\left|\int_\Omega \langle \nabla\left(\frac{\bar\eta_j^p-\bar\eta^p}{g^{p-1}}\right), \nabla \varphi \rangle \di \mu_f\right| \le \mathrm{(I)} + \mathrm{(II)}, 
$$
where 
$$
\begin{array}{rcl}
\mathrm{(I)} & = & \disp (p-1)\int_\Omega \big|\bar\eta_j^p-\bar\eta^p\big||\nabla g| \frac{|\nabla \varphi|}{g^{p}} \\[0.3cm]
\mathrm{(II)} & = & \disp p \left| \int_\Omega \langle \left(\frac{\bar\eta_j}{g}\right)^{p-1} \nabla \varphi, \nabla \bar\eta_j \rangle \di \mu_f - \int_\Omega\langle\left(\frac{\bar\eta}{g}\right)^{p-1}\nabla \varphi, \nabla \bar\eta\rangle \di \mu_f \right|.
\end{array}
$$
As for (I), by H\"older and both the inequalities in  \eqref{elementaty1} we deduce 
$$
\begin{array}{rcl}
\frac{1}{p-1}\mathrm{(I)} & \le & \disp c_\Omega^p \|\nabla \varphi\|_{L^\infty(\Omega)}\|\nabla g\|_{L^p(\Omega)}\left(\int_\Omega \big|\bar\eta_j^p-\bar\eta^p\big|^{\frac{p}{p-1}}\di \mu_f\right)^{\frac{p-1}{p}} \\[0.3cm]
& \le & \disp c_\Omega^p \|\nabla \varphi\|_{L^\infty(\Omega)}\|\nabla g\|_{L^p(\Omega)}\left(\int_\Omega \left|\bar\eta_j^{\frac{p^2}{p-1}}-\bar\eta^{\frac{p^2}{p-1}}\right|\di \mu_f\right)^{\frac{p-1}{p}} \\[0.3cm]
& \le & \disp c_\Omega^p \|\nabla \varphi\|_{L^\infty(\Omega)}\|\nabla g\|_{L^p(\Omega)}C^{\frac{p-1}{p}}\left(\int_\Omega \max\{\bar\eta_j,\bar\eta\}^{\frac{p^2}{p-1}-1}\left|\bar\eta_j-\bar\eta\right|\di \mu_f\right)^{\frac{p-1}{p}} \\[0.3cm]
& \le & \disp c_\Omega^{p + \frac{p}{p-1}-\frac{p-1}{p}} \|\nabla \varphi\|_{L^\infty(\Omega)}\|\nabla g\|_{L^p(\Omega)}C^{\frac{p-1}{p}}\left(\int_\Omega \left|\bar\eta_j-\bar\eta\right|\di \mu_f\right)^{\frac{p-1}{p}}.
\end{array}
$$
Again by H\"older inequality, the last integral goes to zero since $\bar\eta_j \ra \bar \eta$ in $L^p(\Omega)$, which shows that (I)$\ra 0$ as $j \ra +\infty$. Finally, we consider (II). To show that (II)$\ra 0$, using the weak convergence if $\bar\eta_j$ to $\bar\eta$ in $W^{1,p}(\Omega)$ and standard estimates it is enough to prove that 
$$
\left(\frac{\bar\eta_j}{g}\right)^{p-1} \nabla \varphi \ra \left(\frac{\bar\eta}{g}\right)^{p-1} \nabla \varphi \qquad \text{strongly in } L^{\frac{p}{p-1}}(\Omega).
$$
This follows from
$$
\int_\Omega |\nabla \varphi|^{\frac{p}{p-1}} \left|\left(\frac{\bar\eta_j}{g}\right)^{p-1}- \left(\frac{\bar\eta}{g}\right)^{p-1}\right|^{\frac{p}{p-1}} \di \mu_f \le c_\Omega^{p}\|\nabla \varphi\|_{L^\infty(\Omega)}^{\frac{p}{p-1}} \int_\Omega \left| \bar\eta_j^{p-1}-\bar\eta^{p-1}\right|^{\frac{p}{p-1}}\di \mu_f 
$$
and inequalities analogous to those in the second and third lines of \eqref{utilistime}. This concludes the proof of \eqref{claimweakly}.\\
Now, \eqref{claimweakly} implies that
$$
\int_\Omega |\nabla g|^{p-2} \langle \nabla g, \nabla \left(\frac{\bar\eta_j^p}{g^{p-1}}\right) \rangle \di \mu_f \longrightarrow \int_\Omega |\nabla g|^{p-2} \langle \nabla g, \nabla \left(\frac{\bar\eta^p}{g^{p-1}}\right) \rangle \di \mu_f,
$$
so that integrating on $\Omega$ the Lagrangian identity 
$$
\mathcal{L}(\bar\eta_j,g) = |\nabla \bar\eta_j|^p - |\nabla g|^{p-2} \langle \nabla g, \nabla \left(\frac{\bar\eta_j^p}{g^{p-1}}\right) \rangle 
$$
and using the weak lower semicontinuity of the $W^{1,p}$ norm we deduce
\begin{equation}\label{ufi}
0 \le \int_\Omega \mathcal{L}(\bar\eta,g)\di \mu_f \le \liminf_{j\ra +\infty} \int_\Omega \mathcal{L}(\bar\eta_j,g)\di \mu_f.
\end{equation}
Inequalities \eqref{lagrine} and $\mathcal{L}(\bar\eta_j,g)\ge 0$ on $M$ then imply
$$
0 \le \int_\Omega \mathcal{L}(\bar\eta_j,g)\di \mu_f \le \int_M \mathcal{L}(\bar\eta_j,g)\di \mu_f \le Q_V(\bar\eta_j) \longrightarrow 0
$$
as $j \ra +\infty$, so we conclude by \eqref{ufi} that $\mathcal{L}(\bar\eta,g)\equiv 0$ on $\Omega$. By Proposition \ref{prop_lagrangian}, $\bar\eta = cg$ on $\Omega$ for some $c\ge 0$. In fact, $c>0$ since $\|\eta\|_{L^p(U)} = \lim_j \|\eta_j\|_{L^p(U)} \neq 0$, and $c<1$ since $\{x \in \Omega : \eta(x) <g(x)\}$ has positive measure. Therefore, $\bar\eta \equiv \eta$ on $\Omega$, showing that $g$ is a positive multiple of $\eta$ on $\Omega$, that is, $(ii)_{\GSR}$ holds.\\[0.2cm] 
\noindent\textbf{Claim 3. } When $Q_V$ has a ground state $\eta$, $\eta\in C^{1,\mu}_\loc(M)$, is positive and solves $Q_V'(\eta)=0$.\\[0.2cm] 
\noindent \textbf{Proof of Claim 3. } By $(i)_{\GSR} \Rightarrow (ii)_{\GSR}$, all solutions $g \in C^0(M) \cap W^{1,p}_\loc(M)$ of $Q_V'(g) \ge 0$ are proportional. Choosing $g$ to be a positive solution of $Q_V'(g)=0$ (which exists by Proposition \ref{pro115}), we get that $\eta>0$ solves $Q_V'(\eta)=0$ and $\eta \in C^{1,\mu}_\loc(M)$, proving the claim.\\[0.2cm]
\noindent $\mathbf{(iii)_{\GSR} \Rightarrow (i)_{\GSR}}$. Since $\mathrm{Int}(K) \neq \emptyset$ we select a closed smooth geodesic ball $B$ contained in $\mathrm{Int}(K)$. By the monotonicity of capacity, $\capac_{Q_V}(B,g)=0$. Fix an exhaustion $\{\Omega\}_j$ of $M$ with $B\Subset \Omega_1$, let $\eta_j$ be the $Q_V$-capacitor of $(B,\Omega_j,g)$ extended with zero outside $\Omega_j$, and let $\eta$ be the capacitor of $(B,g)$. Then, \eqref{capaccomelimQV} ensures that $Q_V(\eta_j)\ra \capac_{Q_V}(B,g)=0$, and since $\eta_j \ra \eta$ in $L^p_\loc(M)$ we deduce that $\eta$ is the desired ground state.\\[0.2cm]
$\mathbf{(iv)_{\SR}\Rightarrow (iii)_{\SR}}$. Up to enlarging $K$ (capacity increases), we can assume that $K$ is the closure of a relatively compact open set with smooth boundary. Now, consider a solution $\bar g \in C^{1,\mu}_\loc(M)$ of $Q_V'(\bar g)=0$ on $M$. Up to multiplying $\bar g$ by a small positive constant, we can suppose that $\bar g \le g$ on $K$. By the very definition of $Q_V$-capacity,
$$
\capac_{Q_V}(K, g) \ge \capac_{Q_V}(K, \bar g).
$$ 
We claim that $\capac_{Q_V}(K, \bar g)> 0$. Indeed, otherwise, by $(iii)_{\GSR} \Rightarrow (i)_{\GSR}$ just proved, $Q_V$ has a ground state and so all solutions $u \in C^0(M) \cap W^{1,p}_\loc(M)$ of $Q_V'(u) \ge 0$ are proportional. In particular, $\bar g = cg$ for some constant $c>0$, which implies $\capac_{Q_V}(K,g)= c^{-p} \capac_{Q_V}(K, \bar g) = 0$ contradicting our assumptions. Now, since $\bar g$ solves $Q_V'(\bar g)=0$ and $\capac_{Q_V}(K, \bar g)>0$, applying formula \eqref{finedeltunnel} in Proposition \ref{claimbella} with $\bar g$ replacing $g$ we deduce that necessarily the $Q_V$-capacitor $u$ of $(K,\bar g)$ is different from $\bar g$. As $u$ solves $Q_V'(u) \ge 0$ on $M$ and $u = g$ on $K$, $u$ and $\bar g$ are two non-proportional solutions of $Q_V'(v) \ge 0$, proving $(iii)_{\SR}$.\\[0.2cm]
\noindent $\mathbf{(ii)_{\GSR} \Rightarrow (iii)_{\GSR}}$. If, by contradiction, $\capac_{Q_V}(K,g)>0$ for some $(K,g)$, then by $(iv)_{\SR} \Rightarrow (iii)_{\SR}$ above there would exists two non-proportional solutions of $Q_V'(g) \ge 0$, contradicting $(ii)_{\GSR}$.\\[0.2cm]
Now, we have concluded the equivalence $(i)_{\GSR} \Leftrightarrow (ii)_{\GSR} \Leftrightarrow (iii)_{\GSR}$ in the ``ground state" part of the theorem. Combining with Claim $2$, we automatically have the validity of\\ 
$\mathbf{(i)_{\SR}\Leftrightarrow (iii)_{\SR} \Leftrightarrow (iv)_{\SR}}$.\\[0.2cm] 
We are thus left to show $(i)_{\SR} \Leftrightarrow (ii)_{\SR}$. Having observed that $(i)_{\SR}\Rightarrow (ii)_{\SR}$ is obvious (set $w=W$), to conclude we prove that \\[0.2cm] 
$\mathbf{(ii)_{\SR} \Rightarrow (i)_{\SR}}$. Suppose by contradiction that $(i)_{\SR}$ is not true. Then, by Claim $2$, $Q_V$ has a ground state $\eta$, which is positive on $M$ and solves $Q_V'(\eta)=0$. Fix a smooth open set $U$ such that $w \not\equiv 0$ on $U$, let $\{\Omega_j\} \uparrow M$ be an exhaustion of $M$ with $U\Subset \Omega_1$, and let $\eta_j$ be the $Q_V$-capacitor of $(U, \Omega_j, \eta)$. Then, by the equivalence $(i)_{\GSR} \Leftrightarrow (iii)_{\GSR}$ (in particular, by the proof of $(iii)_{\GSR} \Rightarrow (i)_{\GSR}$), $\{\eta_j\}$ is a null sequence. By the subcriticality assumption and since $\eta_j=\eta$ on $U$, 
$$
\int_U w |\eta|^p\di \mu_f \le \int_M w |\eta_j|^p\di \mu_f \le Q_V(\eta_j) \ra 0
$$
as $j \ra +\infty$, contradicting the fact that $\eta>0$ on $M$ and $w\not\equiv 0$ on $U$. 
\end{proof}

\begin{Remark}\label{rem_Yamabeinv}
\emph{We can now give a simple proof of the fact that the positivity of the Yamabe invariant $Y(M)$ in Theorem \ref{teo_zhang} implies the subcriticality of the conformal Laplacian $L_{\metric}$. Indeed, suppose the contrary. Then, by Theorem \ref{teo_alternative}, there exist a ground state $\eta>0$ and a null sequence $\{\eta_j\}$ locally $L^2$-converging to $\eta$. From the chain of inequalities
$$
Y(M) \left(\int_M |\eta_j|^{\frac{2m}{m-2}}\right)^{\frac{m-2}{m}} \le \int_M \left[|\nabla \eta_j|^2 + \frac{s(x)}{c_m} \eta_j^2\right]  \ra 0 \quad \text{as } \, j \ra +\infty,
$$
we deduce that $\eta_j \ra 0$ in $L^{\frac{2m}{m-2}}(M)$, hence locally in $L^2(M)$, a contradiction.
}
\end{Remark}

It is now immediate to prove the equivalence partly mentioned in the Introduction (Proposition \ref{prop_hyperbolicity_simple}). We suggest the interested reader to consult the very recent \cite{devyverfraaspinchover, devyverpinchover} for an investigation on the optimality of the Hardy weights given in items $5)$ and $6)$ below.
\begin{Proposition}\label{prop_hyperbolicity}
Let $M^m$ be a Riemannian manifold of dimension $m \ge 2$, $p \in (1,+\infty)$ and $f \in C^\infty(M)$. The following properties are equivalent:
\begin{itemize}
\item[$1)$] There exists a positive, non-constant solution $g \in C^0(M) \cap W^{1,p}_\loc(M)$ of $\Delta_{p,f}g \le 0$.
\vspace{0.1cm}
\item[$2)$] For some (any) compact $K \subset M$ with non-empty interior, and for some (any) solution $g$ of $\Delta_{p,f}g \le 0$, $\capac(K,g) >0$.
\vspace{0.1cm}
\item[$3)$] $Q_0$ is subcritical on $M$: there exists $w \in L^1_\loc (M)$, $w \ge 0$, $w \not\equiv 0$ on $M$ such that
\begin{equation}\label{130}
\disp \int_M w(x) |\varphi|^p \di \mu_f \le \int_M |\nabla \varphi|^p \di \mu_f \qquad \forall \, \varphi \in \lip_c (M).
\end{equation}
\item[$4)$] $Q_0$ has a weighted spectral gap on $M$: there exists $W \in C^0(M)$, $W>0$ on $M$ such that
\begin{equation}\label{131}
\disp \int_M W(x) |\varphi|^p \di \mu_f \le \int_M |\nabla \varphi|^p \di \mu_f \qquad \forall \, \varphi \in \lip_c (M).
\end{equation}
\item[$5)$] For each non-constant, positive weak solution $u \in W^{1,p}_\loc (M)$ of $\Delta_{p,f} u \le 0$ the following Hardy type inequality holds:
\begin{equation}\label{132}
\left(\frac{p-1}{p}\right)^p  \int_M \frac{|\nabla u|^p}{u^p} |\varphi|^p \di \mu_f \le \int_M |\nabla \varphi|^p \di \mu_f \qquad \forall \, \varphi \in \lip_c (M).
\end{equation}
\item[$6)$] For each $y \in M$, $-\Delta_{p,f}$ has a positive Green kernel $\green(x,y)$ and the following Hardy inequality holds:
\begin{equation}\label{weightpoincG}
\left(\frac{p-1}{p}\right)^p \int_M \frac{|\nabla_x \green|^p}{\green^p} |\varphi|^p\di\mu_f \le \int_M |\nabla \varphi|^p\di\mu_f \qquad \forall \, \varphi \in \lip_c(M).
\end{equation}
\end{itemize}
\end{Proposition}
\begin{Remark}
\emph{Here, a Green kernel $\green(x,y)$ means a distributional solution of
$$
\Delta_{p,f} \green(\cdot,y) = -\delta_y,
$$
where $\delta_y$ is the Dirac delta function at $y$.
}
\end{Remark}

%
%MI SEMBRA CHE USANDO PINCHOVER-TINTAREV, $m\ge p$ SIA AUTOMATICO PER AVERE LA $p$-HYPERBOLICITY, QUINDI QUESTA CONDIZIONE SAREBBE UGUALE ALLE ALTRE. 
%
%
\begin{proof}
The equivalence $1) \Leftrightarrow 2) \Leftrightarrow 3) \Leftrightarrow 4)$ is the particular case $V\equiv 0$ of Theorem \ref{teo_alternative}. Indeed, as any positive constant is a solution of $\Delta_{p,f}u\le 0$, $1)$ can be rephrased as the existence of two non-proportional solutions of $\Delta_{p,f}u \le 0$.
%
%
%assume $1)$. Then by definition \ref{def_pfparab} and by possibly adding a constant, there exists a positive, non trivial $u_1 \in W^{1,p}_\loc (M)$ solution of $\Delta_{p,f} u_1 \le 0$. Choosing for $u_2$ any $\alpha \in \R^+$, $u_1$ and $u_2$ are positive, non-proportional solutions of $\Delta_{p,f} w \le 0$. By Theorem \ref{teo_alternative}, $Q_0$ is subcritical. \par 
%The implication $2) \Leftrightarrow 3)$ is the case $V\equiv 0$ of $(ii) \Leftrightarrow (iii)$ in Theorem \ref{teo_alternative}, and the validity of $3)$ implies the existence of a non-constant, positive $w \in W^{1,p}_\loc (M)$ satisfying $\Delta_{p,f} w \le 0$, implying $1)$. \par
Clearly $5) \Rightarrow 3)$, thus $5)$ yields that $Q_0$ is subcritical on $M$; therefore, we just need to show that $1) \Rightarrow 5)$. Towards this aim observe that if $g \in C^0(M) \cap \in W^{1,p}_\loc (M)$, $g>0$ is a positive, non-constant solution of $\Delta_{p,f} g \le 0$, then by \eqref{113} $\disp z=g^{\frac{p-1}{p}}$ is a positive weak solution of
$$
\disp \Delta_{p,f} z +\left( \frac{p-1}{p} \right)^p \frac{|\nabla u|^p}{u^p} z^{p-1} \le 0 \qquad \text{on } M.
$$
Proposition \ref{prop_lagrangian} and the non-negativity of $\mathcal{L}$ then imply that $Q_V$ is non-negative for
% Proceeding as in the proof of $i) \Rightarrow iii)$ of Proposition \ref{pro115}, letting %$h = \log z$ with 
$$
V =\left( \frac{p-1}{p} \right)^p \frac{|\nabla u|^p}{u^p},
$$ 
which gives \eqref{132}.\par
To prove $6) \Rightarrow 1)$, take a positive Green kernel $\green(x,y)$ for $-\Delta_{p,f}$ at $y$. For large $c>0$, the function $G_c(x) = \min\{c, \green(x,y)\}$ is a non-constant, positive weak solution of $\Delta_{p,f} G_c \le 0$ by the pasting Lemma \ref{lem_pasting}, showing $1)$. Vice versa, if $1)$ holds, then by the equivalence $1) \Leftrightarrow 2)$ we deduce that each compact set has positive capacity with respect to the ``standard" supersolution $g \equiv 1$. This implies, by results in \cite{holopainen} and \cite{holopainen_3}, that for each $y \in M$ there exists a positive Green kernel $\green(x,y)$ for $-\Delta_{p,f}$. Defining again $G_c$ as above, $\Delta_{p,f}G_c \le 0$ and the equivalence $1) \Leftrightarrow 5)$ ensures that the Hardy inequality
\begin{equation}
\left(\frac{p-1}{p}\right)^p \int_M \frac{|\nabla G_c |^p}{G_c^p}|\varphi|^p\di\mu_f \le \int_M |\nabla \varphi|^p\di\mu_f \qquad \text{holds for each } \varphi \in \lip_c(M),
\end{equation}
and \eqref{weightpoincG} follows by letting $c \ra +\infty$ and using the monotone convergence theorem.
\end{proof}
\begin{Remark}
\emph{As a consequence of \cite{serrin_1, serrin_2} and Theorem 1.1 in \cite{kichenveron} (see also \cite{kichenveron_erratum}), if $M = \R^m$ each positive Green kernel $\green$ satisfies 
\begin{equation}\label{asicriticalgenerale}
\frac{|\nabla_x \green|^p}{\green^p}(x,y) \sim \left\{ \begin{array}{ll}
C\, \dist(x,y)^{-m} \log^{-m}\dist(x,y) & \quad \text{if } p=m, \\[0.3cm]
C\, \dist(x,y)^{-p} & \quad \text{if } p<m
\end{array}\right.
\end{equation}
%\dist(x,y)^{-\frac{p(m-1)}{p-1}} & \quad \text{if } p>m, \\[0.3cm]
as $\dist (x,y) \ra 0$, for some explicit $C>0$. In the linear case $p=2$, the first order expansions in \cite{serrin_1,serrin_2} guarantee the validity of \eqref{asicriticalgenerale} on each Riemannian manifold. On the contrary, when $p \neq 2$, the scaling arguments used in \cite{kichenveron} are typical of the Euclidean setting and, although we believe \eqref{asicriticalgenerale} to be true in general, to the best of our knowledge there is still no proof of \eqref{asicriticalgenerale} in a manifold setting.
}
\end{Remark}

The above proposition gives a useful, simple criterion to check the subcriticality of some $Q_V$.
\begin{Proposition}\label{prop_criterion}
Let $(M, \metric)$ be a Riemannian manifold, $f \in C^0(M)$ and $p>1$. Let $V \in L^\infty_\loc(M)$. Suppose that $Q_0$ is subcritical on $M$. If, for some Hardy weight $\hat w$, it holds $V \le \hat w$ and $V \not \equiv \hat w$, then $Q_V$ is subcritical.
\end{Proposition}

\begin{proof}
Indeed, using \eqref{130},
$$
Q_V(\varphi) \doteq \int_M |\nabla \varphi|^p \di \mu_f - \int_M V |\varphi|^p \di \mu_f \ge \int_M (\hat w-V)|\varphi|^p \di \mu_f;
$$
consequently, $w$ in \eqref{09} can be chosen to be $\hat w-V$, proving the subcriticality of $Q_V$.
\end{proof}
\begin{Remark}\label{rem_linearegreen}
\emph{The alternative in Theorem \ref{teo_alternative} is also related to the existence of a global positive Green kernel for $Q_V'$. It has been shown in \cite{pinchovertintarev} (Theorems 5.4 and 5.5) that, on $\R^m$, $Q_V$ has a weighted spectral gap if and only if $Q_V'$ admits a global positive Green kernel. The result depends on Lemma 5.1 therein, which has been proved via a rescaling argument typical of the Euclidean space, and calls for a different strategy in a manifold setting. However, in the linear case we can easily obtain the above equivalence on general manifolds by relating $Q_V'$ to a weighted Laplacian via a standard transformation. In fact, suppose that $p=2$ and that $Q_V$ is non-negative. Let $g$ be a positive solution of $Q_V'(g) = 0$. Then, setting $h= f-2\log g$, by a simple computation the following formula holds weakly for $\varphi \in C^2(M)$:
\begin{equation}\label{identitisolita}
-\Delta_{h}\left( \frac{\varphi}{g}\right) = g^{-1} Q_V'(\varphi). 
\end{equation}
Note that, according to our notation, for smooth $\phi$
$$
\Delta_{h}\phi = g^{-2}e^f \diver(e^{-f}g^2 \nabla \phi)
$$
Integrating by parts, we infer
$$
\frac{1}{2} \int_M \left|\nabla \left(\frac{\varphi}{g}\right)\right|^2 g^2 \di \mu_f = Q_V(\varphi) \qquad \forall \, \varphi \in C^2_c(M).
$$
Therefore, it is readily seen that $Q_V'$ is subcritical if and only if so is $-\Delta_{h}$ with respect to the measure $g^2\di \mu_f$. Now, Proposition \ref{prop_hyperbolicity} guarantees that this happens if and only if $-\Delta_{h}$ has a positive Green kernel $\green(x,y)$. Coming back with the aid of \eqref{identitisolita}, it is easy to see that $\green(x,y)g(x)g(y)$ is a global positive Green kernel for $Q_V'$.
}
\end{Remark}
\section{Hardy weights and comparison geometry} \label{sec_examples}
On a manifold $M$ for which $Q_0$ is subcritical, the criterion in Proposition \ref{prop_criterion} shifts the problem of the subcriticality of $Q_V$ to the one of finding explicit Hardy weights. As we will see in a moment, the construction of weights given in $4)$ and $5)$ of Proposition \ref{prop_hyperbolicity} is compatible with the usual geometric comparison theorems. Therefore, it gives a useful way to produce simple Hardy weights on manifolds satisfying suitable curvature assumptions. In this section, we describe some examples to illustrate the method.\par  
First, we underline the following simple fact. By its very definition, the set of Hardy weights $w \in L^1_\loc(M)$, $w\ge 0$, $w \not\equiv 0$ is convex in $L^1_\loc(M)$. More generally, given a family $\{w_\alpha\}_{\alpha \in A}$ of Hardy weights on $M$ whose index $\alpha$ lies in a measurable space $(A,\mathscr{F})$ ($\mathscr{F}$ a $\sigma$-algebra), and such that the map $$
w \ \ : \ \ (x,\alpha) \in M \times A \longrightarrow w_\alpha(x) \in [0,+\infty]
$$ 
is measurable, for each measure $\lambda$ on $A$ such that $0< \lambda(A) \le 1$ the function 
\begin{equation}\label{hardyweightgen}
\chi(x) = \int_A w(x,\alpha) \di \lambda(\alpha)
\end{equation}
is still a Hardy weight. Indeed, it is enough to apply Tonelli's theorem: for each $\varphi \in \lip_c(M)$,
\begin{equation}\label{hardyweight_integrated}
\begin{array}{lcl}
\disp \int_M \chi |\varphi|^p \di \mu_f & = & \disp \int_A \left[\int_M w_\alpha(x) |\varphi(x)|^p \di \mu_f\right] \di \lambda(\alpha) \\[0.4cm]
& \le & \disp \int_A \left[ \int_M |\nabla \varphi|^p \di \mu_f\right]\di \lambda(\alpha) \le \disp \int_M |\nabla \varphi|^p \di \mu_f.
\end{array}
\end{equation}
Clearly, this construction makes sense also if for $\lambda$-almost all $\alpha \in A$, $w_\alpha$ is a Hardy weight.

\begin{Remark}
\emph{By item $5)$ of Proposition \ref{prop_hyperbolicity}, each Green kernel $\green$ generates a family $w$ indexed by $A=M$:
\begin{equation}\label{lagardygreen}
w(x,y) \doteq \left( \frac{p-1}{p}\right)^p \frac{|\nabla_x \green(x,y)|^p}{\green(x,y)^p} \ \ \ : \ \ \ M \times M \longrightarrow [0,+\infty],
\end{equation}
provided that $w$ is measurable. If $p=2$, the standard construction of a Green kernel in \cite{grigoryan} produces a symmetric kernel, and measurability is obvious. However, measurability seems to be a subtle issue if $p \neq 2$, since $\green(x,y)$ is constructed in \cite{holopainen} and \cite{holopainen_3} by fixing $y$ and finding a solution of $\Delta_{p,f} \green(x,y) = -\delta_y$. The dependence of $\green(x,y)$ from $y$ could be, a priori, wild.
}
\end{Remark}

We now describe two simple measures $\lambda$ that have been considered in the literature when $M = \R^m$, and the corresponding Hardy inequalities. 
\begin{Example}\label{ex_multihardy}[Multipole Hardy weights]\\
\emph{Fix a possibly infinite sequence $\{y_j\} \subset M$, $j \in I \subset \mathbb{N}$, let $\{t_j\}_{j \in I} \subset (0,1]$ be such that $\sum_j t_j \le 1$ and define 
$$
\lambda = \sum_j t_j \delta_{y_j},
$$
where $\delta_{y_j}$ is the Dirac delta function at $y_j$. Being $\lambda$ discrete, measurability of $w$ follows automatically and thus, by \eqref{hardyweight_integrated},
\begin{equation}
\left(\frac{p-1}{p}\right)^p \int_M \left[ \sum_{j \in I} \frac{t_j|\nabla_x \green(x,y_j)|^p}{\green(x,y_j)^p}\right]\big|\varphi(x)\big|^p \di \mu_f \le \int_M \big|\nabla \varphi(x)\big|^p \di \mu_f
\end{equation}
holds for $\varphi \in \lip_c(M)$. 
}
%Note that the sequence $\{y_j\}$ may be an everywhere dense subset of $M$. In this case, the Hardy weight 
%$$
%\chi(x) \doteq \left(\frac{p-1}{p}\right)^p \sum_{j \in I} \frac{t_j|\nabla_x \green(x,y_j)|^p}{\green(x,y_j)^p},
%$$
%(an $L^1_\loc$ function and thus a.e. finite) has the properties that 
%$$
%\limsup_{x \ra y} \chi(x) = +\infty \quad \text{for each } y \in M, \qquad \lim_{x \ra y_j}\chi(x) = +\infty \quad \text{for each } j. 
%$$

\end{Example}
\begin{Example}\label{ex_blowing}[Hardy weights that blow-up along a submanifold]\\
\emph{If $w(x,y)$ in \eqref{lagardygreen} is measurable, for each rectifiable subset $\Sigma \hookrightarrow M^m$ of finite non-zero $k$-dimensional Hausdorff measure $\haus^k$ we can set 
$$
\lambda = \frac{\haus^{k}\llcorner \Sigma}{\haus^{k}(\Sigma)}, 
$$
Then, we have the Hardy inequality 
\begin{equation}
\left(\frac{p-1}{p}\right)^p \int_M \left[ \frac{1}{\haus^k(\Sigma)}\int_\Sigma \frac{|\nabla_x \green(x,y)|^p}{\green(x,y)^p}\di \haus^k(y)\right]\big|\varphi(x)\big|^p \di x \le \int_M \big|\nabla \varphi(x)\big|^p \di x 
\end{equation}
for each $\varphi \in \lip_c(M)$.\\
Hardy weights of this type have been considered, for instance, in \cite{felliterracini_multipolar}: in Theorem 1.1 therein, $M= \R^m= \R^2 \times \R^{m-2}$, $p=2$, $f \equiv 0$ and $\Sigma \subset \R^m$ is a round circle $S_\rho \doteq \mathbb{S}^1(\rho) \times \{0\} \subset \R^2\times \R^{m-2}$ centered at the origin and of radius $\rho$. 
%By a direct computation, in this case for each $x,y \in \R^m$, writing $x=(z,x') \in \R^2 \times \R^{m-2}$ we have 
%$$
%\begin{array}{rcl}
%\disp \frac{1}{4}\frac{|\nabla_x \green|^2}{\green^2} & = & \disp \left(\frac{m-2}{2}\right)^2 \frac{1}{\dist(x,y)^2}, \\[0.4cm]
%\disp \chi(x) & \doteq & \disp \left(\frac{p-1}{p}\right)^p \frac{1}{\haus^k(\Sigma)}\int_\Sigma \frac{|\nabla_x \green(x,y)|^p}{\green(x,y)^p}\di \haus^k(y) \\[0.4cm]
%& = & \disp \frac{(m-2)^2}{8 \pi \rho} \int_{S_\rho} \frac{\di \haus^1(y)}{\dist(x,y)^2} = \left(\frac{m-2}{2}\right)^2 \frac{1}{\sqrt{(\rho^2 + |x|^2)^2 - 4\rho^2|z|^2}}     \\[0.4cm]
%%&=& {\color{red} \disp\left( \frac{m-2}{2} \right)^2 \frac{1}{2\rho} \frac{1}{\dist(x,S_\rho)} }
%&=& \disp\left( \frac{m-2}{2} \right)^2 \frac{1}{\dist(x,S_\rho)\sqrt{\rho^2 + |x|^2 + 2 \rho|z|}} 
%\end{array}
%$$
%Note that 
%$$
%\chi(x) \sim \left(\frac{m-2}{2}\right)^2 \frac{1}{2\rho\dist(x, S_\rho)} \qquad \text{as } \dist(x, S_\rho) \ra 0.
%$$
We remark that Hardy weights of different type but still depending on the distance from a submanifold have been investigated, among others, in \cite{barbatisfilippastertikas}. The reader is also suggested to see the references therein for deepening.
}
\end{Example}
To introduce the results below, we first recall comparison geometry, starting with the definition of a model manifold. Briefly, fix a point $o \in \R^m$. Given $g \in C^2(\R^+_0)$ such that $g>0$ on some open interval $(0,R) \subset \R^+$, $g(0)=0$, $g'(0)=1$, a model $(M^m_g, \di s_g^2)$ is the manifold $B_R(o) \subset \R^m$ equipped with a radially symmetric $C^2$ metric $\di s_g^2$ whose expression, in polar geodesic coordinates $(\rho, \theta)$ centered at $o$ (where $\theta \in \mathbb{S}^{m-1}$), is given by 
$$
\di s_g^2 = \di \rho^2 + g(\rho)^2 \di \theta^2,
$$
$\di \theta^2$ being the standard metric on the unit sphere $\mathbb{S}^{m-1}$. Clearly, $\rho$ is the distance function from $o$, and $M_g$ is complete if and only if $R = +\infty$. At a point $x=(\rho, \theta)$, the radial sectional curvature $K_\rad$ of $M_g$ (that is, the sectional curvature restricted to planes containing $\nabla \rho(x)$), the Hessian and the Laplacian of $\rho$ are given by 
\begin{equation}\label{datimodello}
%\begin{array}{c}
\disp K_\rad (x) = - \frac{g''(\rho)}{g(\rho)}, \quad
\disp \nabla \di \rho(x) = \frac{g'(\rho)}{g(\rho)} \Big( \di s_g^2 - \di \rho \otimes \di \rho\Big), \quad \Delta \rho(x) = (m-1) \frac{g'(\rho)}{g(\rho)}. 
%\end{array}
\end{equation}
By the first formula, in \eqref{datimodello}, a model can also be given by prescribing the radial sectional curvature $-G \in C^0(\R^+_0)$ and recovering $g \in C^2(\R^+_0)$ as the solution of
\begin{equation}\label{jacobi}
\left\{ \begin{array}{l}
g''-Gg = 0 \qquad \text{on } \R^+, \\[0.2cm]
g(0)=0, \quad g'(0)=1,
\end{array}\right.
\end{equation}
on the maximal interval where $g>0$. A sharp condition on $G$ that ensures the positivity of $g$ on the whole $\R^+$ is given by 
$$
G_- \in L^1(\R^+), \qquad t \int_t^{+\infty} G_-(s) \di s \le \frac{1}{4} \qquad \text{on } \R^+,
$$
see Proposition 1.21 in \cite{bmr2}. In particular, if $G \equiv \kappa^2$ for some constant $\kappa \ge 0$, we will denote with $g_\kappa $ the solution $g$ of \eqref{jacobi}: 
%\begin{equation}
%\left\{\begin{array}{l}
%g_\kappa  '' - \kappa^2 g_\kappa  = 0 \qquad \text{on } \R^+ \\[0.2cm]
%g_\kappa (0)=0, \qquad g_\kappa '(0)=1, 
%\end{array}\right. \quad \text{that is,} \quad 
%that is, 
\begin{equation}\label{defsnh}
g_\kappa(\rho) = \left\{\begin{array}{ll}
\rho & \ \text{if } \kappa=0, \\[0.2cm]
\kappa^{-1} \sinh(\kappa\rho) & \ \text{if } \kappa>0.
\end{array}\right.
\end{equation}
When $\kappa =0$, $M_g$ is the Euclidean space $\R^m$, while if $\kappa>0$ our model is the hyperbolic space $\HH^m_\kappa$ of sectional curvature $-\kappa^2$. Clearly, the two examples are radially symmetric with respect to any chosen origin $o$. Hereafter, we will always consider geodesically complete models.\\ 
Given $p \in (1,+\infty)$, $-\Delta_p$ is subcritical on $M_g$ if and only if 
\begin{equation}\label{lanonparabol}
g(\rho)^{-\frac{m-1}{p-1}} \in L^1(+\infty) 
\end{equation}
(the case $p=2$ can be found, for instance, in \cite{grigoryan}, and for $p \neq 2$ the argument of the proof goes along the same lines). In fact, under the validity of \eqref{lanonparabol}, up to an unessential constant, the function 
$$
G(x,o) = G \big( (\rho,\theta), o\big) = \int_\rho^{+\infty} g(s)^{-\frac{m-1}{p-1}}\di s
$$
is the minimal positive Green kernel for $-\Delta_p$ with singularity at $o$.\par
A simplified version of the Hessian and Laplacian comparison theorems from below (see \cite{petersen,prs,bmr2}) can be stated as follows: suppose that $\riem$ has a pole $o$ and, denoting with $r(x)=\dist(x,o)$, that the radial sectional curvature $K_\rad$ satisfies 
\begin{equation}\label{critKrad}
K_\rad(x) \le - G\big(r(x)\big)
\end{equation}
(i.e., for each plane $\pi \le T_xM$ containing $\nabla r(x)$, $K(\pi) \le -G\big(r(x)\big)$). Denote with $g$ the solution of \eqref{jacobi}, and let $(0,R)$ be the maximal interval where $g>0$. Then, in the sense of quadratic forms,
\begin{equation}\label{hessiadasotto}
\nabla \di r(x) \ge \frac{g'(r(x))}{g(r(x))} \Big( \metric - \di r \otimes \di r\Big)
\end{equation}
pointwise on $B_R(o) \backslash \{o\}$, and tracing 
\begin{equation}\label{lapladasotto}
\Delta r(x) \ge (m-1)\frac{g'(r(x))}{g(r(x))},
\end{equation}
whose validity holds weakly on the geodesic ball $B_R(o) \subset M$. As it is apparent from \eqref{datimodello}, $M$ is compared with a model $M_g$ of radial sectional curvature $-G$. Indeed, via Sturm comparison, for \eqref{lapladasotto} to hold it is enough that $g \in C^2(\R^+_0)$ solves the inequality
\begin{equation}\label{jacobiweak}
\left\{ \begin{array}{l}
g''-Gg \le 0 \qquad \text{on } \R^+, \\[0.2cm]
g(0)=0, \quad g'(0)=1,
\end{array}\right.
\end{equation}
so that the model to which $M$ is compared has radial sectional curvature greater than or equal to $-G$. In \cite{bmr3}, we have collected a number of examples of $G$ for which explicit positive solutions $g$ of \eqref{jacobiweak} can be found. These also include manifolds whose sectional curvature can be positive but in a controlled way, loosely speaking manifolds whose compared model is some sort of paraboloid.\\
\par  
With this preparation, we are now ready to discuss the next cases.
\subsection{Hardy weights on manifolds with a pole}\par
%
%We aim to prove the following 
%
%
\begin{Theorem}\label{teo_laprimahardy}
Let $\riem$ be a Riemannian manifold with a pole $o$. Denoting with $r(x)= \dist(x,o)$, assume that the radial sectional curvature satisfies 
$$
K_\rad \le -G\big(r(x)\big), 
$$
for some $G \in C^2(\R^+_0)$. Suppose that $g$ solving \eqref{jacobiweak} is positive on $\R^+$, and, for $p \in (1,+\infty)$, assume that 
\begin{equation}\label{non-parabolicity}
g(r)^{- \frac{m-1}{p-1}} \in L^1(+\infty).
\end{equation}
Then, the Hardy inequality 
\begin{equation}\label{hardymonopole}
\int_M (\chi \circ r)|\varphi|^p \di x \le \int_M |\nabla \varphi|^p \di x 
\end{equation}
holds for each $\varphi \in \lip_c(M)$, where 
\begin{equation}\label{hardypolo}
\chi(t) = \left(\frac{p-1}{p}\right)^p \left[ g(t)^{\frac{m-1}{p-1}} \int_t^{+\infty} g(s)^{-\frac{m-1}{p-1}}\di s\right]^{-p}
\end{equation}
\end{Theorem}
%
%Let $\riem$ have a pole $o$ and radial sectional curvature satisfying \eqref{critKrad}, for some $G \in C^2(\R^+_0)$. Suppose that $g$ solving \eqref{jacobiweak} is positive on $\R^+$. For $p \in (1,+\infty)$, suppose that 
%\begin{equation}\label{non-parabolicity}
%g(r)^{- \frac{m-1}{p-1}} \in L^1(+\infty)
%\end{equation}
%(that is, that the model $M_g$ be $p$-hyperbolic) and 
\begin{proof}
Consider the transplanted Green kernel of the model $M_g$ with pole at $o$: 
\begin{equation}\label{defgreen}
G(x) = \int_{r(x)}^{+\infty} g(s)^{-\frac{m-1}{p-1}}.  
\end{equation}
Then, a computation using \eqref{hessiadasotto} shows that $\Delta_pG \le 0$ on $M \backslash \{o\}$. 

%Indeed, since
%$$
%\Delta_p G=(p-2)|\nabla G|^{p-4} \Hess G (\nabla G,\nabla G)+ |\nabla G|^{p-2} \Delta G
%$$
%and from Gauss Lemma $\Hess r (\nabla r, .)=0$, on $M \backslash \{o\}$ we obtain 
%$$
%\begin{array}{lll}
%\disp \Delta_p G  &=& (p-2) |\nabla G|^{p-4}g(r)^{-3\frac{m-1}{p-1}} \left( \frac{m-1}{p-1} \right) \frac{g'}{g}+ |\nabla G|^{p-2} \Delta G \\[0.4cm]
%\disp &=& |\nabla G|^{p-2} \left[ (p-2) \left( \frac{m-1}{p-1} \right) \frac{g'}{g} g^{- \frac{m-1}{p-1}} + \Delta G \right]\\[0.4cm]
%\disp &=&|\nabla G|^{p-2} \left[ (p-2) \left( \frac{m-1}{p-1} \right) \frac{g'}{g} g^{- \frac{m-1}{p-1}}+ \left( \frac{m-1}{p-1} \right) \frac{g'}{g} g^{- \frac{m-1}{p-1}}- g^{- \frac{m-1}{p-1}} \Delta r \right]\\[0.4cm]
%\disp &=& |\nabla G|^{p-2} g^{-\frac{m-1}{p-1}} \left[ (m-1)\frac{g'}{g}-\Delta r \right] \le 0
%\end{array}
%$$
%where the last inequality is obtained from \eqref{lapladasotto}. 
Furthermore, $G$ satisfies the asymptotic behaviour
\begin{equation}\label{asigreenzero}
G(x) \sim \left\{\begin{array}{ll} 
C >0 & \quad \text{if } p>m, \\[0.1cm]
|\log r| & \quad \text{if } p=m, \\[0.1cm]
\frac{p-1}{m-p} r^{-\frac{m-p}{p-1}} & \quad \text{if } p<m,
\end{array}\right. \qquad \text{as } r \ra 0^+,
\end{equation}
for some constant $C>0$. It thus follows that, for each $c  \in (0,C)$ if $p > m$ and $c \in \R^+$ if $p \le m$, by the pasting Lemma \ref{lem_pasting} the function $G_c = \min\{G,c\}$ is a solution of $\Delta_p G_c \le 0$, hence by \eqref{prop_hyperbolicity}
\begin{equation}
\left(\frac{p-1}{p}\right)^p \int_M \frac{|\nabla G_c |^p}{G_c^p}|\varphi|^p\di x \le \int_M |\nabla \varphi|^p\di x \qquad \text{holds for each } \varphi \in \lip_c(M).
\end{equation}
Letting $c \ra C^-$ when $p>m$, $C \ra +\infty$ when $p \le m$, and since
$$
\left(\frac{p-1}{p}\right)^p \frac{|\nabla G|^p}{G^p}(x) = \chi\big(r(x)\big),
$$
we conclude the validity of \eqref{hardymonopole}. 
\end{proof}
%\begin{equation}
%w(x) = \left(\frac{p-1}{p}\right)^p \frac{|\nabla G(x)|^p}{G(x)^p} = \left(\frac{p-1}{p}\right)^p \left[ g\big(r(x)\big)^{\frac{m-1}{p-1}} \int_{r(x)}^{+\infty} g(s)^{-\frac{m-1}{p-1}}\di s\right]^{-p}
%\end{equation}
%%
%is a Hardy weight for $\Delta_p$, 
The function $\chi \circ r$ has the following asymptotics as $r= r(x) \ra 0^+$:
\begin{equation}\label{asicritical}
\chi(r) \sim \left\{ \begin{array}{ll}
\left(\frac{p-1}{Cp}\right)^p r^{-\frac{p(m-1)}{p-1}} & \quad \text{if } p>m, \\[0.3cm]
\left(\frac{m-1}{m}\right)^m \frac{1}{r^{m} \log^{m}r} & \quad \text{if } p=m, \\[0.3cm]
\left(\frac{m-p}{p}\right)^p \frac{1}{r^p} & \quad \text{if } p<m,
\end{array}\right.
\end{equation}
where, for $p>m$, the constant $C$ is the same as in \eqref{asigreenzero}. Note that, on each case, $\chi \circ r \in L^1_\loc(M)$, in particular the singularity at $o$ is integrable.
\begin{Remark}
\emph{Under the same assumptions on $M$, this example can be extended to deal with the weighted operator $\Delta_{p,f}$ provided that $f= f(r(x))$ is a radial function. In this case, it can be seen that \eqref{non-parabolicity} must be replaced by
$$
\frac{e^{f(r)}}{g(r)^{\frac{m-1}{p-1}}} \in L^1(+\infty),
$$
and that 
$$
G(x) = \int_{r(x)}^{+\infty} \frac{e^{f(s)}\di s}{g(s)^{\frac{m-1}{p-1}}}
$$
is a solution of $\Delta_{p,f} G \le 0$ on $M \backslash \{o\}$ that gives rise to a Hardy weight analogous to \eqref{hardymonopole}.
}
\end{Remark}
The Hardy weight in \eqref{hardymonopole} is explicit once we have an explicit $g$ solving \eqref{jacobiweak}, for example those related to the families of $G$ described in the Appendix of \cite{bmr3}. 
%in Section 3.2 of \cite{bmr2} and in the Appendix of \cite{bmr3}, where such a $g$ can be found, we refer
We refer the reader to the above mentioned paper also for a detailed study of the corresponding Hardy weight (called the ``critical curve" therein) for $p=2$. The modifications needed to deal with general $p$ are straightforward. Here, we just focus on the Euclidean and hyperbolic settings.
%
%
%When $p=2$, we refer the interested reader to \eqref{bmr2. bmr3}, where we give a detailed study of the Hardy weight \eqref{hardymonopole} depending on the radial sectional curvature $G$ of the model to which $M$ is compared. The examples therein include a number of interesting situations, notably . 
%
%
%
\begin{Example}\label{ex_euclhyp}[Hardy weights for Euclidean and hyperbolic spaces]\\
\emph{Consider the Euclidean space, where $g(r) = g_0(r)=r$. Condition \eqref{non-parabolicity} is met if and only if $p<m$, and the Hardy weight \eqref{hardymonopole} has the simple  expression 
\begin{equation}\label{hardyeuclideanmono}
\chi\big(r(x)\big) = \left(\frac{m-p}{p}\right)^p \frac{1}{r(x)^p}.
\end{equation}
Consequently, the Hardy inequality
\begin{equation}\label{laclassica}
\left(\frac{m-p}{p}\right)^p \int_M \frac{|\varphi|^p}{r^p}\di x \le \int_M |\nabla \varphi|^p\di x \qquad \forall \,  \varphi \in \lip_c(M)
\end{equation}
holds on each manifold with a pole and radial sectional curvature $K_\rad \le 0$. Inequality \eqref{laclassica} is classical and well-known in $\R^m$, see \cite{barbatisfilippastertikas}.\\
On the hyperbolic space $\HH^m_\kappa$ of sectional curvature $-\kappa^2$, where $g_\kappa (r) = \kappa^{-1}  \sinh(\kappa r )$, condition \eqref{non-parabolicity} is met for each $m\ge 2$ independently of $p$, but the expression of $\chi(r(x))$ is not so neat. However, an iterative argument allows us to explicitly compute the integral in the expression of $\chi$ in some relevant cases. For $\alpha>0$, set 
$$
I_\alpha(r) \doteq \int_r^{+\infty} \frac{\di s}{g_\kappa (s)^\alpha} \qquad \text{and} \qquad \chi_\alpha(r) \doteq \left(\frac{p-1}{p}\right)^p \left[g_\kappa (r)^\alpha I_\alpha(r)\right]^{-p}.
$$
In view of \eqref{hardypolo}, our case of interest is $\alpha = \frac{m-1}{p-1}$. Writing 
$$
\frac{I_{\alpha+2}}{\kappa^{\alpha+2}} = \int_r^{+\infty} \frac{\cosh(\kappa s) \cosh(\kappa s)}{\sinh^{\alpha+2}(\kappa s)}\di s - \frac{I_\alpha}{\kappa^\alpha},
$$
and integrating by parts, we obtain the recursion formula 
$$
\alpha I_\alpha(r) = \frac{\cosh( \kappa r)}{\kappa^2 g_\kappa (r)^{\alpha+1}} - \frac{\alpha+1}{\kappa^2} I_{\alpha+2}(r)
$$
that yields
$$
\alpha \chi_\alpha^{-1/p} = \frac{p\coth( \kappa r)}{(p-1)\kappa} - \frac{\alpha+1}{\kappa^2g_\kappa (r)^2} \chi_{\alpha+2}^{-1/p}.
$$
Therefore, one may inductively recover $\chi_\alpha$. In particular, if $\alpha=1$ (i.e., in our case of interest, $p=m$), by explicit integration of $I_1(r)$ we get
\begin{equation}\label{hardy2hyp}
\chi_1(r) = \left(\frac{(m-1)\kappa}{m}\right)^m \left[\sinh(\kappa r) \log\left( \frac{e^{\kappa r}+1}{e^{\kappa r}-1} \right)\right]^{-m}, 
\end{equation}
while if $\alpha=2$ (i.e. $m = 2p-1$), again by explicit integration of $I_2(r)$ we deduce
\begin{equation}\label{hardy3hyp}
\chi_2(r) = \left( \frac{2(m-1)\kappa}{m+1}\right)^{\frac{m+1}{2}} \big(1-e^{-2 \kappa r}\big)^{-\frac{m+1}{2}},
%
%\chi_2(r) = \left( \frac{2(p-1)\kappa}{p}\right)^p \frac{1}{(1-e^{-2\kappa r})^p} = \left( \frac{2(m-1)\kappa}
%{m+1}\right)^{\frac{m+1}{2}} \big(1-e^{-2 \kappa r}\big)^{-\frac{m+1}{2}},
\end{equation}
see Example 3.15 in \cite{bmr2}. An important feature of $\chi_\alpha(r)$ is the following: 
\begin{equation}\label{soprachialfa}
\chi_\alpha(r) \ge \left(\frac{p-1}{p}\right)^p \alpha^p \kappa^p, \qquad \chi_\alpha(r) \ra \left(\frac{p-1}{p}\right)^p \alpha^p \kappa^p \quad \text{as } r \ra +\infty.
\end{equation}
Indeed, the limit is straightforwardly computable. As for the first relation, it follows from the following property. To state it, for fixed $\alpha>0$ and for $g$ satisfying $1/g^\alpha \in L^1(+\infty)$, write 
$$
\chi_{g}(r) \doteq \left(\frac{p-1}{p}\right)^p \left( g(r)^\alpha \int_r^{+\infty} \frac{\di s}{g(s)^\alpha}\right)^{-p}.
$$
Then, the next comparison result holds: 
\begin{quote}
if $g_1/g_2$ is non-decreasing on $\R^+$ (respectively, non-increasing), \\[0.2cm]
then $\chi_{g_1} \ge \chi_{g_2}$ on $\R^+$ (resp, $\le$).
\end{quote}
The proof of this fact goes along the same lines as in Proposition 3.12 in \cite{bmr2}, and is left to the interested reader. Using this with $g_1(r) \doteq g_\kappa (r)$ and $g_2(r) \doteq \exp\{\kappa r\}$ we get 
$$
\chi_{g_1}(r) \ge \chi_{g_2}(r) \equiv \left(\frac{p-1}{p}\right)^p \alpha^p \kappa^p,
$$
as claimed.
}
\end{Example}
When each point of the manifold $M$ is a pole, we can construct multipole Hardy weights by the standard procedure described at the beginning of Section \ref{sec_examples}. This is the case if, for example, $M$ is a Cartan-Hadamard manifold, that is, a simply-connected, complete manifold with non-positive sectional curvature.

\begin{Theorem}\label{multihardyCH}
For $m \ge 2$, let $M^m$ be a Cartan-Hadamard manifold satisfying $K \le -\kappa^2$, for some constant $H \ge 0$. Given the solution $g_\kappa $ of \eqref{defsnh}, let 
\begin{equation}\label{condiiperboli}
\left\{ \begin{array}{ll}
p \in (1,m) & \quad \text{if } \kappa =0, 	\\[0.1cm]
p \in (1,+\infty) & \quad \text{if } \kappa >0.
\end{array}\right.
\end{equation}
Then, for each unit measure $\lambda$ on $M$, the Hardy inequality
\begin{equation}\label{multihardy}
\int_M \left[ \int_M\big(\chi \big(\dist(x,y)\big)\big)\di \lambda(y)\right]\big|\varphi(x)\big|^p \di x \le \int_M \big|\nabla \varphi(x)\big|^p \di x 
\end{equation}
holds for each $\varphi \in \lip_c(M)$, where 
\begin{equation}\label{hardypolipolo}
\chi(t) = \left(\frac{p-1}{p}\right)^p \left[ g_\kappa (t)^{\frac{m-1}{p-1}} \int_t^{+\infty} g_\kappa (s)^{-\frac{m-1}{p-1}}\di s\right]^{-p}.
\end{equation}
In particular, when $\kappa=0$, for each $\varphi \in \lip_c(M)$
\begin{equation}\label{multihardyeucli}
\left(\frac{m-p}{p}\right)^p \int_M \left[ \int_M \frac{\di \lambda(y)}{\dist(x,y)^p}\right]\big|\varphi(x)\big|^p \di x \le \int_M \big|\nabla \varphi(x)\big|^p \di x. 
\end{equation}
\end{Theorem}
\begin{proof}
By \eqref{condiiperboli}, \eqref{non-parabolicity} is met for $g=g_\kappa $. Since each $y \in M$ is a pole of $M$, setting $r_y(\cdot) = \mathrm{dist}(\cdot, y)$ one can apply Theorem \ref{teo_laprimahardy} to deduce that \eqref{hardymonopole} holds for each fixed $y$, namely, 
\begin{equation}
\int_M (\chi \circ r_y)|\varphi|^p \di x \le \int_M |\nabla \varphi|^p \di x \qquad \forall \, \varphi \in \lip_c(M),
\end{equation}
where $\chi$ is as in \eqref{hardypolipolo}. The generalization in \eqref{multihardy} follows from the argument at the beginning of Section \ref{sec_examples}.
\end{proof}

\subsection{Hardy weights on minimally immersed submanifolds}\label{ex3sub}

\begin{Theorem}\label{teo_hardyminimal}
%
%
%Let $f : (M^m, \metric) \ra (N^n, ( \, , \, ) )$ be an immersed, minimal submanifold of a Cartan-Hadamard ambient space $N^n$. Suppose that $m \ge 3$. Denoting with $\bar r_q$ the extrinsic distance from a point $q \in N$ evaluated along the immersion $f$, the following Hardy inequality holds:
%\begin{equation}\label{hardyeuclbella}
%\left(\frac{m-2}{2}\right)^2 \int_M \frac{\varphi^2}{(\bar r_q \circ f)^2} \di x \le \int_M |\nabla \varphi|^2 \di x.
%\end{equation}
%\end{Theorem}

%
%\begin{Theorem}\label{teo_hardybellissime}
%
Let $F : (M^m, \metric) \ra (N^n, ( \, , \, ) )$ be an immersed, minimal submanifold of a Cartan-Hadamard ambient space $N^n$, and suppose that the sectional curvature $\bar K$ of $N$ satisfies $\bar K \le -\kappa^2$, for some constant $H \ge 0$. If $\kappa =0$, we assume that $m\ge 3$. Then, given the solution $g_\kappa $ of \eqref{defsnh}, 
%suppose that 
%\begin{equation}\label{parametrinonparab}
%\left\{\begin{array}{ll}
%p \in (1,m) & \quad \text{if } H=0, \\[0.1cm]
%p \in (1,+\infty) & \quad \text{if } H>0.
%\end{array}\right.
%\end{equation}
and denoting with $\bar r_q$ the extrinsic distance from a point $q \in N$ evaluated along the immersion $F$, the following Hardy inequality holds:
\begin{equation}\label{hardysubmanifold}
\int_M (\chi \circ \bar r_q) \varphi^2 \di x \le \int_M |\nabla \varphi|^2 \di x \qquad \forall \, \varphi \in \lip_c(M), 
\end{equation}
where 
\begin{equation}\label{hardypolodim2}
\chi(t) = \frac{1}{4}\left[ g_\kappa (t)^{m-1} \int_t^{+\infty} g_\kappa (s)^{1-m}\di s\right]^{-2}.
\end{equation}
In particular, if $\kappa=0$, 
\begin{equation}\label{hardyeuclbella}
\left(\frac{m-2}{2}\right)^2 \int_M \frac{\varphi^2}{\bar r_q^2} \di x \le \int_M |\nabla \varphi|^2 \di x, \qquad \forall \, \varphi \in \lip_c(M),
\end{equation}
while, if $\kappa >0$, 
\begin{itemize}
\item[-] when $m=2$, that is, $M$ is a surface, for each $\varphi \in \lip_c(M)$
\begin{equation}\label{hardy2hypbella}
\left(\frac{\kappa}{2}\right)^2 \int_M\left[\sinh(\kappa \bar r_q) \log\left( \frac{e^{\kappa\bar r_q}+1}{e^{\kappa \bar r_q}-1} \right)\right]^{-2} \varphi^2 \di x \le \int_M |\nabla \varphi|^2 \di x; 
\end{equation}
\item[-] when $m=3$, for each $\varphi \in \lip_c(M)$
\begin{equation}\label{hardy3hypbella}
\kappa^2 \int_M \frac{\varphi^2}{(1-e^{-2\kappa \bar r_q})^2} \di x \le \int_M |\nabla \varphi|^2 \di x.
\end{equation}
\end{itemize}
\end{Theorem}

\begin{Remark}
\emph{The inequality in \eqref{hardyeuclbella} has been proved in \cite{carron, liwang3} by  combining the comparison for the Hessian of the extrinsic distance function with an integration by parts argument. The case $\kappa>0$ has been considered in \cite[Example 1.8]{liwang3}. However, the Hardy weight found there is skew with \eqref{hardy2hypbella} and \eqref{hardy3hypbella}, in particular it is quite smaller if $\bar{r}_q$ is close to zero. As before, other Hardy weights can be constructed in the way described in Examples \eqref{ex_multihardy} and \eqref{ex_blowing}.
} 
\end{Remark}
\begin{proof}
We mark with a bar each quantity when referred to $N$, so that, for example, $\bar \nabla,  \overline{\dist}$ are the Riemannian connection and the distance function of $N$. For simplicity, we denote with $\bar r_q$ the distance function from $q$ in the manifold $N$, i.e. $\bar r_q(\cdot)  = \overline{\dist}(\cdot, q)$, so that the function $\bar r_q$ in the statement of the theorem is, indeed, $\bar r_q \circ F$. By the Hessian comparison theorem \eqref{hessiadasotto}, for each $q \in N$ it holds
$$
\bar \nabla \di \bar r_q \ge \frac{g_\kappa '(\bar r_q)}{g_\kappa (\bar r_q)}\Big( ( \, , \, ) - \di \bar r_q \otimes \di \bar r_q\Big) \qquad \text{on } N \backslash \{q\}.
$$
For $x\in M$, $q \in N$ define 
\begin{equation}\label{pfam2}
G_q(x) = h\big( \bar r_q\big(F(x)\big) \big), \qquad \text{where } h(t) = \int_{t}^{+\infty} \frac{\di s}{g_\kappa (s)^{m-1}}.
\end{equation}
Observe that the integral in $h(t)$ converges for each $m \ge 2$ when $\kappa>0$, and for each $m \ge 3$ if $\kappa=0$, which accounts for our dimensional restrictions. Denote with $\II$ the second fundamental form of $F$. By the chain rule and using $h'<0$, for each vector field $X$ on $M$ (identified with $F_*(X))$ we get
$$
\begin{array}{lcl}
\disp \nabla\di G_q(X,X) & = &  \disp \disp h''\big( \bar \nabla \bar r_q, X \big)^2 + h' \bar \nabla \di \bar r_{q} (X,X) + h' \big( \bar \nabla \bar r_{q}, \II(X,X)\big) \\[0.2cm]
& \le & \disp \disp \left(h''- \frac{g_\kappa '}{g_\kappa }h'\right)\big( \bar \nabla \bar r_q, X \big)^2 + h' \frac{g_\kappa '}{g_\kappa } |X|^2 + h'\big( \bar \nabla \bar r_{q}, \II(X,X)\big) \\[0.4cm]
& = & \disp \disp - mh'\frac{g_\kappa '}{g_\kappa} \big( \bar \nabla \bar r_q, X \big)^2 + h' \frac{g_\kappa '}{g_\kappa } |X|^2 + h'\big( \bar \nabla \bar r_{q}, \II(X,X)\big) \\[0.4cm]
& = & \disp \disp h'\frac{g_\kappa '}{g_\kappa}\Big[ |X|^2 - m \big( \bar \nabla \bar r_q, X \big)^2\Big] + h'\big( \bar \nabla \bar r_{q}, \II(X,X)\big) 
\end{array}
$$
Tracing with respect to an orthonormal frame $\{e_i\}$ of $M$, using minimality and again $h'<0$, we obtain
\begin{equation}\label{laplagmin}
\disp \Delta G_q \le mh'\frac{g_\kappa '}{g_\kappa } \left( 1- |\bar \nabla^T \bar r_{q}|^2\right) \le 0
\end{equation}
where $\bar \nabla^T$ is the component of the gradient in $N$ which is tangent to $M$. By Proposition \ref{prop_hyperbolicity},
\begin{equation}\label{hardyfraco}
\frac{|\nabla G_q(x)|^2}{4G_q(x)^2} = \left(\frac{h'(\bar r_q)}{2h(\bar r_q)}\right)^2|\bar \nabla^T \bar r_q|^2 = \chi \big(\bar r_q\big(F(x)\big)\big)|\bar \nabla^T \bar r_q|^2.
\end{equation}
is a Hardy weight. Unfortunately, such a weight is not effective, since we cannot control the size of $\bar \nabla^T \bar r_q$. However, we can improve \eqref{hardyfraco} to the  effective Hardy weight $\chi(\bar r_q)$ by using the full information coming from \eqref{laplagmin}. 
%According to \cite{bmr2}, we set
%\begin{equation}\label{defcriticalcurve}
%\chi_H(t) = \left(\frac{h'(t)}{2h(t)}\right)^2 = \left[ \left(- \frac 12 \log \int_t^{+\infty} \frac{\di s}{g_\kappa (s)^{m-1}}\right)'\right]^2.
%\end{equation}
%Then, b
In fact, since $h'<0$, by \eqref{laplagmin} the function $u_1 \doteq \sqrt{G_q}$ solves, on $M \backslash f^{-1}\{q\}$,
\begin{equation}\label{bellarray}
\begin{array}{lcl}
\Delta u_1 & = & \disp \left[\frac{\Delta G_q}{2G_q} - \chi(\bar r_q) |\bar \nabla^T \bar r_q|^2\right] u_1  \le \left[ m \frac{h'}{2h} \frac{g_\kappa '}{g_\kappa } \big( 1-|\bar \nabla^T \bar r_q|^2\big) - \chi(\bar r_q) |\bar \nabla^T \bar r_q|^2\right] u_1 \\[0.5cm]
& \le & \disp \left[ - m \sqrt{\chi(\bar r_q)} \frac{g_\kappa '}{g_\kappa } + \sqrt{\chi(\bar r_q)}|\bar \nabla^T \bar r_q|^2 \left( m \frac{g_\kappa '}{g_\kappa } - \sqrt{\chi(\bar r_q)} \right) \right] u_1.
\end{array} 
\end{equation}
Now, we claim that 
\begin{equation}\label{ode}
\zeta(t) \doteq m \frac{g_\kappa '(t)}{g_\kappa (t)} - \sqrt{\chi(t)}  \ge 0 \qquad \text{for } t \in \R^+.
\end{equation}
Indeed, in the Euclidean case $\kappa=0$, $g_\kappa(t)=t$ and explicit computation gives
$$
\zeta(t) = \frac{m+2}{2t} >0 \qquad \text{on } \R^+.
$$
When $\kappa>0$, $g_\kappa(t) = \kappa^{-1} \sinh(\kappa t)$. A computation gives
\begin{equation}\label{asymptotici}
\zeta(t) \sim \frac{m+2}{2t} \quad \text{as } \, t \ra 0, \qquad \zeta(t) \sim \frac{(m+1)\kappa}{2} \quad \text{as } \, t\ra + \infty.
\end{equation}
Now, $y(t) = \sqrt{\chi(t)}$ solves
\begin{equation}\label{ode2}
y' = 2y^2 - y(m-1) \frac{g_\kappa '}{g_\kappa } \qquad \text{on } \R^+, 
\end{equation}
%so that $y'\ge 0$ whenever 
%$$
%y \ge \frac{m-1}{2}\frac{g_\kappa '}{g_\kappa }.
%$$ 
Suppose that $\zeta(\bar t) \le 0$ for some $\bar t>0$. Then, an inspection of \eqref{ode2} and the fact that $g_\kappa '/g_\kappa $ is decreasing show that $y'>0$ on $[\bar t, +\infty)$, whence there exists $c>0$ such that 
$$
y> \frac{m-1}{2}\frac{g_\kappa '}{g_\kappa } + c  \qquad \text{on } [\bar t, +\infty). 
$$
But then
$$
\lim_{t \ra +\infty} \zeta(t) = m \kappa - \lim_{t \ra +\infty} y(t) \le m \kappa -c -\frac{m-1}{2}\kappa = \frac{m+1}{2}\kappa -c,
$$
contradicting \eqref{asymptotici} and proving the claim. Next, by \eqref{ode} we can use the estimate $|\bar \nabla^T \bar r_q| \le 1$ to conclude
\begin{equation}\label{bella!!!}
\Delta u_1 \le \left[ - m \sqrt{\chi(\bar r_q)} \frac{g_\kappa'}{g_\kappa} + \sqrt{\chi(\bar r_q)} \left( m \frac{g_\kappa'}{g_\kappa} - \sqrt{\chi(\bar r_q)} \right) \right] u_1  = - \chi(\bar r_q) u_1.
\end{equation} 
For $\eps>0$, consider the truncated function $\chi_{\eps}(t)$ given by $\chi_{\eps}(t) = \chi(t)$ if $t \ge 2\eps$, and $0$ otherwise. Then, $\chi_{\eps}(\bar r_q \circ F) \in L^\infty_\loc(M)$ and by \eqref{bella!!!}
\begin{equation}\label{bellissima}
\Delta u_1 + \chi_{\eps}(\bar r_q)u_1 \le 0 \qquad \text{on } M \backslash F^{-1}\{q\}.
\end{equation}
If $F^{-1}\{q\} \neq \emptyset$ observe that the constant function $u_2 \doteq \sqrt{G_q(\eps)}$ solves \eqref{bellissima} on $F^{-1}\{B_{2\eps}(q)\}$. By the pasting Lemma \ref{lem_pasting} with the choices $\Omega_1 \doteq F^{-1}\{B_\eps(q)\}\backslash F^{-1}\{q\}$, $\Omega_2 \doteq F^{-1}\{B_{2\eps}(q)\}\backslash F^{-1}\{q\}$, and since $h'<0$, we deduce that 
$$
u \doteq \left\{ \begin{array}{ll}
\sqrt{G_q(\eps)} & \quad \text{on } F^{-1}\{B_\eps(q)\},\\[0.2cm]
\sqrt{G_q} & \quad \text{on } M \backslash F^{-1}\{B_\eps(q)\}
\end{array}\right.
$$
solves $\Delta u + \chi_{\eps}(\bar r_q)u \le 0$ on $\Omega_2=F^{-1}\{B_{2\eps}(q)\}\backslash F^{-1}\{q\}$. Since the pasting region $F^{-1}\{\partial B_\eps(q)\}$ is internal to $\Omega_2$ and $u$ is smooth in a neighbourhood of $F^{-1}\{q\}$, then clearly $u$ solves
\begin{equation}\label{ancorameglio}
\Delta u + \chi_{\eps}(\bar r_q)u \le 0 \qquad \text{on the whole } M.
\end{equation}
By Proposition \ref{pro115} with $p=2$, $f=1$ and $V = \chi_{\eps}(\bar r_q)$  we deduce that $\lambda_{V}(M) \ge 0$, that is 
$$
\int_M \chi_{\eps}(\bar r_q \circ F) \varphi^2 \di x \le \int_M |\nabla \varphi|^2 \di x \qquad \text{for each } \varphi \in \lip_c(M), 
$$
whence, letting $\eps \ra 0$ and using monotone convergence we deduce \eqref{hardysubmanifold}. The cases \eqref{hardyeuclbella}, \eqref{hardy2hypbella}, \eqref{hardy3hypbella} follow by computing $\chi(r)$ according to Example \ref{ex_euclhyp}, in particular see \eqref{hardyeuclideanmono}, \eqref{hardy2hyp}, \eqref{hardy3hyp}. 
\end{proof}
%}
%
%\end{Example}
%
%

\subsection{Specializing our main theorems: an example}
To illustrate the results of the last two sections, by way of an example we specialize our Theorem \ref{teouno} to the case of quasilinear Yamabe type equations on Cartan-Hadamard manifolds. We underline that all the assumptions in the next Corollary are explicit and easy to check. Clearly, analogous results can be stated using any of Theorems \ref{teo_laprimahardy}, \ref{multihardyCH}, \ref{teo_hardyminimal}, as well as Theorem \ref{teouno_bis} instead of Theorem \ref{teouno}.
\begin{Corollary}\label{cordue}
Let $M$ be a Cartan-Hadamard manifold of dimension $m \ge 3$, and let $p \in (1,m)$. Let $a,b \in L^\infty_\loc(M)$, and suppose that, for some countable set of points $\{y_j\}_{j \in I} \subset M$ and $\{t_j\}_{j\in I}\subset [0,1]$ with $\disp \sum_j t_j \le 1$,
\begin{equation}\label{013}
a(x) \le \left(\frac{m-p}{p}\right)^p \,  \sum_{j \in I}\frac{t_j}{\dist(x,y_j)^p}.
\end{equation}
Furthermore, suppose that 
\begin{itemize}
\item[$i)$] $b_-(x)$ has compact support;
\vspace{0.1cm}
\item[$ii)$] $a(x)=O\big(b(x)\big)$ as $x$ diverges
\vspace{0.1cm}
\item[$iii)$] for some $\theta>0$, $\big(a(x)-\theta b_+(x) \big)_- \in L^1(M)$.
\end{itemize}
Fix a nonlinearity $F(t)$ satisfying \eqref{assu_F}. Then, there exists $\delta>0$ such that if
$$
b(x) \ge -\delta \qquad \text{on }M,
$$
there exists a weak solution $u \in C^{1,\mu }_\loc(M)$ of
\begin{equation} \label{014}
\left\{\begin{array}{l}
\disp \Delta_p u + a(x) u^{p-1} - b(x) F(u) = 0 \qquad \text{on } M,\\[0.2cm]
0 < u \le \|u\|_{L^\infty(M)} < +\infty.
\end{array}\right.
\end{equation}
\end{Corollary}
\begin{proof}
In view of Theorem \ref{teouno}, it is enough to prove that $-\Delta_p$ and $Q_a$ are subcritical on $M$. By Theorem \ref{multihardyCH}, the Hardy inequality
$$
\left(\frac{m-p}{p}\right)^p \,   \int_M  \sum_{j \in I} \frac{t_j}{\dist(x,y_j)^p}|\varphi(x)|^p \di x \le \int_M |\nabla \varphi(x)|^p \di x
$$
holds for $\varphi \in \lip_c(M)$. Therefore, by Proposition \ref{prop_hyperbolicity}, $-\Delta_p$ is subcritical. Next, the fact that the inequality in \eqref{013} is strict on a set of positive measure (the right-hand side is essentially unbounded at each $y_j$) assures that $Q_a$ be subcritical by Proposition \ref{prop_criterion}. 
\end{proof}
\section{Proofs of Theorems \ref{teouno} and \ref{teouno_bis}} \label{esistenza}

We first address the local solvability of the Dirichlet problem for 
$$
\Delta_{p,f}u + A(x) u^{p-1} -B(x)F(u) = 0
$$
when $B \ge 0$.
\begin{Lemma}\label{lem-diriproblem}
Let $M$ be a Riemannian manifold, $p \in (1,+\infty)$ and $f \in C^\infty(M)$. Let $\Omega \Subset M$ be a smooth relatively compact open set, and let $A(x),B(x) \in L^\infty(\Omega)$ satisfy $\lambda_A(\Omega)>0$ and $B \ge 0$ a.e. on $\Omega$.
Then, for each nonlinearity $F(t)$ that satisfies \eqref{assu_F}, and for each $\varphi \in C^{1,\alpha}(\partial \Omega)$, $\alpha \in (0,1)$ such that $\varphi \ge 0$, $\varphi \not \equiv 0$, there exist $\mu \in (0,1)$ and a unique $0<z \in C^{1,\mu}(\overline \Omega)$ solving
\begin{equation}\label{localYamabe0}
\left\{ \begin{array}{ll}
\Delta_{p,f} z + A(x)z^{p-1} - B(x)F(z) = 0 & \quad \text{on } \Omega, \\[0.2cm]
z = \varphi & \quad \text{on } \partial \Omega.
\end{array}\right.
\end{equation} 
\end{Lemma}
\begin{proof}
Take the positive solution $z_0 \in C^{1,\mu}(\overline \Omega)$ of 
\begin{equation}\tag{$P_0$}
\left\{ \begin{array}{ll}
\disp \Delta_{p,f} z_0 + A|z_0|^{p-2}z_0= 0 & \quad \text{on } \Omega, \\[0.2cm]
z_0 = \varphi & \quad \text{on } \partial \Omega,
\end{array}\right. 
\end{equation}
which exists by Proposition \ref{pro115}. Since $B \ge 0$, $z_0$ gives rise to a supersolution for \eqref{localYamabe0}. To construct a subsolution with boundary value $\varphi$, we solve
\begin{equation}\tag{$P_n$}
\left\{ \begin{array}{ll}
\disp \Delta_{p,f} z_1 + (A - B_1)|z_1|^{p-2}z_1= 0 & \quad \text{on } \Omega, \disp \qquad B_1 \doteq B\frac{F(z_0)}{z_0^{p-1}}\\[0.2cm]
z_1 = \varphi & \quad \text{on } \partial \Omega.
\end{array}\right. 
\end{equation}
Since $A-B_1 \le A$, $\lambda_{A-B_1}(\Omega)>0$, thus $z_1$ exists and, by comparison, $z_1 \le z_0$. Since $F(t)/t^{p-1}$ is increasing,
$$
\disp \Delta_{p,f} z_1 + \left(A - B\frac{F(z_1)}{z_1^{p-1}} \right)z_1^{p-1} \ge \disp \Delta_{p,f} z_1 + \left(A - B\frac{F(z_0)}{z_0^{p-1}} \right)z_1^{p-1} = 0. 
$$
Hence $z_1$ is a subsolution for \eqref{localYamabe0}. Applying the subsolution-supersolution method (see \cite{diaz}, Theorem 4.14 page 272) we obtain the existence of $z$ satisfying \eqref{localYamabe0}. The local Harnack inequality in Theorem \ref{teo11} gives $z>0$ on $\overline \Omega$, and uniqueness follows from the comparison Proposition \ref{prop_compagen}.
\end{proof}

Next, we investigate the existence of local uniform lower bounds for solutions of the Dirichlet problem when $\Omega$ varies.
\begin{Lemma}\label{lem_uniformP}
Let $(M, \metric)$ be a Riemannian manifold, $f \in C^\infty (M)$, $p \in (1,+\infty)$, and assume that $Q_0$ is subcritical on $M$. Let $A,B \in L^\infty_\loc(M)$ such that 
\begin{equation}\label{22}
\left\{ \begin{array}{l}
Q_A \, \text{ is non-negative}, \\[0.2cm]
(A(x) - \theta B(x)\big)_- \in L^1(M, \di \mu_f), \qquad \text{for some constant } \, \theta>0. 
\end{array}\right.
\end{equation}
Fix $\eps >0$. Consider a triple of relatively compact, open sets $\Lambda \Subset \Lambda' \Subset \Omega \Subset M$, with $\partial\Omega$  smooth, and let $z\in C^{1, \mu } (\overline{\Omega})$ be a positive solution of
\begin{equation}\label{23}
\left\{ \begin{array}{ll}
\Delta_{p,f} z + A(x)z^{p-1} - B(x)F(z) \le 0 & \quad \text{on } \Omega; \\[0.2cm]
z= \eps & \quad \text{on } \partial \Omega.
\end{array}\right.
\end{equation}
Then, there exists $C>0$ depending on $\eps$, but independent of $\Omega$, such that
\begin{equation}\label{24}
\inf_{\Lambda} z \ge C. 
\end{equation}
\end{Lemma}
\begin{proof}
Observe that, by the half-Harnack inequality in Theorem \ref{teo11}, $z>0$ on $\Omega$. By comparison, we can also suppose that $z$ solves \eqref{23} with the equality sign, whence $z \in C^{1,\mu}(\overline\Omega)$. Fix $\delta \in (0,\eps)$ small enough that
\begin{equation}\label{sottoF_11}
\frac{F(t)}{t^{p-1}} < \theta \qquad \text{if } \, t \in (0,\delta).
\end{equation}
This is possible by \eqref{assu_F}. Let $\eta \in \lip_\loc (\R)$, $\eta(\log\eps)=0$ to be specified later, and define $u = \log z$ on $\Omega$, so that, weakly,
\begin{equation}\label{26}
\left\{ \begin{array}{lr}
\disp\Delta_{p,f}u = -A +B \frac{F(z)}{z^{p-1}} -(p-1)|\nabla u|^p  &{\rm on} \ \Omega \\[0.2cm]
\disp u = \log \eps & {\rm on} \ \partial\Omega.
\end{array}
\right.
\end{equation}
The function $\eta(u)\in \lip_0(\Omega)$ can be used as a test function for \eqref{26} to obtain
\begin{equation}\label{27}
\disp \int_{\Omega}\big[(p-1)\eta(u)-\eta'(u)\big]|\nabla u|^p\di \mu_f = \int_{\Omega} \left[-A + B\frac{F(z)}{z^{p-1}}\right]\eta(u)\di \mu_f.
\end{equation}
Let $L(\delta) \doteq \{x: u(x) < \log\delta\}$. Choose 
$$
\eta(t) = \left[1-\frac{e^{(p-1)t}}{\delta^{p-1}}\right]1_{(-\infty, \log\delta)}(t) \in \lip(\R).
$$
Then, from \eqref{27} and \eqref{22}, and since $\delta \in (0,\eps)$, $\eta(\log\eps)=0$. Plugging in \eqref{27} and using \eqref{sottoF_11} we deduce
\begin{equation}\label{28}
\int_{L(\delta)}|\nabla u|^p \di \mu_f \le \frac{1}{p-1}\int_{L(\delta)}\left[-A + B\frac{F(z)}{z^{p-1}}\right]\eta(u)\di \mu_f \le \frac{1}{p-1}\|(A - \theta B)_-\|_{L^1(M, \di \mu_f)}
\end{equation}
Since $Q_0$ is subcritical on $M$, by Proposition \ref{prop_hyperbolicity} there exists $W \in C^0(M)$, $W>0$ on $M$ such that 
\begin{equation}\label{29}
\int_M W|\varphi|^p\di \mu_f \le \int_M |\nabla \varphi|^p \di \mu_f \qquad \forall \,  \varphi \in \lip_c(M).
\end{equation}
Now the function $\varphi(x) = \big(u(x) -\log\delta\big) 1_{L(\delta)}(x)$, extended with $0$ outside $\Omega$, is an admissible test function for \eqref{29} and with this choice of $\varphi$ we obtain
\begin{equation}\label{210}
\disp \int_{L(\delta)} |\nabla u|^p\di \mu_f \ge \int_{L(\delta)}W|u-\log \delta|^p \di \mu_f. 
\end{equation}
By contradiction assume that there exists a sequence of relatively compact open sets $\{\Omega_j\}$ with smooth boundary such that $\Lambda' \Subset \Omega_j$ and a sequence of associated solutions $\varphi_j$ of \eqref{23}, such that $\inf_\Lambda \varphi_j \ra 0^+$ as $j \ra +\infty$. Using Harnack inequality of Theorem \ref{teo11} $3)$, $\phi_j \ra 0$ uniformly on $\Lambda$ as $j \to +\infty$ (note that, to infer $\phi_j \ra 0$, we need that each $\Omega_j$ contains a fixed domain larger than $\Lambda$, which accounts for the presence of $\Lambda'$). Having fixed $N>0$, we choose $j_0$ large enough that $u_j=\log \varphi_j<-N + \log\delta$ on $\Lambda$ when $j \ge j_0$. Consequently, $\Lambda \subset \{u_j < \log \delta\} \doteq L_j(\delta)$, and from \eqref{28} and \eqref{210} we deduce
$$
\begin{array}{lcl}
\disp N^p \int_{\Lambda}W\di \mu_f & \le & \disp \int_{L_j(\delta)} W|u_j-\log\delta|^p\di \mu_f \le \int_{L_j(\delta)} |\nabla u_j|^p\di \mu_f \\[0.4cm]
& \le & \disp \frac{1}{p-1}\|(A-\theta B)_-\|_{L^1(M, \di \mu_f)}.
\end{array}
$$
This gives a contradiction provided $N$ is large enough, proving the validity of \eqref{24}.
\end{proof}

\begin{Remark}
\emph{To guarantee \eqref{24}, the second condition in \eqref{22} cannot be relaxed too much. To see this, let us suppose that $B \equiv 0$, for which the second in \eqref{22} reads $A_- \in L^1 (M, \di \mu_f)$. We first observe that the validity of \eqref{24} is granted provided that there exists a positive, bounded solution of $Q'_A (u)=0$ on $M$. Indeed, if such a $u$ exists, comparing a solution $z$ of \eqref{23} with $\eps u/\|u\|_{L^\infty(M)}$ with the aid of Proposition \ref{prop_compagen} yields \eqref{24}. If, on the other hand, there exists a positive solution of $Q_A'(u) =0$ with $u(x) \ra +\infty$ as $x$ diverges, \eqref{24} fails: indeed, if we consider a smooth exhaustion $\{\Omega_j\}$ with $\partial \Omega_j \subset \{u \in [j, j+1)\}$, we compare  the solution of \eqref{23} on $\Omega_j$ (with the equality sign) and the function $\eps u/j$, and we let $j \ra 0$, it is easy to see that the left-hand side of \eqref{24} is zero.\\ 
%
%Let us suppose that $B \equiv 0$. Then, the second condition in \eqref{22} becomes $A_- %\in L^1 (M, \di \mu_f)$, which is related to the existence of a positive bounded solution %$u \in C^{1,\mu}_\loc (M)$ of $Q'_A (u)=0$ on $M$.\\ 
%Indeed, suppose that  such a $u$ exists. Up to multiplication by a positive constant, we can assume that $u \le \eps$ on $M$. Let $z \in C^{1,\mu} (\overline{\Omega})$ be a solution of \eqref{23}. Then, by the comparison in Proposition \ref{prop_compagen}, $u \le z$ on $\Omega$ and this clearly implies \eqref{24}. On the other hand if $u \in C^{1,\mu}_\loc (M)$ is a solution of $Q'_A(u)=0$ such that $u(x) \ra +\infty$ as $x$ diverges, we choose smooth $\Omega_j$ such that
%$$
%u^{-1} \{(0,2j)\} \Subset \Omega_j \Subset u^{-1} \{0,3j\}
%$$ 
%and let $z_j \in C^{1,\mu_j}(\overline{\Omega_j})$ be a solution of
% $$
% \left\{
% \begin{array}{ll}
% \disp Q'_A(z_j) = 0 & \quad {\rm on} \ \Omega_j \\[0.2cm]
% \disp z_j = 1 & \quad {\rm on} \ \partial\Omega_j
% \end{array}\right.
% $$
%Note that $u/j$ satisfies
%$$
% \left\{
% \begin{array}{ll}
% \disp Q'_A\left(u/j\right) \ge 0 & \quad {\rm on} \ \Omega_j \\[0.2cm]
% \disp u/j \ge 2 & \quad {\rm on} \ \partial\Omega_j
% \end{array}
% \right.
% $$
%hence, again by comparison $z_j \le u/j$ on $\Omega_j$. Letting $j \ra +\infty$, $z_j \ra 0$ locally uniformly on $M$, preventing the validity of \eqref{24}.\\ 
%
Now, consider the case $p=2$, $f \equiv 0$. In Theorem 3 of \cite{bmr3} we have shown that, if $M$ is a manifold with a pole $o$, dimension $m \ge 3$ and radial sectional curvature $K_{\rm rad} \le 0$, and for $A \in L^\infty _\loc (M)$, conditions
\begin{equation}\label{221}
\left\{\begin{array}{l}
\disp |A(x)| \le \bar A\big(r(x)\big) \le \frac{(m-2)^2}{4} \frac{1}{r(x)^2} \qquad \text{on } M, \\[0.3cm]
\disp r \bar A(r) \in L^1(+\infty),
\end{array} \right.
\end{equation}
imply the existence of a positive bounded solution $u \in C^{1,\mu}_\loc(M)$ of $Q_A'(u)=0$. Thus, in this case \eqref{24} holds true. On the other hand, by Theorem 11 of \cite{bmr3} if, for $A  \in L^\infty_\loc (\R^m)$ 
\begin{equation}\label{223}
\left\{\begin{array}{l}
\disp \bar A_1\big(r(x)\big) \le A(x)  \le \bar A_2\big(r(x)\big) \le \frac{(m-2)^2}{4} \frac{1}{r(x)^2} \qquad \text{on } \R^m, \\[0.4cm] 
\disp r \left[ \bar A_j(r) -k \frac{(m-2)^2}{4r^2} \right] \in L^1(+\infty),
\end{array}\right.
\end{equation}
for some $k<0$ and $j \in \{1,2\}$, then there exists a positive solution $u$ of $Q_A'(u)=0$ on $\R^m$ such that  $u(x) \ra +\infty$ as $x \ra \infty$  and so \eqref{24} cannot hold. Note that \eqref{221} is weaker than $A_- \in L^1(M)$. However, we stress that it seems hard to find conditions analogous of \eqref{221} and \eqref{223} on more general manifolds and for nonlinear operators.
}
\end{Remark}

We now investigate the existence of uniform upper bounds, i.e. independent of $\Omega$, for the solutions of \eqref{localYamabe0} with boundary data $\varphi=1$. The next lemma, Lemma \ref{lem_linftyestimates_intro} of the Introduction, ensures \emph{global} $L^\infty$-estimates. In view of a subtle asymmetry between bounds from below and above, to reach our goal we had to find a new strategy.

\begin{Lemma}\label{lem_linftyestimates}
Let $M$ be a Riemannian manifold, $f \in C^\infty(M)$, $p \in (1,+\infty)$ and fix $F(t)$ satisfying \eqref{assu_F}. Let $A,B \in L^\infty_\loc(M)$ with $B \ge 0$ a.e. on $M$. Assume that either
\begin{itemize}
\item[$(i)$] $B \equiv 0$ and $Q_A$ is subcritical, or
\item[$(ii)$] $B \not\equiv 0$ and $Q_A$ is non-negative.
\end{itemize}
Suppose that there exist a smooth, relatively compact open set $\Lambda \Subset M$ and a constant $c>0$ such that 
\begin{equation}\label{lassunzion}
A \le cB \qquad \text{a.e. on } M \backslash \Lambda, 
\end{equation}
and fix a smooth, relatively compact open set $\Lambda'$ such that $\Lambda \Subset \Lambda'$, and a constant $\eps>0$.\\ 
Then, there exists a constant $C_\Lambda>0$ depending on $\eps, p,f,F,c, A,B, \Lambda, \Lambda'$ but not on $\Omega$ such that for each smooth, relatively compact open set $\Omega$ with $\Lambda' \Subset \Omega$, the solution $0<z \in C^{1,\mu}(\overline \Omega)$ of
\begin{equation}\label{localYamabe}
\left\{ \begin{array}{ll}
\Delta_{p,f} z + A(x)z^{p-1} - B(x) F(z) = 0 & \quad \text{on } \Omega, \\[0.2cm]
z = \eps & \quad \text{on } \partial \Omega.
\end{array}\right.
\end{equation} 
satisfies
\begin{equation}\label{245}
z \le C_\Lambda \qquad \text{on } \Omega.
\end{equation}

\end{Lemma}
\begin{proof}
Without loss of generality, we can suppose that $c>1$. Using that, by \eqref{assu_F}, $F(t)/t^{p-1} \ra +\infty$ as $t \ra +\infty$, we can fix $\alpha > \eps$ such that
\begin{equation}\label{defin_alpha}
\frac{F(t)}{t^{p-1}} \ge c \qquad \text{for } \, t \ge \alpha.
\end{equation}
Consider the open set 
$$ 
U = \left\{ x \in \Omega \ : \  z(x) > \alpha \right\}.
$$ 
We note that $U \Subset \Omega$ and that, by \eqref{defin_alpha}, $z$ solves
\begin{equation}\label{246}
\left\{
\begin{array}{ll}
\disp \Delta_{p,f} z + (A-cB) z^{p-1} \ge 0 & \quad \text{on } \, U \\[0.2cm]
\disp z= \alpha & \quad \text{on } \, \partial U,
\end{array}\right.
\end{equation}
If $U= \emptyset$, for each $\Omega, z$, then \eqref{245} trivially holds with $C_\Lambda = \alpha$. Therefore, suppose that $U \neq 0$ for some $\Omega, z$. By \eqref{246} and \eqref{lassunzion}, $\Delta_{p,f}z \ge 0$ on $U \backslash \Lambda$. Thus, in the case $U \cap \Lambda = \emptyset$, $\Delta_{p,f} z \ge 0$ on $U$. For $\eps>0$ small, choose a smooth increasing sequence $\{U_\eps\}$ exhausting $U$ as $\eps \ra 0$ and such that $z \le \alpha + \eps$ on $\partial U_\eps$. Applying Proposition \ref{prop_compagen} we infer that $\disp z \le \alpha+\eps$ on $U_\eps$, thus letting $\eps \ra 0$ we get $\disp z \le \alpha$ on $U$, contradicting its very definition. Hence, we conclude that $U \cap \Lambda \neq \emptyset$. Note that
$$
\sup_{U \cap \Lambda} z \ge \alpha =z_{|\partial U},
$$
and since $\Delta_{p,f} z \ge 0$ on $U \backslash \Lambda$, again via Proposition \ref{prop_compagen} we obtain 
$$
\sup_{U \backslash \Lambda}z = \sup_{\partial (U\backslash \Lambda)}z \le \max\left\{ \sup_{\partial U}z, \sup_{\partial \Lambda} z \right\} = \max\left\{ \alpha, \sup_{\partial \Lambda \cap U} z \right\} \le \sup_{U \cap \Lambda}z.
$$
It follows that
\begin{equation}\label{247}
\text{if } \, U = \{ z>\alpha \} \neq \emptyset, \quad \text{then} \qquad \sup_U z \equiv \sup_{\Lambda}z.
\end{equation}
To prove \eqref{245} we proceed by contradiction: if it does not hold, and in view of \eqref{247} there exists a sequence of triples $(\Omega_j, z_j, U_{j})$, where $\Omega_j$ is a smooth, relatively compact open set such that $\Lambda' \Subset \Omega_j$, $z_j$ is a solution of 
\begin{equation}\label{249}
\left\{ \begin{array}{l}
\disp \Delta_{p,f} z_j + \left(A-B\frac{F(z_j)}{z_j^{p-1}}\right)z_j^{p-1} = 0 \qquad \text{on } \Omega_j, 
\\[0.2cm]
z_j = \eps \quad \text{on } \partial \Omega_j, 
\end{array}\right. 
\end{equation}
$U_{j} = \left\{x \in \Omega_j \ : \ z_j(x) > \alpha \right\}$, and 
\begin{equation}\label{250}
\|z_j\| \doteq \|z_j\|_{L^\infty(\Lambda)} \ra +\infty \qquad {\rm as} \  j \ra +\infty.
\end{equation}
Up to taking $j$ large enough, we can suppose that $\|z_j\| > 2\alpha$. Consider the rescaled functions 
\begin{equation}\label{251}
\disp h_j \doteq \frac{z_j} {\|z_j\|} \qquad {\rm on} \ \Omega_j
\end{equation} 
and note that they all solve
\begin{equation}\label{252}
\left\{ \begin{array}{l}
\disp \Delta_{p,f} h_j + Ah_j^{p-1} \ge 0 \qquad \text{on } \Lambda', \\[0.2cm]
\disp\sup_\Lambda h_j = 1, \qquad h_j \le \frac{\alpha}{\|z_j\|}< \frac 12  \qquad \text{on } \Omega_j \backslash U_j.
\end{array}\right. 
\end{equation}
By \eqref{252} and the half-Harnack inequality of Theorem \ref{teo11} $(3a)$ there exists $C_H>0$ such that 
\begin{equation}\label{253}
1=\|h_j\|_{L^\infty(\Lambda)} \le C_H \|h_j\|_{L^p(\Lambda',\di \mu_f)} \qquad \forall \, j. 
\end{equation}
Next we define 
\begin{equation}\label{254}
\eta_j (x) = \left\{ \begin{array}{ll}
\disp \frac{z_j (x) - \alpha}{\|z_j\|} & \quad \text{on } U_j, \\[0.3cm]
0 & \quad \text{on } M \backslash U_j. 
\end{array}\right.
\end{equation}
Then, $0 \le \eta_j \in \lip_c(M)$. Let $V = A-cB$. Thus, $V \le 0$ on $M\backslash \Lambda$ and, by \eqref{246}, $Q_V'(z_j) \le 0$ on $U_j$, therefore 
\begin{equation}\label{bored}
\begin{array}{lcl}
\disp 0 \le Q_V(\eta_j) & = & \disp \frac{1}{p} Q'_V (\eta_j)[\eta_j] \\[0.3cm]
%\disp &=& \disp \frac{1}{p} \int_{A_j} \eta_j %\left\{ -\Delta_{p,f} \left( \frac{z_j-\alpha}{\|%z_j \|} \right)-V  \frac{(z_j-\alpha)^{p-1}}{\|%z_j\|^{p-1}}   \right\} \di \mu_f \\[0.4cm]
\disp &= & \disp \frac{1}{p\|z_j\|^{p-1}} \int_{U_j} \left[ |\nabla z_j|^{p-2} \langle \nabla z_j, \nabla \eta_j \rangle -V  (z_j-\alpha)^{p-1}\eta_j \right]    \di \mu_f \\[0.4cm]
%\disp &= & \disp \frac{1}{p\|z_j\|^{p-1}} \int_{A_j} \frac{z_j-\alpha}{\|z_j\|} \left\{ -\Delta_{p,f}  (z_j-\alpha)-V  (z_j-\alpha)^{p-1}   \right\} \di \mu_f \\[0.4cm]
%\disp &= &\disp \frac{1}{p\|z_j\|^p} \int_{A_j} (z_j-\alpha) \left\{ -\Delta_{p,f} z_j-V  (z_j-\alpha)^{p-1}   \right\} \di \mu_f \\[0.4cm]
& \le & \disp  \frac{1}{p\|z_j\|^p} \int_{U_j} V  \left\{ z_j^{p-1} -(z_j- \alpha)^{p-1}\right\} (z_j-\alpha)\di \mu_f \\[0.4cm]
& \le & \disp \frac{1}{p\|z_j\|^p} \int_{U_j \cap \Lambda} V \left[ z_j^{p-1} -(z_j- \alpha)^{p-1}\right](z_j-\alpha)\di \mu_f \\[0.4cm]
& \le & \disp \frac{1}{p\|z_j\|^p} \int_{U_j \cap \Lambda} V_+(x) \left[ z_j^p -(z_j- \alpha)^p\right]\di \mu_f.
\end{array}
\end{equation}
We observe that on $[\alpha,+\infty)$, the function $$\rho(t) = t^p -(t- \alpha)^p$$ is increasing and $$\rho(t) \sim p\alpha t^{p-1} \qquad {\rm as} \ t \ra +\infty.$$  
We thus infer the existence of constants $c_1, c_2 >0$ just depending on $\alpha$ and such that $$\rho(t) \le c_1t^{p-1}+c_2 \qquad \text{on } \, [\alpha, +\infty).$$ 
Hence
$$
\begin{array}{l}
\disp \frac{1}{p\|z_j\|^p} \int_{U_j \cap \Lambda} V_+(x) \left[ z_j^p -(z_j- \alpha)^p\right]\di \mu_f \le \disp \frac{\|V_+\|_{L^\infty (\Lambda)}}{p\|z_j\|^p} \int_{U_j \cap \Lambda}  \big(c_1 z_j^{p-1} + c_2 \big) \di \mu_f \\[0.5cm]
\disp \le \frac{\| V_+\|_{L^\infty (\Lambda)}}{p \|z_j\|} \vol_f (\Lambda) c_1+ \frac{\| V_+\|_{L^\infty (\Lambda)}}{p \|z_j\|^p} \vol_f (\Lambda) c_2 \ \  \longrightarrow 0 \qquad \text{as } \, j \ra +\infty.
\end{array}
$$
This fact, together with inequality \eqref{bored}, implies that $Q_V(\eta_j) \ra 0$ as $j \ra +\infty$. Now, in both cases $(i)$ and $(ii)$ in the statement of the lemma, $Q_V$ is subcritical on $M$. In fact, if $(i)$ holds, $V\equiv A$ and $Q_A$ is assumed to be subcritical, while under the validity of $(ii)$ the subcriticality property \eqref{09} follows from
$$
Q_V(\varphi) = Q_A(\varphi) + \int_M (cB)|\varphi|^p \di \mu_f \ge \int_M (cB)|\varphi|^p \di \mu_f.
$$
By Theorem \ref{teo_alternative}, $Q_V$ has thus a weighted spectral gap, and in particular $\eta_j \ra 0$ in $L^p_\loc(M)$. Since, by definition \eqref{251},
$$
h_j = \eta_j + \frac{\alpha}{\|z_j\|} \qquad \text{on } U_j,
$$
we deduce that 
$$
\disp\|h_j\|_{L^p(\Lambda' \cap U_j,\di \mu_f)} \ra 0 \qquad \text{as } \, j \ra +\infty
$$
and since $\disp h_j \le \frac{\alpha}{\|z_j\|}$ on $\Lambda' \backslash U_j$ we conclude that 
$$
h_j \ra 0 \quad \text{in } \,  L^p(\Lambda', \di \mu_f) \quad \text{as } \, j \ra +\infty.
$$ 
This contradicts \eqref{253} and proves the claimed \eqref{245}.
\end{proof}

Lemmas \ref{lem_uniformP} and \ref{lem_linftyestimates} enable us to prove

\begin{Proposition}\label{prop_lasolulimitata}
Let $\riem$ be a Riemannian manifold, $f \in C^\infty (M)$, $p \in (1,+\infty)$ and suppose that $Q_0$ is subcritical on $M$. Let $V  \in L^\infty(M)$ have compact support and assume that $Q_V$ is subcritical on $M$.  Then, there exists a positive solution $g \in C^{1,\mu}_\loc(M)$ of 
\begin{equation}\label{233}
-Q'_V (g) = \Delta_{p,f} g + V  |g|^{p-2}g=0 \qquad {\rm on} \ M
\end{equation}
satisfying
\begin{equation}\label{234}
C^{-1} \le g(x) \le C \qquad {\rm on} \ M
\end{equation} 
for some constant $C>1$.
\end{Proposition}
\begin{proof}
Let $\Lambda$ be a smooth, relatively compact open set such that $V\equiv 0$ on $M\backslash \Lambda$, and consider an exhaustion $\{\Omega_j\}$ of $M$ by smooth, relatively compact open sets with $\overline{\Lambda} \subset \Omega_1$. We let $\varphi_j \in C^{1,\mu_j} (\overline{\Omega}_j)$ be the positive solution of
\begin{equation}\label{235}
\left\{
\begin{array}{l}
\disp Q_V' (\varphi_j)= 0 \qquad {\rm on} \ \Omega_j, \\[0.2cm]
\varphi_j=1 \qquad \qquad {\rm on} \ \partial\Omega_j,
\end{array}
\right.
\end{equation}
whose existence is granted by $iv)$ of Proposition \ref{pro115}. Our assumptions enable us to apply Lemma \ref{lem_uniformP} and \ref{lem_linftyestimates} and to conclude the existence of a constant $C>1$ independent of $j$, for which
\begin{equation}\label{236}
C^{-1} \le \varphi_j (x) \le C \qquad {\rm on} \ \Lambda.
\end{equation}
On the other hand, because of \eqref{235} and $V \equiv 0$ on $M \backslash \Lambda$, we have
$$
\left\{
\begin{array}{l}
\disp \Delta_{p,f} \varphi_j=0 \qquad {\rm on} \ \Omega_j \backslash \overline{\Lambda} \\[0.2cm]
\varphi_j =1 \quad {\rm on} \ \partial\Omega_j, \quad  C^{-1} \le \varphi_j \le C \quad {\rm on} \ \partial\Lambda,
\end{array}
\right.
$$
and thus, by the comparison Proposition \ref{prop_compagen},
$$
\disp C^{-1} \le \varphi_j (x) \le C \qquad {\rm on} \ \Omega_j.
$$
Hence, the $\varphi_j$'s are uniformly bounded from above and below. By elliptic estimates $\varphi_j \ra g$ for some $g \in C^{1,\mu}_\loc (M)$ solving $Q'_V (g)=0$ and satisfying \eqref{234}.
\end{proof} 
We are ready to prove Theorem \ref{teouno}. We rewrite its statement for the convenience of the reader.
\begin{Theorem}\label{teo_second}
Let $M^m$ be a Riemannian manifold, $f \in C^\infty(M)$, $p \in (1,+\infty)$, and suppose that $Q_0$ is subcritical on $M$. Let $a,b \in L^\infty_\loc(M)$. Assume that $Q_a$ is subcritical and that 
\begin{itemize}
\item[$i)$] $b_-$ has compact support; 
\vspace{0.1cm}
\item[$ii)$] for some $\theta>0$, $(a - \theta b_+)_- \in L^1(M,\di \mu_f)$;
\vspace{0.1cm}
\item[$iii)$] $a(x) = O\big(b_+(x)\big)$ as $x$ diverges.
\end{itemize}
Then, there exists $\delta>0$ such that, if 
\begin{equation}\label{labbrutta_proof}
b(x) \ge -\delta \qquad \text{on } M, 
\end{equation}
we can find a solution $u \in C^{1,\mu}_\loc(M)$ of
\begin{equation}\label{266}
\left\{\begin{array}{l}
\Delta_{p,f} u + a(x) u^{p-1} -b(x) F(u) = 0 \qquad \text{on } M, \\[0.2cm]
0< u \le \|u\|_{L^\infty(M)} < +\infty \qquad \text{on } M.
\end{array}\right.
\end{equation}
Moreover, if $ii)$ and $iii)$ are replaced by the stronger condition 
\begin{equation}\label{267}
a(x) \asymp b_+(x) \qquad \text{as $x$ diverges,} 
\end{equation}
then $u$ can be constructed with the further property that $\inf_M u >0$.
\end{Theorem}
\begin{proof}
By assumption $iii)$, there exists a relatively compact, open subset $\Lambda$ containing $\supp(b_-)$, and a constant $C \ge \theta$ such that 
\begin{equation}\label{268}
a(x) \le Cb_+(x) \qquad \text{on } M\backslash \Lambda.
\end{equation}
The subcriticality of $Q_a$ on $M$ implies, via Theorem \ref{teo_alternative},  that $Q_a$ has a weighted spectral gap, that is, there exists $W \in C^0(M)$, $W>0$ on $M$ such that
\begin{equation}\label{269}
\disp \int_M W(x) |\varphi|^p \di \mu_f \le Q_a(\varphi) \qquad \forall \, \varphi \in \lip_c(M).
\end{equation}
In particular, if $A(x)= a(x) + W(x)1_\Lambda(x)$ then $Q_A$ is still subcritical on $M$. Fix $\eps>0$ and a relatively compact subset $\Lambda'$ with $\Lambda \Subset \Lambda'$. Next, consider a smooth, relatively compact open set $\Omega$ with $\Lambda' \Subset \Omega$. Applying Lemma \ref{lem-diriproblem} with the choices $A=a$ (respectively, $A= a+ W1_\Lambda$) and $B= b_+$, we produce solutions $\varphi_\infty$, $\varphi_0$ of 
\begin{equation}\label{bellacostruzione}
\left\{ \begin{array}{ll}
\Delta_{p,f} \varphi_\infty + a \varphi_\infty^{p-1} - b_+ F(\varphi_\infty) = 0 & \quad \text{on } \Omega, \\[0.2cm]
\varphi_\infty = \eps & \quad \text{on } \partial \Omega,
\end{array}\right. 
\end{equation}
\begin{equation}\label{bellacostruzione2}
\qquad \left\{ \begin{array}{ll}
\Delta_{p,f} \varphi_0 + (a+ W1_\Lambda) \varphi_0^{p-1} - b_+ F(\varphi_0) = 0 & \quad \text{on } \Omega, \\[0.2cm]
\varphi_0 = \eps & \quad \text{on } \partial \Omega.
\end{array}\right.
\end{equation}
By comparison, $\varphi_\infty \le \varphi_0$. By Lemma \ref{lem_uniformP}, because of assumption $ii)$ we can guarantee the existence of a constant $\bar C_\Lambda(\eps)>0$ (we emphasize its dependence on $\eps$), independent of $\Omega$, such that
\begin{equation}\label{varphiinftydasotto}
\varphi_\infty \ge \bar C_\Lambda(\eps)  \qquad \text{on } \, \Lambda.
\end{equation}
On the other hand, \eqref{268} implies that $a+ W1_\Lambda \le Cb_+$ outside of $\Lambda$. Hence, Lemma \ref{lem_linftyestimates} ensures the existence of $C_\Lambda(\eps)$ independent of $\Omega$ and such that
\begin{equation}\label{unifupbound}
\big(\varphi_\infty \le \big) \, \varphi_0 \le C_\Lambda(\eps) \qquad \text{on } \Omega,
\end{equation}
and we can also consider the sharpest one, that is,
\begin{equation}\label{Cdeltaepsilon}
C_\Lambda(\eps) \doteq \sup_{\footnotesize \begin{array}{c} \Omega : \Omega \text{ open smooth,} \\
\Lambda' \Subset \Omega \Subset M
\end{array}} \|\varphi_{0}\|_{L^\infty(\Omega)},  
\end{equation}
where $\varphi_0$ solves \eqref{bellacostruzione2}. As a consequence of the comparison Proposition \ref{prop_compagen}, $C_\Lambda(\eps)$ is non-increasing as a function of $\eps$.\par 
Define 
\begin{equation}\label{defdelta}
\delta = \big(\min_{\overline\Lambda}W\big)\left[\frac{C_\Lambda(\eps)^{p-1}}{F(C_\Lambda(\eps))}\right], 
\end{equation}
and observe that, by assumption \eqref{labbrutta_proof}, our definition \eqref{defdelta} of $\delta$ and the fact that $F(t)/t^{p-1}$ is increasing,
\begin{equation}\label{lachiave!}
b_- \left[\frac{F(\varphi_0)}{\varphi_0^{p-1}}\right] \le \delta 1_\Lambda \left[\frac{F(C_\Lambda(\eps))}{C_\Lambda(\eps)^{p-1}}\right] = \big(\min_{\overline \Lambda}W\big)1_\Lambda \le W 1_\Lambda,
\end{equation}
on $\Lambda$, and the same relation clearly holds on $\Omega\backslash \Lambda$ since there $b_- \equiv 0$. Inserting into \eqref{bellacostruzione2} we obtain
$$
\Delta_{p,f} \varphi_0 + a\varphi_0^{p-1} - b F(\varphi_0) \le 0 \qquad \text{on } \, \Omega,
$$
that is, $\varphi_0$ is a supersolution for the problem 
\begin{equation}\label{equazvarphi}
\left\{ \begin{array}{l}
\Delta_{p,f} u_\Omega + a u_\Omega^{p-1} -b F(u_\Omega) = 0 \qquad \text{on } \Omega, \\[0.2cm]
\disp u_\Omega = \eps \quad \text{on } \partial \Omega.
\end{array}\right.
\end{equation}
On the other hand, $\varphi_\infty$ is a subsolution for \eqref{equazvarphi}, and the subsolution-supersolution method (see \cite{diaz}, Theorem 4.14 page 272) gives the existence of $u_\Omega$ solving \eqref{equazvarphi}, satisfying
\begin{equation}\label{bounds_uomega}
\varphi_\infty \le u_\Omega \le \varphi_0 \le C_\Lambda(\eps), 
\end{equation}
and, by \eqref{varphiinftydasotto},
\begin{equation}\label{bounds_uomega_sotto}
u_\Omega \ge \bar C_\Lambda(\eps) \qquad \text{on } \, \Lambda.
\end{equation}
Indeed, in our setting we can describe a simple iteration scheme to produce $u_\Omega$: set $V_0 = a+ W 1_\Lambda$. For $n \ge 1$, we inductively define $\varphi_n \in C^{1,\mu}(\overline \Omega)$ as the positive solution of 
\begin{equation}\label{varphin1}\tag{$P_n$}
\left\{ \begin{array}{ll}
\Delta_{p,f} \varphi_n + V_n \varphi_n^{p-1} - b_+ F(\varphi_n) = 0 & \quad \text{on } \Omega, \\[0.2cm]
\varphi_n = \eps & \quad \text{on } \partial \Omega
\end{array}\right.
\end{equation}
where 
\begin{equation}\label{defin_Vn}
V_n \doteq a + b_- \left(\frac{F(\varphi_{n-1})}{\varphi_{n-1}^{p-1}}\right).
\end{equation}
We claim that each $\varphi_n$ exists and that $\{\varphi_n\}$ is a non-increasing sequence bounded below by $\varphi_\infty$. Indeed, by \eqref{lachiave!}, $V_1 \le V_0$ and so $\lambda_{V_1}(\Omega) >0$, which ensures the existence of $\varphi_1$ by Proposition \ref{lem-diriproblem}. Moreover, $\varphi_1$ solves
\begin{equation}
\left\{ \begin{array}{ll}
\Delta_{p,f} \varphi_1 + V_0 \varphi_1^{p-1} - b_+ F(\varphi_1) \ge 0 & \quad \text{on } \Omega; \\[0.2cm]
\varphi_1 = \eps & \quad \text{on } \partial \Omega,
\end{array}\right.
\end{equation}
whence by comparison $\varphi_1 \le \varphi_0$ on $\Omega$. Now, this last inequality (and the monotonicity of $F(t)/t^{p-1}$) gives $V_2 \le V_1$, so that $\lambda_{V_2}(\Omega) >0$ and $(P_2)$ admits a unique positive solution $\varphi_2$. Again from $V_2 \le V_1$, $\varphi_2$ turns out to be a subsolution of $(P_1)$, hence $\varphi_2 \le \varphi_1$ by comparison. Repeating the argument above shows the monotonicity of $\{\varphi_n\}$. The positivity of $\varphi_{n-1}$ ensures that $V_n \ge V_\infty$, so by comparison $\varphi_n \ge \varphi_\infty$ for each $n$. The inequalities $a \le V_n \le V$ and $\varphi_\infty \le \varphi_n \le \varphi_0$ for each $n$ then guarantee, via Theorem \ref{teo11} $(2)$, that there exists $\mu \in (0,1)$ such that the $C^{1,\mu}$-norm of $\varphi_n$ on $\Omega$ is uniformly bounded. Therefore, a subsequence of $\{\varphi_{n_k}\}_k \subset \{\varphi_n\}_n$ converges in the $C^1$-norm, as $k\ra +\infty$, to some non-negative $v_\Omega \in C^1(\overline \Omega)$, and since $\{\varphi_n\}$ is a non-increasing sequence the whole $\{\varphi_n\}$ converges to $v_\Omega$ uniformly. By letting $k \ra +\infty$ in the weak definition of \eqref{varphin1} along the subsequence $\{n_k\}$, we deduce that $v_\Omega$ is a weak solution of \eqref{equazvarphi}.
Now, we choose an exhaustion $\Omega_j$, and let $u_j=u_{\Omega_j}$ be the solution of \eqref{equazvarphi} on $\Omega_j$ constructed above. Note that 
$$
\varphi_{\infty,j} \le u_j \le C_\Lambda(\eps) \ \text{ on } \Omega_j, \qquad u_j \ge \bar C_\Lambda(\eps) \ \ \text{ on } \, \Lambda.
$$
where $\varphi_{\infty,j}$ solves \eqref{bellacostruzione} on $\Omega_j$. Hence, by local elliptic estimates the sequence $\{u_j\}$ subconverges to some global solution $u \ge 0$ of \eqref{266} on the whole $M$, satisfying 
$$
u \le C_\Lambda(\eps) \quad \text{on } \, M, \qquad u \ge \bar C_\Lambda(\eps) \quad \text{on } \, \Lambda.
$$
By Harnack inequality, $u>0$ on $M$, and we have proved the first part of our theorem. The final step is to guarantee that, when $a \asymp b_+$, $\inf_Mu>0$.\par 
This will be accomplished by proving corresponding lower bounds for $\varphi_{\infty,j}$. Up to reducing $\theta$, we can suppose that 
$$
a -\theta b_+ \ge 0
$$ 
outside some compact set. Using the assumption that $F(t)/t^{p-1} \ra 0$ as $t \ra 0$, fix $\alpha \in (0, \eps)$ small enough that 
$$
\frac{F(t)}{t^{p-1}} \le \theta \qquad \text{for } \, t \in [0, \alpha]. 
$$
Inspecting problem \eqref{bellacostruzione} and noting that $\alpha<\eps$, we deduce that on the subset $U_j \doteq \{ \varphi_{\infty, j} < \alpha\} \Subset \Omega_j$ the function $\varphi_{\infty,j}$ solves $\Delta_{p,f} \varphi_{\infty,j} + (a- \theta b_+) \varphi_{\infty,j}^{p-1} \le  0$, hence in particular
\begin{equation}
\left\{ \begin{array}{ll}
\Delta_{p,f} \varphi_{\infty,j} - (a- \theta b_+)_- \varphi_{\infty,j}^{p-1} \le  0 & \quad \text{on } U_j, \\[0.2cm]
\varphi_{\infty,j} = \alpha & \quad \text{on } \partial U_j,
\end{array}\right. 
\end{equation}
In view of the boundary regularity requirements to apply Proposition \ref{prop_compagen}, we fix a smooth open set $S_j$ satisfying 
$$
\big\{ x : \varphi_{\infty,j} \le \alpha/2 \big\} \Subset S_j \Subset U_j,
$$
so that 
\begin{equation}\label{bellacost}
\left\{ \begin{array}{ll}
\Delta_{p,f} \varphi_{\infty,j} - (a- \theta b_+)_- \varphi_{\infty,j}^{p-1} \le  0 & \quad \text{on } S_j, \\[0.2cm]
\varphi_{\infty,j} \ge \frac{\alpha}{2} & \quad \text{on } \partial S_j,
\end{array}\right. 
\end{equation}
%%
%Two cases may occur: either
%\begin{itemize}
%\item[$(1)$] there exists $j_0$ such that $\varphi_{\infty,j} < \alpha/2$ on $\Lambda$ for each $j \ge j_0$, or
%\item[$(2)$] there exists a subsequence, still called $\{\varphi_{\infty,j}\}$, and $x_j \in \Lambda$, such that $\varphi_{\infty,j}(x_j) \ge \alpha/2$ for each $j$.
%\end{itemize}
%We first examine $(A)$. In this case, $\Lambda \Subset S_j$ for each $j \ge j_0$. Setting $V= -(a-\theta b_+)_-$, by the subcriticality of $Q_0$ and since $V \le 0$ we obtain that $Q_V$ is subcritical, and moreover $V_- \in L^1(M, \di \mu_f)$ because of $ii)$. Therefore, we can apply Lemma \ref{lem_uniformP} to deduce the existence of a constant $C>0$ such that $\varphi_{\infty,j} \ge C$ on $\Lambda$, for each $j \ge j_0$. Since $u_j \ge \varphi_{\infty,j}$, passing to the limit we get $u \ge C$ on $\Lambda$. By Harnack inequality applied, as usual, with potential $V = a-bF(u)/u^{p-1}$, we deduce that $u>0$ on $M$, as desired.\par
%%
%%
%On the other hand, if $(B)$ is satisfied, by compactness $x_j \ra x_0 \in \overline{\Lambda}$ up to a subsequence, and, since $u_j \ge \varphi_{\infty,j}$, $\disp u_j (x_j) \in \left[\alpha/2, C_\Lambda(\eps) \right]$. It follows that $u(x_0) \in [1/2, C_\Lambda(\eps)]$. Again by Harnack inequality, $u>0$ on $M$.
%\par
Next, we use the fact that $V \doteq -(a(x)-\theta b_+ (x))_-$ is compactly supported and $Q_V$ is subcritical on $M$ (being $Q_0$ subcritical by assumption, and $V \le 0$): by Proposition \ref{prop_lasolulimitata}, there exists $g \in C^{1,\mu}_\loc(M)$ solution of
\begin{equation}\label{261}
\left\{ \begin{array}{l}
\Delta_{p,f} g - (a-\theta b_+)_- g^{p-1} = 0 \qquad \text{on } M, \\[0.2cm]
0 < \inf_M g \le \|g\|_{L^\infty(M)} < +\infty.
\end{array}\right.
\end{equation}
Rescaling $g$ by multiplying by a positive constant we can assume that
\begin{equation}\label{262}
\disp 0 < \inf_M g \le \sup_M g \le \frac{\alpha}{2}.
\end{equation}
Comparing \eqref{bellacost} and \eqref{261} on $S_j$, by Proposition \ref{prop_compagen} we infer that $\varphi_{\infty,j} \ge g$ on $S_j$, and thus by \eqref{262}
$$
u_j \ge \disp \varphi_{\infty,j} \ge \min\left\{\frac{\alpha}{2},g\right\}=g \qquad \text{on } \, \Omega_j
$$
Passing to the limit, we finally get
$$
\disp u(x) \ge g(x) \quad \text{on } \, M,
$$
and $\inf_M u>0$ follows since $\inf_M g>0$. 
\end{proof}
In the proof of the above result, the parameter $\eps$ plays no role. However, a judicious choice of $\eps$ is crucial in the following proof of Theorem \ref{teouno_bis}.
\begin{Theorem}\label{teouno_bis_2} 
Let $M^m$ be a Riemannian manifold, $f \in C^\infty(M)$ and $p \in (1,+\infty)$. Suppose that $Q_0$ is subcritical on $M$ and let $a \in L^\infty_\loc(M)$ be such that $Q_a$ is subcritical on $M$. Consider $b \in L^\infty_\loc(M)$, and assume
\begin{itemize}
\item[$i)$] $b_-(x)$ has compact support;
\vspace{0.1cm}
\item[$ii')$]  $a(x) \le 0$ outside a compact set;
\vspace{0.1cm}
\item[$iii')$]  $a(x),b(x) 	\in L^1(M,\di \mu_f)$. 
\end{itemize}
Fix a nonlinearity $F(t)$ satisfying \eqref{assu_F}. Then, there exists a sequence $\{u_k\} \subset C^{1,\mu }_\loc(M)$ of distinct weak solutions of
\begin{equation}\label{011primo_2}
\left\{\begin{array}{l}
\disp \Delta_{p,f} u_k + a(x) u_k^{p-1} - b(x) F(u_k) = 0 \qquad \text{ on M}\\[0.2cm]
0 < u_k \le \|u_k\|_{L^\infty(M)} < +\infty,
\end{array}\right.
\end{equation}
such that $\|u_k\|_{L^\infty(M)} \ra 0$ as $k \ra +\infty$. If we replace $ii')$ and $iii')$ by the stronger condition 
$$
iv') \qquad a(x), b(x) \quad \text{have compact support,}
$$
then each $u_k$ also satisfies $\inf_M u_k >0$.
\end{Theorem}

%
%
%\begin{Proposition}\label{prop_quandoale0}
%%
%In the assumptions of the previous , suppose further that
%\begin{equation}\label{avanti!}
%a(x) \le 0 \qquad \text{outside some compact set.} 
%\end{equation}
%Then, the conclusions of Proposition \ref{teo_second} holds without the requirement \eqref{labbrutta_proof}. Moreover, there exists a countable number of positive distinct solutions $\{u_k\}$ of \eqref{266}, satifying 
%$$
%\|u_k\|_{L^\infty(M)} \ra 0 \qquad \text{as } \, k \ra +\infty 
%$$
%and, if $\inf_M u_- >0$, also $\inf_M u_k>0$. 
%\end{Proposition}
%
%
\begin{proof}
We first observe that our set of assumptions on $a$ and $b$ is a special case of the one in Theorem \ref{teo_second}, namely we just include the requirement that $a \le 0$ outside some compact set. Hence, the constructions in the previous theorem hold as well in our setting, and we will refer to the proof of Theorem \ref{teo_second} also for relevant definitions.\par
Up to enlarging $\Lambda$ we can assume that $a \le 0$ on $M\backslash \Lambda$. Let $\Omega$ satisfy $\Lambda'\Subset \Omega$. By Lemma \ref{lem_linftyestimates} applied with $A=a+W1_\Lambda$, $B\equiv 0$ there exists a uniform constant $\hat C$, independent of $\Omega$, such that the solution $\psi$ of
\begin{equation}\label{ilsoprabound}
\left\{ \begin{array}{ll}
\Delta_{p,f} \psi + (a+ W 1_\Lambda) \psi^{p-1} = 0 & \quad \text{on } \Omega, \\[0.2cm]
\psi=1 & \quad \text{on } \partial \Omega
\end{array}\right.
\end{equation}
satisfies $\psi \le \hat C$ on $\Omega$. By comparison and recalling our definition of $\varphi_0$ in \eqref{bellacostruzione2}, we get $\varphi_0 \le \eps\psi$. Therefore, by our definition of $C_\Lambda(\eps)$, 
\begin{equation}\label{clambda_azero}
C_\Lambda(\eps) \le \eps\hat C \ra 0 \qquad \text{as } \eps \ra 0.
\end{equation}
Combining \eqref{clambda_azero}, the monotonicity of $C_\Lambda(\eps)$ and property $F(t)/t^{p-1} \ra 0$ as $t \ra 0$ in \eqref{assu_F}, we deduce that, given any $b$ with $b_-$ compactly supported, there exists $\eps_0$ sufficiently small such that
\begin{equation}\label{relationsepsilon}
b_- \le \big(\min_{\overline\Lambda}W\big)\left[\frac{C_\Lambda(\eps)^{p-1}}{F(C_\Lambda(\eps))}\right] \qquad \text{for each } \eps \le \eps_0.
\end{equation}
Fix such $\eps_0$ and follow the arguments in Theorem \ref{teo_second} with $\eps=\eps_0$. Note that, as in \eqref{lachiave!},
$$
b_-(x) \left( \frac{F(\varphi_0(x))}{\varphi_0(x)^{p-1}}\right) \le b_-(x) \left[ \frac{F(C_\Lambda(\eps))}{C_\Lambda(\eps)^{p-1}}\right] \le 1_\Lambda(x)\big(\min_{\overline\Lambda}W\big)\le W(x) 1_\Lambda(x), 
$$
which is the key step to produce the local solutions of \eqref{equazvarphi}, and now it does not require \eqref{labbrutta_proof}. Proceeding with the construction, we get a solution $u_0$ of 
\begin{equation}\label{266nuovo}
\left\{\begin{array}{l}
\Delta_{p,f} u_0 + a(x) u_0^{p-1} -b(x)F(u_0) = 0 \qquad \text{on } M, \\[0.2cm]
0< u_0  \le C_\Lambda(\eps_0) \qquad \text{on } M,
\end{array}\right.
\end{equation}
and $\inf_M u_0>0$ when $(iv')$ holds. Next, choose $\eps_1<\eps_0$ small enough that 
$$
C_\Lambda(\eps_1) < \min_{\overline\Lambda} u_0. 
$$
Proceeding as above with $\eps_1$ replacing $\eps_0$ we get a solution $u_1$ of
\begin{equation}\label{266nuovo2}
\left\{\begin{array}{l}
\Delta_{p,f} u_1 + a(x) u_1^{p-1} -b(x)F(u_1) = 0 \qquad \text{on } M, \\[0.2cm]
0< u_1  \le C_\Lambda(\eps_1) \qquad \text{on } M,
\end{array}\right.
\end{equation}
and $\inf_M u_1>0$ when $(iv')$ is in force. By our choice of $\eps_1$, $u_1 < u$ on $\Lambda$, thus in particular the two solutions are different. We can now repeat the procedure inductively by choosing, at each step, $\eps_k < \eps_{k-1}$ such that 
$$
C_\Lambda(\eps_k) < \min_{\overline\Lambda} u_{k-1},
$$
obtaining a solution $u_k$ of 
$$
\Delta_{p,f} u_k + a(x) u_k^{p-1} - b(x) F(u_k) = 0 \qquad \text{on } M
$$
satisfying 
\begin{equation}\label{laproprietafinale}
0< u_k \le C_\Lambda(\eps_k) \ \text{ on } M, \quad u_k < u_{k-1} \ \text{ on } \Lambda, \quad \inf_M u_k>0 \ \text{ when } (iv') \text{ holds.}
\end{equation}
By construction, $C_\Lambda(\eps_k) \ra 0$ as $k \ra +\infty$, hence $\{u_k\}$ is the desired sequence.
\end{proof}
\begin{Remark}\label{rem_counterex1}
\emph{The key point that allows, in Theorem \ref{teouno_bis_2}, to get rid of \eqref{labbrutta_proof} is the validity of the asymptotic relation  
\begin{equation}\label{laimpo}
C_\Lambda(\eps) \ra 0^+ \qquad \text{as } \eps \ra 0^+, 
\end{equation}
which is granted via the presence of an uniform $L^\infty$-bound for solutions $\psi$ of \eqref{ilsoprabound}. For general $a,b$, just satisfying  the assumptions of Proposition \ref{teo_second}, \eqref{laimpo} may not hold. As an example, consider the hyperbolic space $\HH^m$ of sectional curvature $-1$. For each $\tau \ge 1$, the radial functions
$$
u_\tau(x) = \left( 2 \cosh^2\left(\frac{r(x)}{2}\right)\right)^{-\frac{m-2}{2}} \beta_\tau \left(\tanh\left(\frac{r(x)}{2}\right)\right),
$$ 
where
$$
\beta_\tau(t) = \frac{(\tau^2-t^2)^{-\frac{m-2}{2}}}{m(m-2)\tau^2}
$$
are all solutions of 
$$
\Delta u + \frac{m(m-2)}{4} u - u^{\frac{m+2}{m-2}}=0 \qquad \text{on } \HH^m.
$$
Moreover, they are decreasing functions of $r(x)$, and the sequence $\{u_\tau\}$ is monotone decreasing. For each $\eps>0$, consider $\Omega = u_1^{-1}\{(\eps, +\infty)\}$. Then, for $\eps$ small enough in such a way that $\Omega \neq \emptyset$, by the definition of $C_\Lambda(\eps)$ in \eqref{Cdeltaepsilon} we deduce
$$
C_\Lambda(\eps) \ge \|u_1\|_{L^\infty(\Omega)} = u_1(0) =  \frac{2^{-\frac{m-2}{2}}}{m(m-2)},
$$
preventing from the validity of \eqref{laimpo}.
}
\end{Remark}

We briefly comment on the sharpness of the subcriticality assumption for $Q_a$ in Theorem \ref{teouno}, and for this reason we now state a result that improves on a theorem in \cite{litamyang}. First of all, we extend definition \eqref{010} to arbitrary subsets $\Lambda \subset M$, that is, we define the fundamental tone $\lambda_V(\Lambda)$ by setting
\begin{equation}\label{295}
\lambda_V(\Lambda)=\sup\lambda_V(\Omega)
\end{equation}
where the supremum is taken over all open subsets $\Omega \subset M$ with smooth boundary such that $\overline \Lambda \subset \Omega$. 
\begin{Proposition}\label{pro296}
Let $\riem$ be a Riemannian manifold $f \in C^\infty (M)$, $p \in (1, \infty)$ and let $a (x) \in L^\infty_\loc (M)$, $b(x) \in C^0 (M)$. Define
$$
\qquad B_0 \doteq \{x \in M \ : \ b(x) \le 0\},
$$  
Let $\Omega$ be an open domain containing $\overline B_0$ and such that there exists a positive, bounded solution $\disp u \in C^0(\overline \Omega) \cap W^{1,p}_\loc (\Omega)$ of
\begin{equation}\label{297}
\Delta_{p,f} u + a(x) u^{p-1}- b(x)F(u) \le 0 \qquad {\rm on} \ \Omega,
\end{equation}
for some nonlinearity $F(t)$ that satisfies \eqref{assu_F}. Then
\begin{equation}\label{298}
\lambda_a(B_0) \ge 0.
\end{equation}
\end{Proposition}
\begin{Remark}\label{rema299} 
\emph{Let us consider the case $p=2$, $f \equiv 0$. The first Dirichlet eigenvalue of the Laplacian on a geodesic ball $B_r$ grows like $r^{-2}$ as $r \ra 0^+$, thus $\lambda_a(B_r) >0$ provided $r$ is sufficiently small and one may think that condition \eqref{298} expresses the fact that $b_-$ is, loosely speaking, small in a spectral sense.
}
\end{Remark}
\begin{proof}
Let $u$ be as above and by contradiction assume that $\lambda_a(B_0) \doteq \lambda<0$. Then, by definition \eqref{295} we can find a sequence $\{U_i\}$ of open sets with smooth boundaries such that 
$$
\overline B_0 \subset U_i \subset \overline U_i \subset \Omega, 
%\quad \overline{U}_{i+1} \subset U_i \ \text{ for %each } i, 
%\qquad \bigcap_{i=1}^{+\infty} U_i = B_0
\qquad \lambda_a(U_i) \ra \lambda \quad \text{as } i \ra +\infty.
$$
Take a nested sequence $\{V_i\}$ of smooth open sets shrinking to $B_0$ such that:
$$
V_i \subset \left\{ x: b(x) < \frac 1 i \right\}, \qquad \overline B_0 \subset V_{i+1} \subset \overline V_{i+1} \subset V_i \subset \Omega, \qquad \bigcap_{i=1}^{+\infty} V_i = \overline B_0.  
$$
Then, up to replacing $U_i$ with $V_i \cap U_i$ and smoothing corners, by the monotonicity of eigenvalues we can further suppose that $\{U_i\}$ satisfies 
\begin{equation}\label{intUi}
\bigcap_{i=1}^{+\infty} U_i = \overline B_0, \qquad b < \frac 1 i \ \text{ on } U_i.
\end{equation}
%Next, for each $k$, let $U_{jk} \doteq \{x \in U_j : \dist (x, \partial U_j) > 1/k\}$. Consider an exhaustion $\{A_k\}$ of $M$, and let $\lambda_{jk} = \lambda_a(U_{j,k} \cap A_k)$. By definition of $\lambda_a(U_j)$, $\lambda_{jk} \ra \lambda_a(U_j)$ as $k \ra +\infty$, and thus for each $j$ there  
%
% 
By the definition of $\lambda_a(U_i)$, there exists a smooth, relatively compact open set $\Omega_i \Subset U_i$ for which $\lambda_i \doteq \lambda_a(\Omega_i) \le \lambda_a(U_i) + 1/i$, so that clearly $\lambda_i \ra \lambda$ as $i$ diverges. Corresponding to $\lambda_i$ there exists a positive eigenfunction $\varphi_i \in C^{1,\mu}(\overline{\Omega}_i)$ satisfying
$$
\left\{
\begin{array}{ll}
\disp Q_a' (\varphi_i)=\lambda_i |\varphi_i|^{p-2} \varphi_i & \quad \text{on } \Omega_i \\[0.2cm]
\disp \varphi_i =0 & \quad \text{on } \partial\Omega_i.
\end{array}
\right.
$$
Setting $\disp h=\log u$ using \eqref{113} and \eqref{297} we see that $h$ solves
$$
\disp\Delta_{p,f} h \le -a(x) + b(x) \frac{F(u)}{u^{p-1}} -(p-1) \left| \nabla h\right|^p.
$$
Integrating on $\Omega_i$ against $\varphi_i^p$ and proceeding as in the proof of $i) \Rightarrow iii)$ in Proposition \ref{pro115} we obtain
$$
\disp\int_{\Omega_i} a(x) \varphi_i^p \di \mu_f -\int_{\Omega_i} b(x) \frac{F(u)}{u^{p-1}} \varphi_i^p \di \mu_f \le \int_{\Omega_i} |\nabla \varphi_i|^p \di \mu_f.
$$
Now, since the $\varphi_i$'s are eigenfunctions, $\disp pQ_a (\varphi)=\lambda_i \| \varphi_i \|^p_{L^p (\Omega_i, \di \mu_f)}$. Therefore, inserting into the above, observing that $b < 1/i$ on $\Omega_i$ and using the monotonicity of $f(t)/t^{p-1}$ we deduce
\begin{equation}\label{intspectralinal}
0 \ge \int_{\Omega_i} \left[ -\lambda_i -b(x) \frac{F(u)}{u^{p-1}}\right] \varphi_i^p \di \mu_f \ge \int_{\Omega_i} \left[ -\lambda_i -\frac{F(\|u\|_{L^\infty (\Omega)})}{i\|u\|_{L^\infty(\Omega)}^{p-1}}\right] \varphi_i^p \di \mu_f.
\end{equation}
Concluding, as $\lambda_i \ra \lambda < 0$, taking $i$ sufficiently large the previous inequality yields the desired contradiction.
\end{proof}
In view of Proposition \ref{pro296}, and since we made no assumptions on the size of the set $\{x : b(x) =0\}$ in Theorem \ref{teouno}, the existence of a bounded solution of \eqref{011} on $M$ when $b$ changes sign requires at least that $Q_a \ge 0$ on the whole $M$. We feel interesting to investigate the validity of Theorem \ref{teouno} when assumption $i)$ is replaced with the requirement that $Q_a$ be non-negative and with a ground state.\par
\begin{Remark} \label{rem_labrutta}
\emph{We conclude this section with a remark on the role of \eqref{labbrutta_general} in Theorem \ref{teouno}. that is, $\|b_-\|_{L^\infty(M)} \le \delta$. Denote with $B_0$ the relatively compact set $\{x: b(x)<0\}$, and for each $\eps \in (0,1)$ choose a smooth relatively compact set $\Omega_\eps \subset \{ x : b_-(x) \ge (1-\eps)\|b_-\|_{L^\infty(M)}\}$. 
%
% 
%\begin{equation} \label{aster} u(x) \ge K \qquad \forall \ x \in B_o = \{x \in M \ : \ b(x) < 0\} \end{equation} where $K$ is a positive constant INDEPENDENT from $b_-$. Note that this is the case if the positive solution $u$ is constrained for instance as in Theorem \ref{teouno} (the sub-solution $u_-$ is independent from $b_-$). We fix $\epsilon \in (0,1)$ and let $\Omega_\epsilon$ the open set given by $$ \Omega_\epsilon = \left\{x \in B_o \: \ b_- (x) > (1-\epsilon) \max_{\overline{B_o}} b_-  \right\} $$. Note that we can choose $\epsilon>0$ so that $$ \partial \Omega_\epsilon = \left\{x \in B_o \: \ b_- (x) = (1-\epsilon) \max_{\overline{B_o}} b_-  \right\}  $$ is sufficiently smooth (say $C^{1,\alpha}$ depending on the regularity of $b_-$ on the open set $B_o$). 
Then we let $\varphi_\eps$ be the positive eigenfunction of  
$$ 
\left\{ \begin{array}{l} Q'_a(\varphi_\eps) = \lambda_a(\Omega_\eps) \varphi_\epsilon^{p-1} \qquad \text{on } \, \Omega_\eps \\[0.2cm]
\disp \varphi_\eps =0 \quad \text{on } \, \partial \Omega_\eps, \qquad \varphi_\eps>0 \quad \text{on } \, \Omega_\eps.
\end{array} \right. 
$$ 
Note that $\lambda_a(\Omega_\eps) >0$ by assumption. Reasoning as in Proposition \ref{pro296} to get \eqref{intspectralinal} we	obtain 
$$
\lambda_a (\Omega_\eps) \|\varphi_\eps \|^p_{L^p (\Omega_\eps)} \ge \int_{\Omega_\eps} b_- \frac{F(u)}{u^{p-1}} \varphi_\eps^p \di \mu_f, 
$$ 
hence using the definition of $\Omega_\eps$
$$
\lambda_a (\Omega_\eps) \|\varphi_\eps\|^p_{L^p (\Omega_\eps)} \ge (1-\eps)\|b_-\|_{L^\infty(M)} \left[\inf_{B_0} \frac{F(u)}{u^{p-1}}\right] \|\varphi_\eps\|^p_{L^p (\Omega_\eps)}
$$ 
in other words
\begin{equation}\label{bellanec}
\|b_-\|_{L^\infty(M)} \left[\inf_{B_0} \frac{F(u)}{u^{p-1}}\right] \le \inf_{\eps \in (0,1)} \frac{\lambda_a(\Omega_\eps)}{1-\eps}.
\end{equation}
The above inequality helps to understand the relationship between the $L^\infty$-norm of the negative part of $b$ and a lower bound for $u$ on $B_0$. In particular, in view of the relation $F(t)/t^{p-1} \ra 0$ as $t \ra 0$ in \eqref{assu_F}, a larger size of $b_-$ forces $u$ to squeeze towards zero on $B_0$. 
}
%This information can be made more precise by observing that, if $u$ comes from our construction in Theorem \ref{teouno}, then $\sup_Mu \le C_\Lambda$, for some constant $C_\Lambda$ depending on $b$ and on the geometry of the operator $Q_V'$. Therefore, writing equation \eqref{ancoraya} in the form $\Delta_{p,f} u + V u^{p-1} =0$ with $V = a-bu^{\sigma-p+1}$ and using Harnack inequality of Theorem \ref{teo11} $3)$, 
%
%
\end{Remark}

\section{Proofs of our geometric corollaries, and concluding comments}\label{sec_applications}
In this section, we prove Theorems \ref{teodue}, \ref{coruno}, \ref{teo_tipohyperb} and their Corollaries  \ref{cor_asflat}  and \ref{corquattro}.
\begin{proof}[Proof of Theorem \ref{teodue}]
Taking into account that the conformal factor $u$ in the deformation 
$$
\widetilde{\metric} = u^{\frac{4}{m-2}} \metric
$$
shall satisfy \eqref{02}, the theorem follows immediately from Theorems \ref{teouno} and \ref{teouno_bis}. When the two-sided bound 
\begin{equation}\label{294}
\disp C^{-1} \metric \le \widetilde{\metric} \le C \metric
\end{equation}
holds, $(M, \widetilde{\metric})$ is complete if and only if $\disp (M,\metric)$ is so. Furthermore, because of \eqref{294} since $\metric$ is non-parabolic the same holds for $\widetilde{\metric}$. Indeed, it is easy to see that \eqref{294} induces a similar two-sided bound between the capacities $\capac$ and $\widetilde{\mathrm{cap}}$ of the Laplace-Beltrami operators of the two metrics (with, say, supersolution $g = 1$), whence the preservation of parabolicity follows from Theorem \ref{teo_alternative}. This concludes the proof.
%\footnote{\tcr{indeed, if we denote with $\widetilde \nabla$, $\di \widetilde x$, $\|\cdot\|$ the gradient, volume form and norm induced by $\widetilde \metric$. Then, for each $\varphi \in \lip_c(M)$ 
%$$
%C^{-\frac{m-2}{2}}\int_M |\nabla \varphi|^2\di x \le  \int_M \|\widetilde \nabla \varphi\|^2\di \widetilde x \le C^{\frac{m-2}{2}}\int_M |\nabla \varphi|^2 \di x,
%$$
%so the capacities of a compact set $K$ with non-empty interior (the ``standard" ones for the Laplacians $-\Delta$ and $-\widetilde \Delta$, and with $g\equiv 1$ as chosen supersolution) are related via 
%$$
%C^{-\frac{m-2}{2}}\mathrm{cap}(K) \le \widetilde{\mathrm{cap}}(K) \le 
%C^{\frac{m-2}{2}}\mathrm{cap}(K).
%$$
%Thus, by Theorem \ref{teo_alternative} the preservation of non-parabolicity is immediate.}}.
\end{proof}
\begin{proof}[Proof of Theorem \ref{coruno}] It follows directly from case $(II)$ of Theorem \ref{teodue}: it is enough to observe that, if $s(x) \ge 0$ on the whole $M$ and $M$ is non-parabolic (i.e., $-\Delta$ is subcritical), then the conformal Laplacian $L_{\metric}$ is subcritical. 
\end{proof}
\begin{proof}[Proof of Corollary \ref{cor_asflat}] 
We begin with performing a ``reduction" argument that goes back to the original works of Schoen-Yau on the positive mass theorem, see \cite{leeparker}, pp. 82-83. Via a cut-off procedure and a careful analysis of Schr\"odinger operators on weighted spaces, they showed that there exists a conformal deformation 
\begin{equation}\label{background}
\metric_1=u_1^{\frac{4}{m-2}}\metric
\end{equation}
of the original  asymptotically flat metric in such a way that $(M,\metric_1)$ is still asymptotically flat and has zero scalar curvature outside a compact set. Moreover, $\metric_1$ is uniformly equivalent to $\metric$ (actually much more is true, but this is enough for our purposes). Next, by the very definition of asymptotic flatness, the metric $\metric_1$ on each end $U_j$ with respect to the compact set $K$ is bi-Lipschitz equivalent to the Euclidean one on $\R^m\backslash B_r(0)$, hence proceeding as in the proof of Theorem \ref{teodue} we deduce that $(M,\metric_1)$ is non-parabolic. Consequently, since $\metric_1$ has non-negative scalar curvature, the conformal Laplacian $L_{\metric_1}$ is subcritical. Applying previous Theorem \ref{coruno} to the background manifold $(M, \metric_1)$, we get the existence of a family of conformal deformations to scalar curvature $\widetilde{s}(x)$ which, after composing with the deformation \eqref{background}, concludes the proof of the corollary.  
\end{proof}
\begin{proof}[Proof of Theorem \ref{teo_tipohyperb}] It follows directly from $(I)$ of Theorem \ref{teodue}. Note that, according to Remark \ref{rem_linearegreen}, $L_{\metric}$ is subcritical if and only if it admits a positive Green kernel. 
\end{proof}
\begin{proof}[Proof of Corollary \ref{corquattro}]
Condition $K \le -\kappa^2$ implies, via Theorem \ref{teo_laprimahardy}, that the Hardy inequality 
$$
\int_{M} (\chi\circ r) \varphi^2 \di x \le \int_M|\nabla \varphi|^2 \di x, 
$$
holds for each $\varphi \in \lip_c(M)$, where 
$$
\chi(r) = \frac{1}{4} \left( g_\kappa (r)^{m-1} \int_r^{+\infty} \frac{\di s}{g_\kappa (s)^{m-1}} \right)^{-2}.
$$
Now, since by \eqref{soprachialfa} the Hardy weight satisfies
$$
\chi(r) \ge \frac{(m-1)^2\kappa^2}{4} \qquad \text{on } \R^+,
$$
using assumption \eqref{iposcalarehyp} we deduce that 
$$
-\frac{s(x)}{c_m} = - \frac{m-2}{4(m-1)}s(x) \le \frac{(m-1)^2\kappa^2}{4} \le \chi\big(r(x)\big),
$$
thus the conformal Laplacian $L_{\metric} = -\Delta + s/c_m$ is subcritical by Proposition \ref{prop_criterion} (clearly, $-s/c_m \not\equiv (\chi\circ r)$ since this latter tends to infinity at $o$). Now, if \eqref{nuovascalarehyp} holds outside of a compact set, assumption $i)$ of Theorem \ref{teodue} is met. Tracing $K \le -\kappa^2$ we deduce 
$$
s(x) \le -m(m-1)\kappa^2,
$$
which coupled with \eqref{iposcalarehyp} and \eqref{nuovascalarehyp} implies $s(x) \asymp \widetilde s(x)$ as $x$ diverges. Applying Theorem \ref{teo_tipohyperb} we eventually have the desired conformal deformation to a uniformly equivalent metric $\widetilde{\metric}$.
\end{proof}
We conclude with a couple of remarks. In the Introduction, the prototype cases of Euclidean and hyperbolic space helped us to have a picture of the variety of phenomena concerning the prescribed curvature problem. We have seen that the uniqueness of the conformal deformation in Theorem \ref{teouno} fails to hold for sign-changing $\widetilde s(x)$, and that fastly decaying solutions coexist with solutions bounded from below and above by positive constants. In particular, assumptions like \eqref{nuovascalarehyp} do not imply a control of the decay of the conformal factor from both sides by two comparable quantities. However, when $\widetilde s <0$ on the whole $M$, something more precise can be said about uniqueness and asymptotic behaviour of solutions $u$ of the Yamabe equation \eqref{02}. As above, consider the prototype case of $\HH^m_\kappa$, and suppose that 
\begin{equation}\label{attheend}
-C_1 \le \widetilde s(x) \le -C_2<0 \qquad \text{on } \, \HH^m_\kappa.
\end{equation}
By Theorem 3.4 in \cite{rigolizamperlin} with the choice $\beta=0$ (or even by Theorem 4 in \cite{avilesmcowen}), assumption \eqref{attheend} guarantees that the conformal deformation given in Aviles-McOwen's Theorem \ref{teo_avilesmcowen} and in Corollary \ref{corquattro} is the \emph{unique} conformal deformation realizing $\widetilde s(x)$ and such that the conformal factor satisfies $\inf_{\HH^m_\kappa} u>0$. Moreover, by \cite{rrv2, prsmemoirs} (a simpler form can also be found in Theorem 2.3 of \cite{rigolizamperlin}), in the same assumptions each solution of \eqref{02} satisfies $\sup_{\HH^m_\kappa}u < +\infty$. On the other hand, since $\widetilde s <0$ on $\HH^m_\kappa$, estimates from below for positive solutions of the Yamabe type equation \eqref{03} have been provided in \cite{rigolizamperlin}. 
%Among the various results therein, only Theorem 2.4 can be applied to the situation just described (the %reader can check that our case of interest is not treated neither in Theorem 2.1, nor in Remark 2.2 nor %in Proposition 2.2 of \cite{rigolizamperlin}). 
Applying Theorem 2.4 of \cite{rigolizamperlin} with the choices
$$
m >4, \ \ \delta = 0, \ \ \beta = -1-\epsilon, \ \ \alpha<0 \ \text{ arbitrary,} \ \  \sigma = \frac{m+2}{m-2}, \ \ \gamma = \frac{m-2}{2}H > H 
$$
with $\epsilon >0$,  we deduce that any solution $u$ of \eqref{02} satisfies 
\begin{equation}\label{boundbellooww}
u(x) \ge C e^{-\frac{m-2}{2}\kappa r (x)} \qquad \text{on } \HH^m_\kappa,
\end{equation}
for some $C>0$. Indeed, the right-hand side in \eqref{boundbellooww} is exactly the asymptotic decay of the solutions that create the conformally deformed metrics in Theorem \ref{cor_tipoiperb}, and in fact it is also the decay of a radial solution of $L_{\metric} u = 0$ on $\HH^m_\kappa$. In summary, when \eqref{attheend} holds, the decay of the solutions $\{u_j\}$ in Theorem \ref{cor_tipoiperb} is the minimal one that a solution of the Yamabe equation with \eqref{attheend} in force can have (if $m >4$), while the function $u$ produced in Corollary \ref{corquattro} (we call it $\hat u$) is the unique solution which is bounded below by a positive constant, and indeed it also has the maximal possible order at infinity, as each solution shall be bounded above by a constant. This intriguing scenario is enriched by the fact that, by Theorem 1.1 in \cite{rrv2}, the Yamabe equation on $\HH^m_\kappa$ also admits a solution $u_c$ giving rise to a complete metric $\widetilde \metric$ whenever 
\begin{equation}\label{siamoallafine}
\widetilde s <0 \ \text{ on } M, \ \text{and } \quad \widetilde s(x) \ge -Cr(x)^2 \quad \text{as } r \ra +\infty, 
\end{equation}
for some $C>0$. It is still not clear whether $u_c$ coincides with $\hat u$ or not, or even if \eqref{siamoallafine} ensures the existence of a whole infinite family of solutions, distinct from $\{u_j\}$ and $\hat{u}$, giving rise to complete metrics.
\begin{Remark}
\emph{Based on the hints in \cite{rrv}, we conjecture that there exists a conformal deformation of the hyperbolic metric that gives rise to a complete metric of scalar curvature $\widetilde s(x)$ whenever $|\widetilde s(x)| \le Cr(x)^2$, where $r(x)$ is the distance from a fixed origin of $\HH^m_\kappa$ and $C>0$. See also \cite{bmr3} for some comments.
}
\end{Remark}

%
%
%
%\section{Appendix}\label{appendix}
%\appendix
\section*{Appendix: the obstacle problem and the pasting lemma}
The aim of this section is to present a proof of the pasting Lemma \ref{lem_pasting}. The argument is divided in three steps. First, observe that our assumptions in Lemma \ref{lem_pasting} imply $\lambda_V(\Omega_2)\ge 0$. Therefore, the obstacle  problem that we shall consider below is solvable on relatively compact open subsets of $\Omega_2$. Secondly, the minimizing properties of its solutions yield a quick proof of the fact that the minimum of two positive supersolutions is still a supersolution. Finally we obtain Lemma \ref{lem_pasting} by refining the argument used in the second step. The idea of the proof is close to that in Section 3 of \cite{marivaltorta}. Hereafter, each $W^{1,p}$-norm is intended to be with respect to the measure $\di \mu_f$

Let $\Omega \Subset M$ be a relatively compact, open subset and $V \in L^\infty(\Omega)$. Given $\psi$ measurable and $\theta \in W^{1,p}(\Omega)$ such that $\psi \le \theta$ a.e. on $\Omega$, we define the non-empty, closed, convex set
\begin{equation}\label{A3}
\Kpt \doteq \Big\{\varphi\in \wup \ \vert \  \ \varphi \ge \psi \ \text{ a.e. and } \ \varphi-\theta\in
\wupz\Big\}.
\end{equation}
We say that $u\in \Kpt$ solves the obstacle problem if
\begin{gather}\label{A4}
Q_V'(u)[\varphi -u] \geq 0 \qquad \text{for each } \varphi \in \Kpt,
\end{gather} 
that is, weakly,
\begin{equation}
\int_\Omega |\nabla u|^{p-2} \langle \nabla u, \nabla (\varphi-u) \rangle \di \mu_f - \int_\Omega V |u|^{p-2} u(\varphi-u) \di \mu_f \ge 0
\end{equation}
Note that, for each non-negative $\widehat{\varphi} \in C^1_c (\Omega)$,  the function $\varphi= u+ \widehat{\varphi} \in \Kpt$, and putting into \eqref{A4}  we get that $u$ solving \eqref{A4} satisfies $Q_V'(u) \ge 0$, that is, $u$ is a supersolution. 
%Note that, if $u \in \Kpt$ is a local minimizer of $Q_V$ on $\Kpt$, then $u$ solves the obstacle problem. Indeed, for %each $\varphi \in \Kpt$, by convexity $u+t(\varphi-u) \in \Kpt$ for each $t \in [0,1]$, and from $Q_V(u+t(\varphi-u)) \ge Q_V(u)$ for $t$ small we get
%$$
%0 \le \left. \frac{\di}{\di t}\right|_0 Q_V \big( u+t(\varphi-u)\big) = Q_V'(u)[\varphi-u].
%$$
%
%
We address the solvability of the obstacle problem in the next
\begin{Theorem}\label{A5}
Let $M$ be a Riemannian manifold, $f \in C^\infty (M)$, $p \in (1,+\infty)$ and $V \in L^\infty_\loc(M)$. Let $\Omega \Subset M$ be a relatively compact open set with Lipschitz boundary for which $\lambda_V(\Omega) > 0$. Suppose that the obstacle $\psi$ satisfies $0 \le \psi\le \theta$ a.e. on $\Omega$, for some $\theta \in W^{1,p} (\Omega)$. Then, there exists a solution $u \in \Kpt$ of \eqref{A4}.
\end{Theorem}
\begin{proof}
We consider the translated set 
$$
\bKpt \doteq \Kpt -\theta = \big\{ \bar{g} \ : \ \bar{g}+\theta \in \Kpt \big\} \subset \wupz.
$$ 
Note that, since $\psi \ge 0$, $\forall \ \bar{g} \in \bKpt$ we have $\bar{g} +\theta \ge 0$. \par 
We now define the functional $\F : \bKpt \to \wupzst$ by setting: $\forall \, \bar{u} \in \bKpt$ and $\varphi \in \wupz$,
\begin{equation}\label{A7}
\F(\bar{u})[\varphi] \doteq Q_V'(\bar u + \theta)[\varphi] = \A(\bar{u})[\varphi]-\B (\bar{u})[\varphi]
\end{equation}
with
$$
\A(\bar{u})[\varphi]=\int_\Omega \left|\nabla (\bar{u}+\theta) \right|^{p-2} \langle \nabla (\bar{u}+\theta),\nabla \varphi \rangle \di \mu_f
$$
and
$$
\B(\bar{u})[\varphi]=\int_\Omega V  \left| \bar{u}+\theta \right|^{p-2} \left( \bar{u}+\theta \right)\varphi \di \mu_f.
$$
Clearly, $\bar u \in \bKpt$ solves the obstacle problem
\begin{equation}\label{A6}
\F(\bar{u})[\bar\varphi -\bar u] \ge 0 \qquad \forall \, \bar{\varphi} \in \bKpt
\end{equation}
if and only if $u = \bar u + \theta \in \Kpt$ is a solution of the obstacle problem \eqref{A4}. \par  

According to Theorem 8.2, p. 247 in \cite{jllions}, to solve the obstacle  problem it is enough to verify that:
\begin{itemize}
\item[1.] $\F$ is pseudo-monotone on $\bKpt$, that is,
\begin{itemize}
\item[i)] $$\F \ : \ \left(\wupz , \|\cdot\|_{\wupz} \right) \ra \left(\wupzst , \|\cdot\|_{\wupzst} \right) $$ is bounded.
\item[ii)] if $u_i, u \in \bKpt$ and $u_i \rightharpoonup u$ in $\left( \wupz , {\rm weak} \right)$ as $i \ra +\infty$  and  
\begin{equation}\label{ipopseudomon}
\limsup_{i \to +\infty} \F(u_i)[u_i-u] \le 0
\end{equation}  
then 
\begin{equation}\label{liminfpseudomonot}
\liminf_{i \to +\infty} \F(u_i) [u_i-\varphi] \ge \F(u) [u-\varphi] \qquad \forall \, \varphi \in \wupz.
\end{equation}
\end{itemize}
\item[2.] $\F$ is coercive on $\bKpt$, that is,
\begin{itemize}
\item[iii)] there exists $\bar{\varphi} \in \bKpt$ for which
$$
\disp \frac{\F(\bar{u}) [\bar{u}-\bar{\varphi}]}{\|\bar{u}\|_{\wupz}}\ra +\infty
$$
if $\|\bar{u}\|_{\wupz} \ra +\infty$ with $\bar{u} \in \bKpt$.
\end{itemize}
\end{itemize}
We first address the pseudo-monotonicity of $\F$.
\begin{itemize}
\item[i)] Boundedness follows since, for each $\varphi \in \wupz$, by H\"older inequality we have
$$
\left| \F(\bar u)[\varphi]\right| \le C\left( 1+ \|\bar u\|_{\wup}^{p-1}\right)\|\varphi\|_{\wup} 
$$
for some constant $C>0$.
\item[ii)] Let $u_i \rightharpoonup u$ weakly in $\wupz$ with $u_i, u \in \bKpt$ and suppose \eqref{ipopseudomon}. Since $u_i \rightharpoonup u$, $\{u_i\} $ is bounded in $\wupz$; since $\partial \Omega$ is Lipschitz, by Rellich-Kondrachov compactness theorem \mbox{ $\|u_i -u\|_{L^p (\Omega)}  \ra 0$} and almost everywhere. Then, it is easy to see that $\BB(u_i) [u_i-u] \ra 0$. Therefore, by \eqref{ipopseudomon} 
\begin{equation}\label{A1primo}
\limsup_i \Aa(u_i) [u_i-u] \le 0.
\end{equation}
Since $\Aa$ is a monotone operator and $u_i \rightharpoonup u$, from \eqref{A1primo} we have
\begin{align*}
\disp 0 &\le \big(\Aa(u_i) - \Aa(u) \big)[u_i-u] = \Aa(u_i)[u_i-u]-  \Aa(u) [u_i-u]\\[0.2cm]
&= \Aa(u_i)[u_i-u] + o(1)
\end{align*}
as $i \ra +\infty$, and we therefore deduce that $\lim_i \big( \Aa(u_i) - \Aa(u) \big) [u_i-u] = 0$. By Browder lemma, see Lemma 3 in \cite{browder}, since $\Aa$ is strictly monotone we obtain that $u_i \ra u$ strongly in $\wupz$. Now we show that $\F$ is sequentially continuous from $(\wupz,\|\cdot\|_{\wupz})$ to $(\wupzst, {\rm weak})$. Let $u_i$ be a sequence in $\wupz$ which converges strongly to $u$. Up a subsequence we have $(u_k (x),\nabla u_k (x)) \ra (u(x),\nabla u(x))$ for a.e. $x \in \Omega$. 
%For any $\varphi \in \wupz$ we consider
%$$
%\disp \left({\cal A}(u_k)-{\cal A}(u) \right)\varphi= \int_\Omega |\nabla (u_k+\theta)|^{p-2}\langle \nabla (u_k+\theta), \nabla \varphi \rangle -|\nabla (u+\theta)|^{p-2}\langle \nabla (u+\theta), \nabla \varphi \rangle  \di \mu_f.
%$$
Set for convenience  
$$
X_k (x)= \nabla \left( u_k (x)+\theta(x) \right), \qquad X(x)=  \nabla \left( u (x) +\theta (x) \right).
$$ 
If $k \to \infty$, $\left(X_k(x), \left| X_k(x) \right|\right) \to \left(X(x), \left| X(x) \right|\right)$ for a.e. $x \in \Omega$. If we set $\nabla \varphi =Y$, we have
$$
\begin{array}{l}
\disp \big({\cal A}(u_k)-{\cal A}(u)\big)[\varphi] =\disp \int_\Omega \langle Y, X_k \left| X_k \right|^{p-2}-X \left|X \right|^{p-2} \rangle \di \mu_f = (I) + (II),
\end{array}
$$
where 
$$
\begin{array}{rcl}
\disp (I) & =& \disp \int_\Omega \frac{\left| X_k \right|^{p-2} \langle Y,X_k \rangle}{1+\left| X_k\right|^{p-1}} \left[ \left|X_k \right|^{p-1} - \left|X \right|^{p-1} \right]\di \mu_f;   \\[0.4cm]
\disp (II) & = & \disp \int_\Omega \langle Y, \frac{\left|X_k \right|^{p-2} X_k}{1+ \left|X_k \right|^{p-1}}-\frac{\left|X \right|^{p-2} X}{1+ \left|X \right|^{p-1}}\rangle \left( 1 + \left|X \right|^{p-1} \right) \di \mu_f.
\end{array}
$$
%
% $$
%\begin{array}{lcl}
%\disp \big({\cal A}(u_k)-{\cal A}(u)\big)[\varphi] & =& \disp \int_\Omega \langle Y, X_k \left| X_k \right|^{p-2}-X \left|X \right|^{p-2} \rangle \di \mu_f \\[0.4cm]
%\disp & =& \disp \int_\Omega \left\{ \frac{\left| X_k \right|^{p-2} \langle Y,X_k \rangle}{1+\left| X_k\right|^{p-1}} \left[ \left( 1+\left|X_k \right|^{p-1} \right) - \left(1+ \left|X \right|^{p-1} \right)\right]\right.   \\[0.4cm]
%\disp & & \disp + \left.  \langle Y, \frac{\left|X_k \right|^{p-2} X_k}{1+ \left|X_k \right|^{p-1}}-\frac{\left|X \right|^{p-2} X}{1+ \left|X \right|^{p-1}}\rangle \left( 1 + \left|X \right|^{p-1} \right) \right\} \di \mu_f.
%\end{array}
%$$
%
Since the integrand in $(II)$ is bounded by $2|Y| \left(1+\left| X \right|^{p-1} \right) \in L^1(\Omega)$, by Lebesgue theorem 
%$$
%\begin{array}{lcl}
%\disp \left| \langle Y, \frac{\left|X_k \right|^{p-2} X_k}{1+ \left|X_k \right|^{p-1}}-\frac{\left|X \right|^{p-2} X}{1+ \left|X \right|^{p-1}}\rangle \left( 1 + \left|X \right|^{p-1} \right) \right| & \le & \disp |Y| \left( 1 + \left|X \right|^{p-1} \right) \\[0.1cm]
%& & \disp \cdot \left| \frac{\left|X_k \right|^{p-2} X_k}{1+ \left|X_k \right|^{p-1}}-\frac{\left|X \right|^{p-2} X}{1+ \left|X \right|^{p-1}} \right| \\[0.4cm] 
%& \le & \disp \disp 2|Y| \left(1+\left| X \right|^{p-1} \right) \in L^1(\Omega),
%\end{array}
%$$
%where the last inequality follows from H\"older inequality, by Lebesgue theorem
\begin{equation}\label{primointe}
(II) \ra 0 \qquad \text{as } k \ra +\infty.
\end{equation}
%$$
%\disp 2 \ge \left| \frac{\left|X_k \right|^{p-2} X_k}{1+ \left|X_k \right|^{p-1}}-\frac{\left|X \right|^{p-2} X}{1+ \left|X \right|^{p-1}} \right| \to 0 \qquad {\rm a.e.} \ \ {\rm on} \ \Omega
%$$
%as $k \to \infty$ for Lebesgue we obtain
%\begin{equation}\label{primointe}
%\disp  \int_\Omega \left| \langle Y, \frac{\left|X_k \right|^{p-2} X_k}{1+ \left|X_k \right|^{p-1}}-\frac{\left|X \right|^{p-2} X}{1+ \left|X \right|^{p-1}}\rangle \left( 1 + \left|X \right|^{p-1} \right) \right| \di \mu_f \to 0 \qquad {\rm as} \ \ k \to \infty.
%\end{equation}
We now consider $(I)$.
%\begin{equation}\label{eqA}
%\disp \int_\Omega \frac{\left| X_k \right|^{p-2} \langle Y,X_k \rangle}{1+\left| X_k\right|^{p-1}} \left(\left|X_k \right|^{p-1}  -  \left|X \right|^{p-1}\right) \di \mu_f.
%\end{equation}
Fix $\epsilon >0$. Then, by Egoroff theorem there exists $\Omega_\epsilon \subset \Omega$ such that $\mu_f (\Omega \backslash \Omega_\epsilon) < \epsilon$ and $|X_k| \ra |X|$ uniformly on $\Omega_\epsilon$ as $k \ra +\infty$. Since 
%We then have
%$$
%\begin{array}{l}
%\disp \int_\Omega \frac{\left| X_k \right|^{p-2} \langle Y,X_k \rangle}{1+\left| X_k\right|^{p-1}} \left(\left|X_k \right|^{p-1}  -  \left|X \right|^{p-1}\right) \di \mu_f = \int_\\ \\
%=\disp \int_{\Omega-\Omega_\epsilon }\frac{\left| X_k \right|^{p-2} \langle Y,X_k \rangle}{1+\left| X_k\right|^{p-1}} \left(\left|X_k \right|^{p-1}  -  \left|X \right|^{p-1}\right) \di \mu_f \\ \\
%\disp +\int_{\Omega_\epsilon} \frac{\left| X_k \right|^{p-2} \langle Y,X_k \rangle}{1+\left| X_k\right|^{p-1}} \left(\left|X_k \right|^{p-1}  -  \left|X \right|^{p-1}\right) \di \mu_f
%\end{array}
%$$
%
$$
\left|\frac{\left| X_k \right|^{p-2} \langle Y,X_k \rangle}{1+\left| X_k\right|^{p-1}} \right| \le |Y| \qquad {\rm on} \ \Omega
$$
we therefore obtain
$$
\int_{\Omega_\epsilon}  \left|\frac{\left| X_k \right|^{p-2} \langle Y,X_k \rangle}{1+\left| X_k\right|^{p-1}} \left(\left|X_k \right|^{p-1}  -  \left|X \right|^{p-1}\right) \right| \di \mu_f \to 0 \qquad \text{as } k\ra +\infty.
$$
On the other hand, using H\"older inequality we deduce
$$
\begin{array}{l}
\disp \left| \int_{\Omega\backslash\Omega_\epsilon }\frac{\left| X_k \right|^{p-2} \langle Y,X_k \rangle}{1+\left| X_k\right|^{p-1}} \left(\left|X_k \right|^{p-1}  -  \left|X \right|^{p-1}\right) \right| \le \|Y \|_{L^p(\Omega)}  \left\|\left| X_k \right|^{p-1}-\left|X \right|^{p-1} \right\|_{L^{\frac{p}{p-1}}(\Omega\backslash \Omega_\epsilon)} \\[0.4cm]
\disp \le C\|Y \|_{L^p(\Omega)} \Big\|\left| X_k \right|^{p}+\left|X \right|^{p} \Big\|^{\frac{p-1}{p}}_{L^{1}(\Omega\backslash \Omega_\epsilon)} \le C\|Y \|_{L^p(\Omega)} \Big\|\big| |X_k - X| + |X| \big|^{p}+\left|X \right|^{p} \Big\|^{\frac{p-1}{p}}_{L^{1}(\Omega\backslash \Omega_\epsilon)} \\[0.4cm]
\le \disp C \|Y \|_{L^p(\Omega)} \left(\disp 2^{p} \Big\|\left|X_k-X \right|^p \Big\|_{L^1(\Omega)} + \left( 2^p+1 \right) \Big\|\left|X \right|^p\Big\|_{L^1(\Omega \backslash \Omega_\eps)} \right)^{\frac{p-1}{p}}.
\end{array}
$$
For some constant $C>0$ only depending on $p$. The first integral converges to $0$ because $\|X_k-X \|_p \to 0$ as $k \to \infty$, while the second is infinitesimal, if $\epsilon \to 0$, by the absolute continuity of the integral. Thus 
$$
0 \le \limsup_{k \ra +\infty} |(I)| \le C(\epsilon),
$$
where $C(\epsilon)\ra 0^+$ as $\epsilon \ra 0$. By the arbitrariness of $\eps$, $(I) \ra 0$ as $k\ra +\infty$ and, combining with \eqref{primointe}, $\left( \Aa(u_k)-\Aa(u)\right)[\varphi] \ra 0$ as $k \ra +\infty$. Next, in an analogous way, it can be shown that 
$$
\big(\BB (u_k)-\BB (u) \big)[\varphi] \ra 0 \qquad \text{as } k \ra +\infty.
$$
%where $w_k(x)=u_k(x)+\theta(x)$ and $w(x)=u(x)+\theta(x)$. Since $w_k(x) \ra w(x)$ for a.e. $x \in \Omega$ as $k \to +\infty$, reasoning as after equation \eqref{eqA} we have
%$$
%\disp  \left<{\cal B}(u_k)-{\cal B}(u),\varphi \right> \longrightarrow 0 \qquad {\rm as} \ \ k \to \infty
%$$
%
We have thus proved that 
$$
{\cal F} (u_k) \rightharpoonup {\cal F} (u) \qquad {\rm in} \ \ W_0^{1,p} (\Omega)^* \quad {\rm as} \ \ k \to \infty,
$$
for some subsequence $\{u_k\}$ of the original $\{u_i\}$. A simple reasoning by contradiction thus shows that the whole $\F(u_i) \rightharpoonup \F(u)$ weakly on $\wupz^*$, proving the sequential continuity of $\F$. 
%
%for 
%To complete the proof we show that 
%$$
%{\cal F} (u_n) \rightharpoonup {\cal F} (u) \qquad {\rm in} \ \ W_0^{1,p} (\Omega)^* \quad {\rm as} \ \ n \to \infty
%$$
%If it does not exist $\epsilon >0$, $\varphi \in W_0^{1,p} (\Omega)$ and a sub-sequence $\left\{u_{n(i)} \right\}_{i \in \N}$ of $\left\{u_{n} \right\}_{n \in \N}$ such that
%\begin{equation}\label{contro}
%\left|\langle {\cal F}(u_{n(i)})-{\cal F}(u),\varphi \rangle \right| >\epsilon \qquad \forall \ i \in \N
%\end{equation}
%But the reasoning of before, from $\left\{u_{n(i)} \right\}_{i \in \N}$ it can extract a subsequence $\left\{u_{m(i)} \right\}_{i \in \N}$ such that 
%$$
%{\cal F} (u_{m(i)}
%) \rightharpoonup {\cal F} (u) \qquad {\rm in} \ \ W_0^{1,p} (\Omega)^* \quad {\rm as} \ \ i \to \infty
%$$
%This is in contradiction with \eqref{contro}. \par \vskip 0.2 truecm
Therefore, since $\|u_i-u \|_{\wupz} \ra 0$ as $i \to +\infty$, for each $\varphi \in \wupz$, we have:
$$
\F(u_i)[ u_i -\varphi] = \F(u_i)[u_i - u] + \F(u_i)[u-\varphi] \ra 0 + \F(u)[u -\varphi]
$$
as $i \ra +\infty$, which because of \eqref{ipopseudomon} proves the pseudo-monotonicity of $\F$.\\
\item[iii)] We are left to prove the coercivity of $\F$. This is a more or less predictable consequence of the assumption $\lambda_V(\Omega) >0$, but some technical details suggest to provide a full proof. We shall prove the validity of $iii)$ with the choice $\bar{\varphi} \equiv 0 \in \bKpt$. \par 
First we observe that, by Cauchy-Schwarz inequality,
\begin{equation}\label{eqcoerc}
\disp  \Aa(\bar{u}) [\bar{u}] \ge \disp \int_\Omega |\nabla(\bar{u}+\theta)|^{p}\di \mu_f - \int_\Omega |\nabla(\bar{u}+\theta)|^{p-1} |\nabla \theta|  \di \mu_f. 
%& \ge & \|\nabla(u+\theta)\|_{L^p(\Omega)}^{p} - C\|\nabla(u+\theta)|_{L^p(\Omega)}^{p-1}. \\[0.4cm]
\end{equation}
Next, using
$$
|X+Y|^p \ge |X|^p -p |X|^{p-1}|Y|, \qquad |X+Y|^{p-1} \le 2^{p-1}\big(|X|^{p-1}+|Y|^{p-1}\big)
$$ 
and H\"older inequality into \eqref{eqcoerc} we obtain
%\begin{equation}\label{A9}
%\begin{array}{l}
%\disp \Aa(\bar{u})[\bar{u}]  \ge  \disp \int_\Omega |\nabla \bar{u}|^p \di \mu_f - p \int_\Omega |\nabla \bar{u}|^{p-1}|\nabla \theta| \di \mu_f - 2^{p-1}\int_\Omega |\nabla \bar{u}|^{p-1} |\nabla \theta| \di \mu_f \\[0.4cm]
%\disp -2^{p-1} \int_\Omega |\nabla \theta|^p \di \mu_f \ge \|\nabla \bar{u} \|_{L^p(\Omega)}^p -C_3 \|\nabla \bar{u} \|_{L^p(\Omega)}^{p-1}-C_2 
%%\int_\Omega |\nabla \bar{u}|^p \di \mu_f - \left( p+2^{p-1} \right) \left\{ \int_\Omega |\nabla \bar{u}| \di \mu_f \right\}^\frac{p-1}{p} \left\{ \int_\Omega |\nabla \theta|^p \di \mu_f \right\}^\frac{1}{p} \\[0.4cm]
%%\disp -2^{p-1} \int_\Omega |\nabla \theta|^p \di \mu_f   
%\end{array}
%\end{equation}
\begin{equation}\label{A9}
\disp \Aa(\bar{u})[\bar{u}]  \ge \|\nabla \bar{u} \|_{L^p(\Omega)}^p -C_3 \|\nabla \bar{u} \|_{L^p(\Omega)}^{p-1}-C_2 
\end{equation}

%Hence,
%\begin{equation}\label{A9}
%\disp\Aa (\bar{u}) (\bar{u}) \ge \|\nabla \bar{u} \|_{L^p(\Omega)}^p -C_3 \|\nabla \bar{u} \|_{L^p(\Omega)}^{p-1}-C_2
%\end{equation}
for some constants $C_2,C_3>0$ independent of $\bar{u}$. \par 
To deal with   $\BB(\bar{u})[\bar{u}]$ we first observe that, since $\theta \ge \psi \ge 0$ and $\bar{u}+\theta \ge 0$, we have
\begin{equation}\label{trouble}
\begin{array}{lcl}
\disp -\BB(\bar{u})[\bar{u}] & = &  \disp - \int_\Omega V  |\bar{u}|^p \di \mu_f - \int_\Omega V \left[|\bar{u}+\theta|^{p-2}(\bar{u}+\theta) - |\bar{u}|^{p-2} \bar{u}\right] \bar{u} \di \mu_f \\[0.4cm]
\disp & \ge & \disp -\int_\Omega V  |\bar{u}|^p \di \mu_f - \disp \|V\|_{L^\infty (\Omega)}\int_{\{\bar{u} \ge \theta\}} \left[|\bar{u}+\theta|^{p-1} - |\bar{u}|^{p-1}\right] \bar{u} \di \mu_f \\[0.4cm]
& & - \disp \|V\|_{L^\infty (\Omega)}\int_{\{\bar{u} < \theta\}} \left[|\bar{u}+\theta|^{p-1} - |\bar{u}|^{p-1}\right] \bar{u} \di \mu_f.
\end{array}
\end{equation}
Since $\bar{u}\ge -\theta $, on the set $\{ x \in \Omega \ : \ \bar{u}(x) \le \theta(x)\}$ we have $\bar{u} (x) \in [-\theta,\theta]$; hence the third integral in the right-hand side of the above is bounded above by $\disp C_5=3^{p-1}2\|V\|_{L^\infty(\Omega)}\|\theta\|_{L^p(\Omega)}^p$. As for the second integral, on $\{\bar{u} \ge \theta\}$, the following elementary inequality holds:
$$
(\bar{u}+\theta)^{p-1} - \bar{u}^{p-1} \le \left\{ \begin{array}{ll}
\disp \bar{u}^{p-1} + \theta^{p-1} -\bar{u}^{p-1} = \theta^{p-1} & \quad \text{if } p-1 \le 1, \\[0.3cm]
\disp \bar{u}^{p-1}\left[\left( 1+ \frac{\theta}{\bar{u}}\right)^{p-1} -1\right] \le\bar{u}^{p-1} \left[p 2^{p-1}\frac{\theta}{\bar{u}}\right]   & \quad \text{if } p-1 > 1.
\end{array}\right.
$$
Therefore, 
$$
\begin{array}{l}
\disp \int_{\{\bar{u} \ge \theta\}} \big[(\bar{u}+\theta)^{p-1} - \bar{u}^{p-1}\big]\bar{u} \di \mu_f \\[0.4cm] \le \left\{ \begin{array}{ll}
\disp \int_{\{\bar{u} \ge \theta\}} \theta^{p-1}\bar{u} \di \mu_f \le \|\theta \|_{L^p(\Omega)}^{p-1}\|\bar{u}\|_{L^p(\Omega)}  & \quad \text{if } p-1 \le 1, \\[0.3cm]
\disp p 2^{p-1} \int_{\{\bar{u} \ge \theta\}} \theta\bar{u}^{p-1} \di \mu_f \le p2^{p-1}\|\theta \|_{L^p(\Omega)}\|\bar{u}\|_{L^p(\Omega)}^{p-1}  & \quad \text{if } p-1 > 1.
\end{array}
\right.
\end{array}
$$
Inserting the obtained estimates into \eqref{trouble} we finally get
\begin{equation}\label{A11}
\disp -\BB(\bar{u}) [\bar{u}] \ge - \int_\Omega V |\bar{u}|^p \di \mu_f - C_4\|\bar{u}\|^{\max\{p-1,1\}}_{L^p(\Omega)} - C_5
\end{equation}
for some constants $C_4,C_5>0$ independent of $\bar{u}$. \par 
Combining  \eqref{A9} and \eqref{A11} we obtain
\begin{equation}\label{A12}
\F(\bar{u})[\bar{u}] \ge \| \nabla \bar{u}\|^p_{L^p (\Omega)}-C_3 \| \nabla \bar{u}\|^{p-1}_{L^p(\Omega)}-C_6-C_4 \| \bar{u} \|_{L^p(\Omega)}^{\max \left\{p-1,1\right\}}-\int_\Omega V  |\bar{u}|^p \di \mu_f.
\end{equation}
On the other hand, using Rayleigh characterization of $\lambda_V(\Omega)$, since $\bar u \in \wupz$ we get
$$
\disp \|\nabla \bar{u}\|^p_{L^p(\Omega)} -\int_\Omega V  |\bar{u}|^p \di \mu_f \ge \lambda_V(\Omega) \| \bar{u}\|^p_{L^p(\Omega)},
$$
thus,
\begin{equation}\label{cisiamo!}
\disp  \F(\bar{u}) [\bar{u}] \ge \lambda_V(\Omega) \|\bar{u}\|_{L^p(\Omega)}^{p} - C_4\|\bar{u}\|^{\max\{p-1,1\}}_{L^p(\Omega)} - C_3\|\nabla \bar{u}\|_{L^p(\Omega)}^{p-1} - C_6 
\end{equation}
for some constants $C_4, C_3, C_6>0$ and independent of $\bar{u}$. \par 
Since $\bar{u} \in \wupz$, by Poincar\'e inequality on $\Omega$, there exists a constant $C_P>0$ independent of $\bar{u}$ such that 
$$
\|\bar{u}\|_{L^p(\Omega)} \leq C_P \|\nabla \bar{u}\|_{L^p(\Omega)}.
$$ 
Now, let $\bar{u}_k \in \bKpt$ be any sequence such that $\|\bar{u}_k\|_{\wup} \ra +\infty$ as $k \ra +\infty$. Again by Poincar\'e inequality, $\|\nabla \bar{u}_k\|_{L^p(\Omega)} \ra +\infty$ as $k \ra +\infty$ , and two cases may occur: either
$$
(a) \quad \frac{\|\bar{u}_k\|_{L^p(\Omega)}}{\|\nabla \bar{u}_k\|_{L^p(\Omega)}} \ra 0, \qquad \text{or} \quad (b) \quad \limsup_{k \ra +\infty} \frac{\|\bar{u}_k\|_{L^p(\Omega)}}{\|\nabla \bar{u}_k\|_{L^p(\Omega)}} = c >0.
$$
In the case $(a)$, using \eqref{A12}, we deduce
$$
\begin{array}{l}
\disp \frac{ \F(\bar{u}_k)[\bar{u}_k]}{\|\bar{u}_k\|_{\wupz}} \ge \\[0.4cm]
 \disp \ge \frac{\|\nabla \bar{u}_k \|_{L^p(\Omega)}^{p} - \|V\|_{L^\infty (\Omega)} \|\bar{u}_k\|_{L^p(\Omega)}^p - C_4\|\bar{u}_k\|^{\max\{p-1,1\}}_{L^p(\Omega)} - C_3\|\nabla \bar{u}_k\|_{L^p(\Omega)}^{p-1} - C_6}{(1+C_P)\|\nabla \bar{u}_k\|_{L^p(\Omega)}} \ra +\infty
\end{array}
$$
as $\|\nabla \bar u_k\|_{L^p(\Omega)} \ra +\infty$. \par

In case $(b)$, for each subsequence (still denoted with $\{\bar{u}_k\}$) satisfying 
$$
\frac{\|\bar{u}_k\|_{L^p(\Omega)}}{\|\nabla\bar{u}_k\|_{L^p(\Omega)}} \ra \bar c \in (0,c] \qquad \text{for} \ k \to +\infty
$$
using \eqref{cisiamo!} and $\lambda_V(\Omega) >0$ we get
$$
\disp \frac{ \F(\bar{u}_k)[\bar{u}_k]}{\|\bar{u}_k\|_{\wupz}}  \ge  \disp \frac{\lambda_V(\Omega) \|\bar{u}_k\|_{L^p(\Omega)}^{p} - C_4\|\bar{u}_k\|^{\max\{p-1,1\}}_{L^p(\Omega)} - C_3\|\nabla \bar{u}_k\|_{L^p(\Omega)}^{p-1} - C_6}{(1+C_P)\|\nabla\bar{u}_k\|_{L^p(\Omega)}} \ra +\infty
$$
as $\|\nabla \bar u_k\|_{L^p(\Omega)} \ra +\infty$. This concludes the proof of the coercivity of $\F$.
\end{itemize}
 
\end{proof}
For $\theta=\psi$ we set ${\cal K}_\psi$ for ${\cal K}_{\psi,\psi}$. Next, we prove a minimizing properties for solutions of the obstacle problem.
%
%
%The next Proposition shows that the converse is also true, namely that a solution of the obstacle problem is indeed a global minimizer for $Q_V$ on $\Kpt$.
%
%\begin{Proposition}\label{prop_minimizing}
%%
%Let $(M,\metric)$ be a Riemannian manifold, $f \in C^\infty (M)$, $p \in (1,+\infty)$ and $V \in L^\infty_\loc (M)$. Let $\Omega \Subset M$ be an open set with smooth boundary, such that $\lambda_V(\Omega) > 0$. Suppose that the obstacle $\psi$ satisfies $0 \le \psi\le \theta$ a.e. on $\Omega$, and let $u \in \Kpt$ be a solution of the obstacle problem \eqref{A4} on $\Kpt$. Then, $Q_V(u) \le Q_V(\varphi)$ for each $\varphi \in \Kpt$.
%\end{Proposition}
%%
%%
%\begin{proof}
%%
%Let $\K_\eps = \Kpt + \eps$ and consider the obstacle problem on $\K_\eps$
%\end{proof}
%

%
\begin{Proposition}\label{pro_compaosta}
Let $u \in \Kpt$, $u \not \equiv 0$ be a solution of the obstacle problem with $\psi \ge 0$. Suppose that $w \in \wup$, $w \not\equiv 0$ solves $Q_V'(w) \ge 0$ on $\Omega$, such that $\min\{u,w\}\in \Kpt$. If 
\begin{equation}\label{technical}
\frac u w, \, \, \frac w u \in L^\infty(\Omega), 
\end{equation}
then $u\leq w$ on $\Omega$.
\end{Proposition}
\begin{proof}
Since $w,u$ are non-negative, nonzero supersolutions ($u$ being a solution of the obstacle problem), by the half-Harnack inequality we have $u>0$ and $w>0$ on $\Omega$.
%
%Theorem 5.4 page 235 in \cite{ZM}, $u \in C^0(\overline \Omega)$. 
%Furthermore, since $\min\{u,w\} \in {\cal K}_\psi$,  $u(x),w(x) \ge \psi(x)>0$ on $\overline \Omega$. \par 
Combining with \eqref{technical} and since $\Omega$ is relatively compact, we deduce that the following set is non-empty.  
$$
A \doteq \{c>1 : cw(x) \ge u(x) \quad \text{for a.e. } x \in \Omega\}.
$$ 
Let $c_0 = \inf A$, so that $c_0w \ge u$ a.e. on $\Omega$. We shall show that $c_0=1$. By contradiction suppose $c_0>1$, and choose $1<c<c_0$ close enough to $c_0$ to have
\begin{equation}\label{occhio}
\frac{(cw)^p}{u^{p-1}} \ge \psi \quad \text{on } \Omega. 
\end{equation}
This is possible since, for $c \ge c_0^{\frac{p-1}{p}}$ we have
$$
\frac{(cw)^p}{u^{p-1}} \ge \frac{(cw)^p}{(c_0w)^{p-1}} = \frac{c^p}{c_0^{p-1}} w \ge w \ge \psi, \qquad \text{on} \qquad \overline \Omega.
$$
Consider the non-empty set $U = \{ x \in \Omega \ : \ u(x)>cw(x)\}$, and define
$$
\varphi= u + \min\left\{\frac{(cw)^p-u^p}{u^{p-1}}, 0\right\} = \min \left\{ u, \frac{(cw)^p}{u^{p-1}} \right\}.
$$
Because of \eqref{occhio}, $\varphi \ge \psi$ and $\varphi-\theta =u-\theta=0$ on $\partial\Omega$; in other words, $\varphi \in \Kpt$. \par  Since $u$ solves the obstacle problem \eqref{A4}, using the above $\varphi$ we deduce
\begin{equation}\label{primaosta}
\begin{array}{lcl}
\disp 0 \le Q_V'(u)[\varphi-u] & = & \disp\int_U |\nabla u|^{p-2}\langle \nabla u, \nabla \left(\frac{(cw)^p-u^p}{u^{p-1}}\right)\rangle \di \mu_f \\[0.4cm]
& & - \disp \int_U V \big((cw)^p-u^p\big)\di \mu_f.
\end{array}
\end{equation}
On the other hand, applying the definition of supersolution to $cw$
with the non-negative test function 
$$
\disp\widetilde{\varphi}= \frac{\big(u^p-(cw)^{p}\big)_+}{(cw)^{p-1}}\in W^{1,p}_0(\Omega),
$$ 
we get
\begin{equation}\label{secondasuper}
\begin{array}{lcl}
\disp 0 \le Q_V'(cw)[\widetilde \varphi] & = & \disp\int_U |\nabla (cw)|^{p-2}\langle \nabla (cw), \nabla \left(\frac{u^p-(cw)^p}{(cw)^{p-1}}\right)\rangle \di \mu_f \\[0.4cm]
&& - \disp \int_U V \left(u^p-(cw)^p\right)\di \mu_f.
\end{array}
\end{equation}
Summing up \eqref{primaosta} and \eqref{secondasuper} we get
\begin{equation}\label{intemag000}
\int_U |\nabla (cw)|^{p-2}\langle \nabla (cw), \nabla \left(\frac{u^p-(cw)^p}{(cw)^{p-1}}\right)\rangle \di \mu_f - \int_U |\nabla u|^{p-2}\langle \nabla u, \nabla \left(\frac{u^p-(cw)^p}{u^{p-1}}\right)\rangle \di \mu_f \ge 0.
\end{equation}
Set $z = \max\{u,cw\}$, and note that \eqref{intemag000} is equivalent to $I(cw,z) \le 0$. Indeed, in the definition \eqref{111'} of $I$, the part of the integral outside $U$ is zero since $z \equiv cw$. To conclude, applying Proposition \ref{anane} we deduce $I(cw,z)=0$, so that $cw$ and $z$ (hence $u$) are proportional. Since $cw=u$ on $\partial U$ we conclude that $u \equiv cw$ on $U$, contradicting the definition of this latter.
\end{proof}
\begin{Remark}
\emph{A typical case when \eqref{technical} is automatically met is when the data $\psi,\theta$ satisfy $\psi,\theta \in \wup \cap C^0(\overline \Omega)$ and $\theta>0$ on $\partial \Omega$, which we will frequently use. In fact, when $\psi, \theta \in C^0(\overline\Omega)$, the solution $u \in \Kpt$ of the obstacle problem is continuous on $\overline \Omega$. In this respect, see Theorem 5.4, page 235 of \cite{ZM}. 
}
\end{Remark}

\begin{Remark}
\emph{Although it is not explicitly stated, condition $\lambda_V(\Omega) \ge 0$ is automatic by assuming the existence of a $u$ solving the obstacle problem and $\psi \ge 0$, $\psi \not \equiv 0$. Indeed, $u$ is a positive solution of $Q_V'(u) \ge 0$ on $M$, and $\lambda_V(\Omega) \ge 0$ follows from Proposition \ref{pro115}. 
}
\end{Remark}
In particular from the above proposition we deduce the following characterization.
\begin{Corollary}\label{corsuper}
Let $\Omega \Subset M$ be a relatively compact open set, and let $0 \le \theta \in \wup \cap C^0(\overline \Omega)$ be such that $\theta>0$ on $\partial \Omega$. Then, a solution $u$ of the obstacle problem on $\Kpt$ with obstacle $0 \le \psi \le \theta$, $\psi \not \equiv 0$ is the minimal $w \in \Kpt$ satisfying $Q_V' (w) \ge 0$ on $\Omega$. Consequently, such a solution is unique.
\end{Corollary}
A second important consequence of Proposition \ref{pro_compaosta} is
\begin{Proposition}\label{minsupmaxsub}
Let $w_1,w_2\in W^{1,p}_\loc(M) \cap C^0 (M)$ be positive solutions of $Q_V'(w_j) \ge 0$ on $M$, $j=1,2$. Then, 
$$
w \doteq \min\{w_1,w_2\}
$$ 
solves 
\begin{equation}\label{A19}
Q_V'(w) \ge 0 \qquad {\rm on} \ \ M.
\end{equation}
\end{Proposition}
\begin{proof}
Fix any relatively compact open set $\Omega$ with smooth boundary and consider a solution $s$ of the obstacle problem ${\cal K}_w$ on $\Omega$. This is possible since $w \in \wup$. From $s \in {\cal K}_w$ we have $s \ge w$, and, being a solution of the obstacle problem, $Q_V' (s) \ge 0$. 
%By Theorem 5.4 page 235 in \cite{ZM}, $s \in C^0(\overline \Omega)$ and in particular $s=w$ on $\partial\Omega$. 
Next, since $w_j >0$ on $\overline \Omega$ and $Q'_V(w_j)\ge 0$, we can apply Proposition \ref{pro_compaosta} to obtain $s \le w_j$ for each $j$, whence $s \le w$. This shows that $s\equiv w$, so that $w$ solves \eqref{A19} as claimed.
\end{proof}
We are now ready for the 
\begin{proof}[Proof of Lemma \ref{lem_pasting}] First we show that $u \in W^{1,p}_\loc(\overline \Omega_2)$. Let $U\subset \overline\Omega_2$ be a relatively compact open set. Without loss of generality we can assume that $\Omega_1 \cap U \neq \emptyset$, for otherwise $u \equiv u_2$ on $U$ and the sought is immediate. Since $z=\min \{u_1-u_2,0\} \in W^{1,p}(\Omega_1)$ is zero on $\partial \Omega_1 \cap \Omega_2$, there exists a sequence $\{\varphi_j\} \subset C^0(\overline{U \cap \Omega}_1) \cap W^{1,p}(U\cap \Omega_1)$,  $\varphi_j \equiv 0$ on some neighbourhood of $\partial \Omega_1 \cap \Omega_2$, converging in the $W^{1,p}(U \cap \Omega_1)$ norm to $z$. Using that $z\le 0$ on $\Omega_1$, one can take, for instance, 
\begin{equation}\label{sequence}
\varphi_j = \left( z + \frac{1}{j} \right)_-.
\end{equation}
We extend each $\varphi_j$ to a continuous function on $U\backslash \Omega_1$ by setting $\varphi_j \equiv 0$ on $U\backslash \Omega_1$, so that $\varphi_j \in W^{1,p}(U)$ and clearly $\varphi_j \rightarrow z 1_{\Omega_1}$ in $W^{1,p}(U)$, which shows that $z1_{\Omega_1}\in W^{1,p}(U)$. It follows that $u = u_2 + \varphi_j \in W^{1,p}(U)$ converges to $u = u_2 + z1_{\Omega_1}$, whence $u\in W^{1,p}(U)$. \par
To prove that $Q_V' (u) \ge 0$ on $\Omega_2$, we shall reduce ourselves to Proposition \ref{pro_compaosta}. We take a smooth open set $U \Subset \Omega_2$ which, without loss of generality, intersects $\Omega_1$. Since $\lambda_V(U) >0$ by monotonicity of eigenvalues, Theorem \ref{A5} guarantees the existence of a solution $s$ of the obstacle problem in ${\cal K}_u$ on $U$ which, applying Theorem 5.4, page 235 in \cite{ZM} is continuous on $\overline U$. We want to show that $s \equiv u$ on $U$. First, since $u>0$ on $\overline U$, using Proposition \ref{minsupmaxsub} we get
\begin{equation}\label{quasiquasi..}
s \le u_2 \qquad \text{on } U. 
\end{equation}
Hence, $s=u_2=u$ on $U \backslash \Omega_1$. Consequently, since also $s \in \cal K_u$, $s-u=0$ on $\partial U \cup (\partial \Omega_1 \cap U) = \partial (\Omega_1 \cap U)$, and so by standard theory 
\begin{equation}\label{A24}
s-u \in W_0^{1,p} (\Omega_1 \cap U).
\end{equation}
(one can construct an approximating sequence as in \eqref{sequence} above). Therefore, $s$ is also a solution of the obstacle problem on the closed convex set
$$
\hat{\cal K}_u = \Big\{\varphi\in W^{1,p}(\Omega_1 \cap U) \ \vert \  \ \varphi \ge u \ \text{ a.e. and } \ \varphi-u\in
W^{1,p}_0(\Omega_1 \cap U)\Big\}.
$$
As $u_1$ is a positive supersolution on $\Omega_1 \cap U$ and $\min\{s,u_1\} \in \hat{\cal K}_u$, by Proposition \ref{pro_compaosta}   
\begin{equation}\label{ancoraquasi..}
s \le u_1 \qquad \text{on } \Omega_1 \cap U.
\end{equation}
Coupling with \eqref{quasiquasi..} we get $s \le u$ on $\Omega_1 \cap U$, and combining with $s \ge u$ on $U$ ($s \in \cal K_u$), $s=u = u_1$ on $U \backslash \Omega_1$, we conclude $s \equiv u$ on $U$. This shows that $u$ solves $Q_V'(u) \ge 0$ on $U$. As $U \Subset \Omega_2$ is arbitrary, $Q_V'(u) \ge 0$ on $\Omega_2$, concluding the proof. 
\end{proof}

\vspace{0.7cm}

\noindent \textbf{Acknowledgements}: The second author is supported by the grant PRONEX - N\'ucleo de An\'alise Geom\'etrica e Aplicac\~oes 
Processo nº PR2-0054-00009.01.00/11.\\
It is a pleasure to thank Y. Pinchover for having suggested us, after we posted a first version of this work on arXiv, the very recent \cite{devyverfraaspinchover, devyverpinchover} which contain results tightly related to those in Section \ref{sec_criticality}. 

\bibliographystyle{plain}
\bibliography{bibliosignchanging_complete}

\end{document}